\documentclass[a4paper,12pt,fleqn]{smfart}

\usepackage{ucs}
\usepackage[latin1]{inputenc}

\usepackage[francais,english]{babel} \AutoSpaceBeforeFDP
\usepackage[T1]{fontenc}
\usepackage{amstext,amssymb,amsthm}
\usepackage{varioref}

\usepackage{layout}
\usepackage{smfthm}
\usepackage{graphicx}
\usepackage{subfig}
\usepackage{epsfig}
\usepackage[all]{xy}
\usepackage{enumerate}
\usepackage{version}

\usepackage{MnSymbol}
\usepackage{wasysym}

\usepackage{url}
\usepackage{geometry}
\usepackage[usenames,dvipsnames]{color}
\definecolor{darkred}{rgb}{0.9,0.,.33}
\definecolor{darkblue}{rgb}{0.,0.,.6}
\definecolor{darkgreen}{rgb}{0.,.6,0.1}
\usepackage[colorlinks=true,urlcolor=darkblue,linkcolor=darkred,citecolor=darkgreen]{hyperref}

\newcommand{\C}{\mathcal{C}}
\newcommand{\R}{\mathbb{R}}
\newcommand{\N}{\mathbb{N}}
\newcommand{\Z}{\mathbb{Z}}

\newcommand{\G}{\mathcal{G}}
\newcommand{\s}{\mathrm{SL_3(\mathbb{R})}}
\renewcommand{\ss}{\mathrm{SL_{n+1}(\mathbb{R})}}
\renewcommand{\C}{\mathcal{C}}
\newcommand{\Cc}{\mathrm{C}}

\newcommand{\U}{\mathcal{U}}
\newcommand{\F}{\mathcal{F}}

\renewcommand{\O}{\Omega}
\newcommand{\g}{\gamma}
\renewcommand{\G}{\Gamma}

\renewcommand{\S}{\mathbb{S}}
\renewcommand{\P}{\mathbb{P}^2}
\newcommand{\Pd}{(\mathbb{P}^{2})^*}
\newcommand{\PP}{\mathbb{P}}
\newcommand{\Aut}{\textrm{Aut}}
\newcommand{\Sp}{\textrm{Sp}}
\newcommand{\Ax}{\textrm{Axe}}

\newcommand{\Vol}{\textrm{Vol}}

\newcommand{\Hol}{\textrm{hol}}
\newcommand{\dev}{\textrm{dev}}

\newcommand{\Quo}{\O/_{\G}}
\newcommand{\Quc}{\C/_{\G}}

\theoremstyle{plain}
\newtheorem{fait}{Fait}

\theoremstyle{definition}

\theoremstyle{remark}
\newtheorem*{rem}{Remarque}

\title{Surface Projective Convexe de volume fini}

\author{\href{mailto:ludovic.marquis@umpa.ens-lyon.fr}{Ludovic Marquis}}
\date{} 
\urladdr{\href{http://www.umpa.ens-lyon.fr/~lmarquis/}{www.umpa.ens-lyon.fr/~lmarquis/}}

\begin{document}
\renewcommand{\labelitemi}{$\bullet$} 



\frontmatter
\maketitle

\begin{abstract}
Une surface projective convexe est le quotient d'un ouvert
proprement convexe $\Omega$ de l'espace projectif réel $\PP^2(\R)$ par un sous-groupe discret $\Gamma$ de
$\textrm{SL}_3(\R)$. Nous donnons plusieurs caractérisations du fait qu'une
surface projective convexe est de volume fini pour la mesure de
Busemann. On en d\'eduit que si $\Omega$ n'est pas un triangle alors $\Omega$ est strictement convexe, \`a bord
$\Cc^1$ et qu'une surface projective convexe $S$ est de volume fini
si et seulement si la surface duale est de volume fini.
\end{abstract}

\begin{altabstract}
A convex projective surface is the quotient of a properly convex
open $\Omega$ of the projective real space $\PP^2(\R)$ by a discret subgroup $\Gamma$ of $\textrm{SL}_3(\R)$. We give
some caracterisations of the fact that a convex projective
surface is of finite volume for the Busemann's measure. We deduce
of this that if $\Omega$ is not a triangle then $\Omega$ is strictly convex, with $\Cc^1$ boundary and
that a convex projective surface $S$ is of finite volume if and
only if the dual surface is of finite volume.
\end{altabstract}

\mainmatter

\section*{Introduction}

\subsection{Exemples de convexes divisibles}
\par{
Soit $\C$ une partie de l'espace projectif r\'eel $\PP^n=\PP^n(\R)$, on dira
que $\C$ est \emph{convexe} lorsque l'intersection de $\C$ avec toute droite de $\PP^n$ est connexe. Une partie convexe $\C$ est dite \emph{proprement convexe} lorsqu'il existe un ouvert
affine contenant l'adh\'erence $\overline{\C}$ de $\C$. Elle est
dite \emph{strictement convexe} lorsque tout segment inclus dans le bord
$\partial \C$ de $\C$ est trivial.}
\\
\par{
Le but de ce texte est d'\'etudier les ouverts proprement convexes
$\O$ de $\PP^n$ qui poss\`edent "beaucoup de sym\'etries". Un cas qui
a \'et\'e beaucoup \'etudi\'e est celui o\`u "beaucoup de sym\'etrie" signifie
qu'il existe un sous-groupe discret $\G$ de $\ss$ qui pr\'eserve
$\O$ et tel que le quotient $\Quo$ est compact. De tels ouverts
s'appellent des \emph{convexes divisibles} et on dit alors que
$\G$ \emph{divise} $\O$. Nous allons dans ce texte remplacer l'hypoth\`ese de compacit\'e
du quotient $\Quo$ par une hypoth\`ese de "finitude de volume", et nous restreindre \`a la dimension 2. Mais
commençons par donner des exemples du cas compact.
}
\\
\par{
L'exemple le plus simple de convexe divisible est le simplexe.
Toute base $\mathcal{B}$ de $\R^{n+1}$ d\'efinit un pavage de
$\PP^n$ en $2^n$ simplexes. La composante neutre du stabilisateur de chaque simplexe
ouvert est le groupe $D$ des matrices diagonales dans la base
$\mathcal{B}$ \`a coefficients positifs. $D$ est un groupe de Lie
ab\'elien isomorphe \`a $\R^n$ qui agit simplement transitivement sur
chaque simplexe ouvert $S$. Tout r\'eseau de $D$ divise donc $S$. On
vient donc de construire un convexe divisible non strictement
convexe. On remarque que dans cet exemple tout groupe qui divise
$S$ agit de façon r\'eductible sur $\R^{n+1}$, et $S$ est r\'eductible
au sens suivant.}
\\
\par{
Un ouvert proprement convexe $\O$ est \emph{r\'eductible} si l'une des deux composantes connexes $\Cc$ de $\pi^{-1}(\O)$ ($\pi$ est la projection naturelle $\pi :
\R^{n+1}-\{ 0 \} \rightarrow \PP^n$) est \emph{r\'eductible}. Ce qui
signifie qu'il existe une d\'ecomposition $\R^{n+1} = E_1 \oplus
E_2$ non triviale et des c\^ones convexes $\Cc_1$ de $E_1$ et $\Cc_2$ de
$E_2$ tel que $\Cc= \Cc_1 + \Cc_2$. Sinon, on dit qu'ils sont
\emph{irr\'eductibles}. Vey a montr\'e dans \cite{Vey} que tout convexe
divisible se d\'ecompose en un produit de convexes divisibles irr\'eductibles. On
s'intéresse donc avant tout aux convexes divisibles irr\'eductibles.}
\\
\par{
Parmi les convexes divisibles il y a une famille qui se distingue
des autres, celles des convexes divisibles homog\`enes c'est-\`a-dire ceux pour
lesquels le groupe $\Aut(\O)$ des transformations de $\ss$ qui
pr\'eserve $\O$ agit transitivement. Les travaux de Koecher,
Vinberg et Borel ont permis de classifier les convexes divisibles homog\`enes. Voici la liste des convexes divisibles irr\'eductibles homog\`enes:}
\\
\par{
Les espaces hyperboliques $\mathbb{H}^n  = \pi(\{ x \in \R^{n+1} \, | \,
x_1^2 + x_2^2 + ... + x_n^2 -x_{n+1}^2 >   0 \textrm{ et } x_{n+1}
> 0 \})$ forment la liste compl\`ete (avec $n \geqslant 1$) des
convexes divisibles strictement convexes et homog\`enes. Le groupe
d'automorphisme de $\mathbb{H}^n$ est bien entendu $SO_{n,1}(\R)$.
On remarquera qu'en toute dimension $n\geqslant 1$, il existe un unique convexe divisible strictement convexe et homog\`ene.}
\\
\par{Il existe quatre familles de convexes divisibles irr\'eductibles non
strictement convexes et homog\`enes. En voici la liste avec $n \geqslant 2$:
\begin{itemize}
\item $\Pi_n(\R)=\pi($ \{ Les matrices r\'eelles $(n+1) \times (n+1)$ sym\'etriques d\'efinies
positives) \}, il est de dimension $m=\frac{(n-1)(n+2)}{2}$ et son
groupe d'automorphisme est $\ss$.

\item $\Pi_n(\mathbb{C}) =\pi($ \{ Les matrices complexes $(n+1) \times (n+1)$
hermitiennes d\'efinies positives \}), il est de dimension $m=n^2-1$
et son groupe d'automorphisme est $\mathrm{SL}_{n+1}(\mathbb{C})$.

\item $\Pi_n(\mathbb{H}) =\pi($ \{ Les matrices quarternioniques $(n+1) \times (n+1)$
hermitiennes d\'efinies positives \}), il est de dimension
$m=(2n+1)(n-1)$ et  son groupe d'automorphisme est
$\mathrm{SL}_{n+1}(\mathbb{H})$.

\item $\Pi_3(\mathbb{O})$ un convexe "exceptionnel" de dimension
26 et tel que
Lie($\Aut(\Pi_3(\mathbb{O})))=\mathfrak{e}_{6(-26)}$.
\end{itemize}

Par cons\'equent, contrairement au cas strictement convexe, il
n'existe pas de convexe divisible irr\'eductible non strictement convexe et homog\`ene en
toute dimension.}
\\
\par{Expliquons succinctement l'histoire de cette classification. \`{A} la fin des ann\'ees 50, Koecher et Vinberg ont classifi\'e les ouverts proprement convexes sym\'etriques de $\PP^m$ (\cite{Vin}). Dans les ann\'ees 60, Borel a montr\'e dans \cite{Bor} que tout groupe r\'eductif contient un r\'eseau cocompact. On peut d\'eduire de cela que tout ouvert $\O$ proprement convexe et sym\'etrique est divisible, puisque si $\O$ est sym\'etrique alors le groupe $\Aut(\O)$ est un groupe r\'eductif qui agit transitivement et proprement sur $\O$. Le dernier pas vers la classification des convexes divisibles homog\`enes a \'et\'e fait par Vinberg (\cite{Vin2}) qui a classifi\'e les ouverts proprement convexes homog\`enes. Il r\'esulte de cette classification que tout ouvert proprement convexe homog\`ene est sym\'etrique si et seulement si le groupe $\Aut(\O)$ est unimodulaire. Par cons\'equent, tout ouvert proprement convexe homog\`ene est divisible si et seulement si il est sym\'etrique.
}
\\
\par{
Kac et Vinberg ont construit les premiers exemples de convexe
divisible strictement convexe et non homog\`ene dans \cite{KaV} \`a
l'aide de groupe de Coxeter. Johnson et Millson ont construit en
toute dimension $n \geqslant 2$ des convexes divisibles
irr\'eductibles, strictement convexes et non homog\`enes (\cite{JoMil}) en
d\'eformant des r\'eseaux cocompacts de $\textrm{SO}_{n,1}(\R)$. Kapovich et
Benoist ont construit (Benoist pour $n=4$ dans \cite{Beno1} et
Kapovich pour $n \geqslant 4$ dans \cite{Kapo}) des convexes
divisibles strictement convexes, non homog\`enes et non
quasi-isom\'etriques \`a l'espace hyperbolique $\mathbb{H}^n$ en toute
dimension $n \geqslant 4$.}

\subsection{Description des principaux r\'esultats}

\par{
 Revenons au but de ce texte. Tout ouvert proprement convexe est
 naturellement muni d'une m\'etrique Finsl\'erienne (distance de Hilbert) et de la
 mesure associ\'ee (mesure de Busemann). Le but de
 ce texte est d'\'etudier les ouverts proprement convexe de $\P$ pour
 lesquels il existe un sous-groupe discret $\G$ de $\s$ qui pr\'eserve
 $\O$ et tel que le quotient $\Quo$ muni de la mesure $\mu$ h\'erit\'ee de la mesure de Busemann soit de volume
 fini.}
 \\
\par{
Nous allons d\'emontrer les th\'eor\`emes suivants:

\begin{theo}(Corollaire \ref{Zariski})\label{0}
Soit $\G$ un sous-groupe discret de $\s$ qui pr\'eserve un ouvert
proprement convexe $\O$ de $\P$. Si $\mu(\Quo) < \infty$ et $\O$
n'est pas un triangle alors l'adh\'erence de Zariski de $\G$ est:
\begin{itemize}
\item $\s$ ou
\item Un conjugu\'e de $\mathrm{SO}_{2,1}(\R)$ et l'ouvert $\O$ est un ellipsoïde (i.e l'intérieur d'une ellipse).
\end{itemize}
\end{theo}

\begin{theo}\label{1}(Th\'eor\`eme \ref{typefini})
Toute surface admettant une structure
projective proprement convexe de volume fini est de type fini.
\end{theo}

\begin{theo}\label{2}(Corollaire \ref{caract})
Soit $S$ une surface sans bord et de type fini, une
structure projective proprement convexe sur $S$ est de volume fini si
et seulement si l'holonomie des lacets \'el\'ementaires (D\'efinition \ref{elem}) de $S$ est parabolique.
\end{theo}

On obtiendra ensuite les r\'esultats suivants:
\begin{theo}\label{3}(Th\'eor\`eme \ref{dual})
Soient $\O$ un ouvert proprement convexe et $\G$ un sous-groupe
discret qui pr\'eserve $\O$, l'action de $\G$ sur $\O$ est de
covolume fini si et seulement si l'action de $^t\G$ sur l'ouvert dual $\O^*$ est de covolume fini.
\end{theo}

\begin{theo}\label{4}(Corollaire \ref{strict} et th\'eor\`eme \ref{c1})
Soient $\O$ un ouvert proprement convexe et $\G$ un sous-groupe
discret qui pr\'eserve $\O$, on suppose que l'action de $\G$ sur
$\O$ est de covolume fini et que $\O$ n'est pas un triangle.
Alors, $\O$ est strictement convexe et le bord $\partial \O$ de
$\O$ est $C^1$.
\end{theo}

\begin{theo}\label{5}(Th\'eor\`eme \ref{enslimite})
Soient $\O$ un ouvert proprement convexe et $\G$ un
sous-groupe discret de $\s$ qui pr\'eserve $\O$, on suppose que $\G$
n'est pas virtuellement ab\'elien. Alors, l'action de $\G$ sur $\O$ est de
covolume fini si et seulement si $\G$ est de type fini et l'ensemble limite $\Lambda_{\G}$ de $\G$ v\'erifie $\Lambda_{\G}=\partial \O$.
\end{theo}
\par{
Cette \'etude permet d'obtenir sur l'espace des structures projectives marqu\'ees proprement convexes de volume fini sur la surface de genre $g$ avec $p$ pointes un syst\`eme de coordonn\'ees \`a la Fenchel-Nielsen qui montre que cet espace est hom\'eomorphe \`a $\R^{16g-16+6p}$. Ce syst\`eme de coordonn\'ees g\'en\'eralise celui  employ\'e par Goldman dans le cas compact (\cite{Gold1}). Cette \'etude sera soumise tr\`es prochainement (\cite{moi3}).
}
\\
\par{
Voici le plan de ce texte. La premi\`ere partie est une introduction \`a la g\'eom\'etrie de Hilbert. Les d\'emonstrations de certains th\'eor\`emes peuvent \^etre trouv\'es dans \cite{CVV1,CVV2,Beno3}. On donne quand m\^eme les d\'emonstrations de certains th\'eor\`emes pour faciliter la lecture de ce texte. Cette partie a pour but de d\'efinir la mesure de Busemann et de donner des exemples de parties de volume fini et infini pour celle-ci.
}
\\
\par{
Dans la seconde partie, on \'etudie la dynamique d'un \'el\'ement de $\s$ qui pr\'eserve un ouvert proprement convexe. Cette partie constitue une \'etude \'el\'ementaire mais essentiel pour nos r\'esultats.
}
\\
\par{
Dans la troisi\`eme partie on montre le th\'eor\`eme \ref{0}. Le point cl\'e \'etant de montrer l'irr\'eductibilit\'e du groupe $\G$.
}
\\
\par{
Le but de la quatri\`eme partie est de donner une courte d\'emonstration d'un th\'eor\`eme de Lee (\cite{JL}). Ce th\'eor\`eme assure l'existence d'un domaine fondamental convexe et localement fini pour l'action d'un groupe discret sur un ouvert proprement convexe. Ce r\'esultat est le point de d\'epart de l'\'etude des surfaces projectives proprement convexes de volume fini. On introduit aussi la notion de secteur qui permettra d'\'etudier les surfaces projectives proprement convexes \`a l'infini.
}
\\
\par{
L'objet de la cinqui\`eme partie est de montrer les th\'eor\`emes \ref{1} et \ref{2}. Pour cela, on utilise abondamment les parties 2 et 4. On commencera par d\'efinir pr\'ecisement les notions de surface projective, surface projective proprement convexe et surface projective proprement convexe de volume fini. Ensuite, on montrera le th\'eor\`eme \ref{1},
en utilisant une minoration uniforme de l'aire de tout triangle id\'eal. Enfin, on montre le th\'eor\`eme \ref{2}, les outils essentiels sont le th\'eor\`eme de Lee, la notion de secteur et les estimations de volume de la partie \ref{Hil}.
}
\\
\par{
La sixi\`eme partie a pour but de montrer les th\'eor\`emes \ref{3}, \ref{4}, \ref{5}. Pour montrer la strict-convexit\'e on utilise la m\^eme id\'ee que dans le cas compact mais on a besoin de raffinement. Ensuite, on d\'efinit la notion de surface duale, le th\'eor\`eme \ref{4} suit. Cette dualit\'e permettra de montrer que $\O$ est \`a bord $C^1$. Enfin, on montrera le th\'eor\`eme \ref{5}.
}
\\
\par{
Je remercie Yves Benoist pour ses nombreux conseils et nos nombreuses discussions sur ce sujet, Constantin Vernicos pour ces r\'eponses toujours tr\`es rapides \`a mes questions. Je remercie aussi Benjamin Favetto et Mathieu Cossutta, l'un pour ces conseils de r\'edaction et l'autre pour quelques discussions autour de ce sujet. Enfin, je remercie le rapporteur anonyme dont les remarques m'ont permis d'améliorer ce texte de façon significative.
}

\section{G\'eom\'etrie de Hilbert}\label{Hil}

Cette partie constitue une introduction \`a la g\'eom\'etrie de
Hilbert. Un expos\'e plus complet peut \^etre trouv\'e dans les
articles \cite{CVV1, CVV2, Beno3}.

\subsection{La m\'etrique d'un ouvert proprement convexe}\label{base}

Soit $\O$ un ouvert proprement convexe de $\PP^n$, Hilbert a
introduit sur de tels ouverts une distance, la distance de
Hilbert, d\'efinie de la façon suivante:

Soient $x \neq y \in \O$, on note $p,q$ les points d'intersection
de la droite $(xy)$ et du bord $\partial \O$ de $\O$ tels que $x$
est entre $p$ et $y$, et $y$ est entre $x$ et $q$ (voir figure \ref{dis}). On pose:

\begin{figure}[!h]
\begin{center}
\includegraphics[trim=0cm 12cm 0cm 0cm, clip=true, width=6cm]{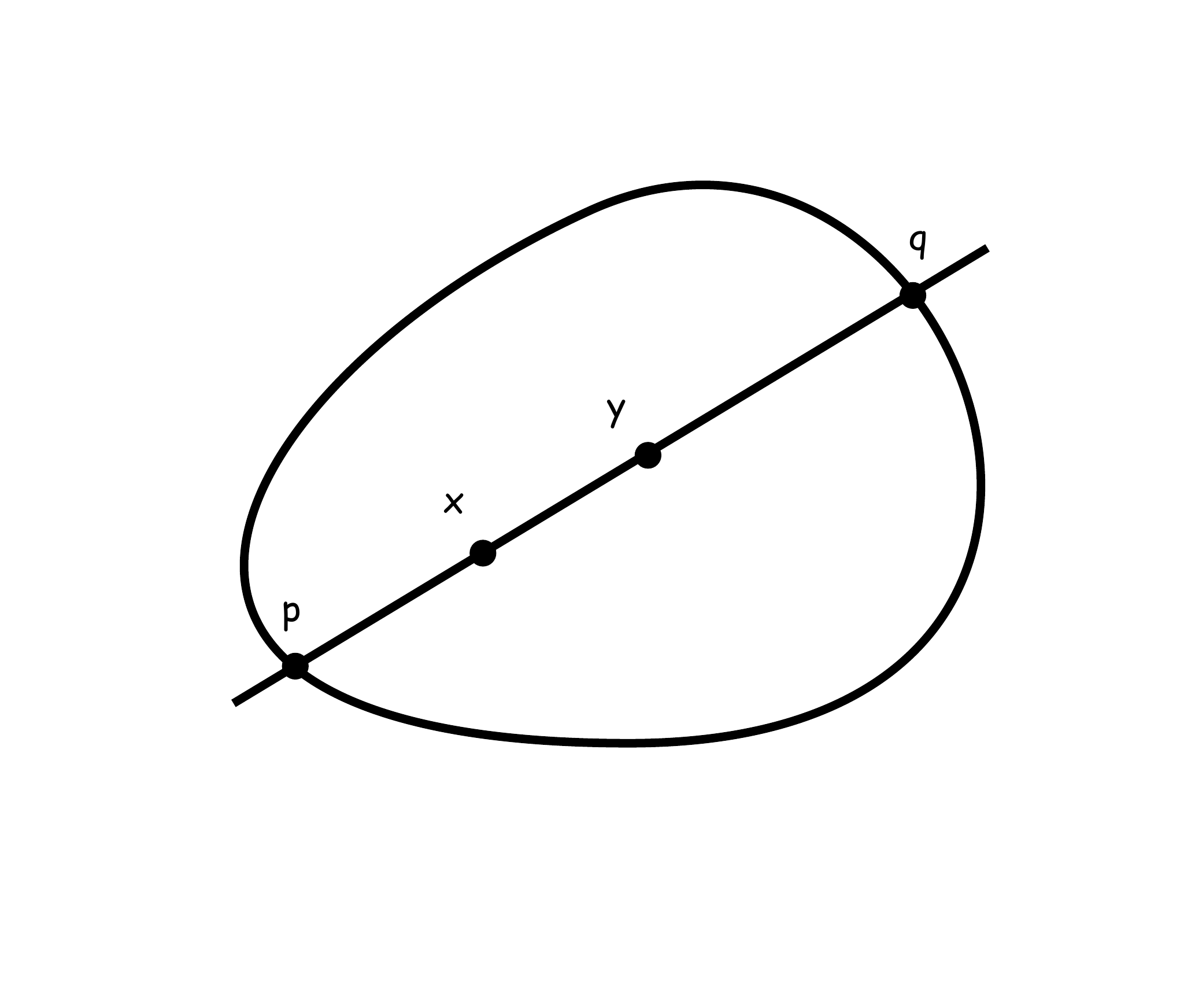}
\caption{La distance de Hilbert} \label{dis}
\end{center}
\end{figure}

$$
\begin{array}{ccc}
d_{\O}(x,y) = \ln ([p:x:y:q]) = \ln \Big(\frac{\|p-y \|\cdot \|
q-x\|}{\| p-x \| \cdot \| q-y \|} \Big) & \textrm{et} &
d_{\O}(x,x)=0
\end{array}
$$
\begin{itemize}
\item $[p:x:y:q]$ d\'esigne le birapport des points $p,x,y,q$.

\item $\| \cdot \|$ est une norme euclidienne quelconque sur un
ouvert affine $A$ qui contient l'adh\'erence $\overline{\O}$ de $\O$.
\end{itemize}

\begin{rem}
Il est clair que $d_{\O}$ ne d\'epend ni du choix de $A$, ni du
choix de la norme euclidienne sur $A$.
\end{rem}

\begin{fait}
Soit $\O$ un ouvert proprement convexe de $\PP^n$,
\begin{itemize}
\item $d_{\O}$ est une distance sur $\O$.

\item $(\O,d_{\O})$ est un espace m\'etrique complet.

\item La topologie induite par $d_{\O}$ coïncide avec celle
induite par $\PP^n$.

\item Le groupe $\Aut(\O)$ des transformations projectives de $\ss$ qui pr\'eservent $\O$ est un sous-groupe ferm\'e de $\ss$ qui agit par isom\'etrie sur $(\O,d_{\O})$. Il agit donc proprement sur $\O$.
\end{itemize}
\end{fait}

On peut trouver une d\'emonstration de cet \'enonc\'e dans \cite{Beno3}.

\subsection{La structure finsl\'erienne d'un ouvert proprement convexe}

Soit $\O$ un ouvert proprement convexe de $\PP^n$, la m\'etrique de
Hilbert $d_{\O}$ est induite par une structure finsl\'erienne sur
l'ouvert $\O$. On identifie le fibr\'e tangent $T \O$ de $\O$ \`a
$\O \times A$.

Soient $x \in \O$ et $v \in A$, on note $p^+$ (resp. $p^-$) le
point d'intersection de la demi-droite d\'efinie par $x$ et $v$
(resp $-v$) avec $\partial \O$.

On pose: $\|v\|_x = \Big(\frac{1}{\|x-p^-\|} + \frac{1}{\| x-
p^+\|} \Big) \| v \|$.

\begin{figure}[!h]
\begin{center}
\includegraphics[trim=0cm 12cm 0cm 0cm, clip=true, width=6cm]{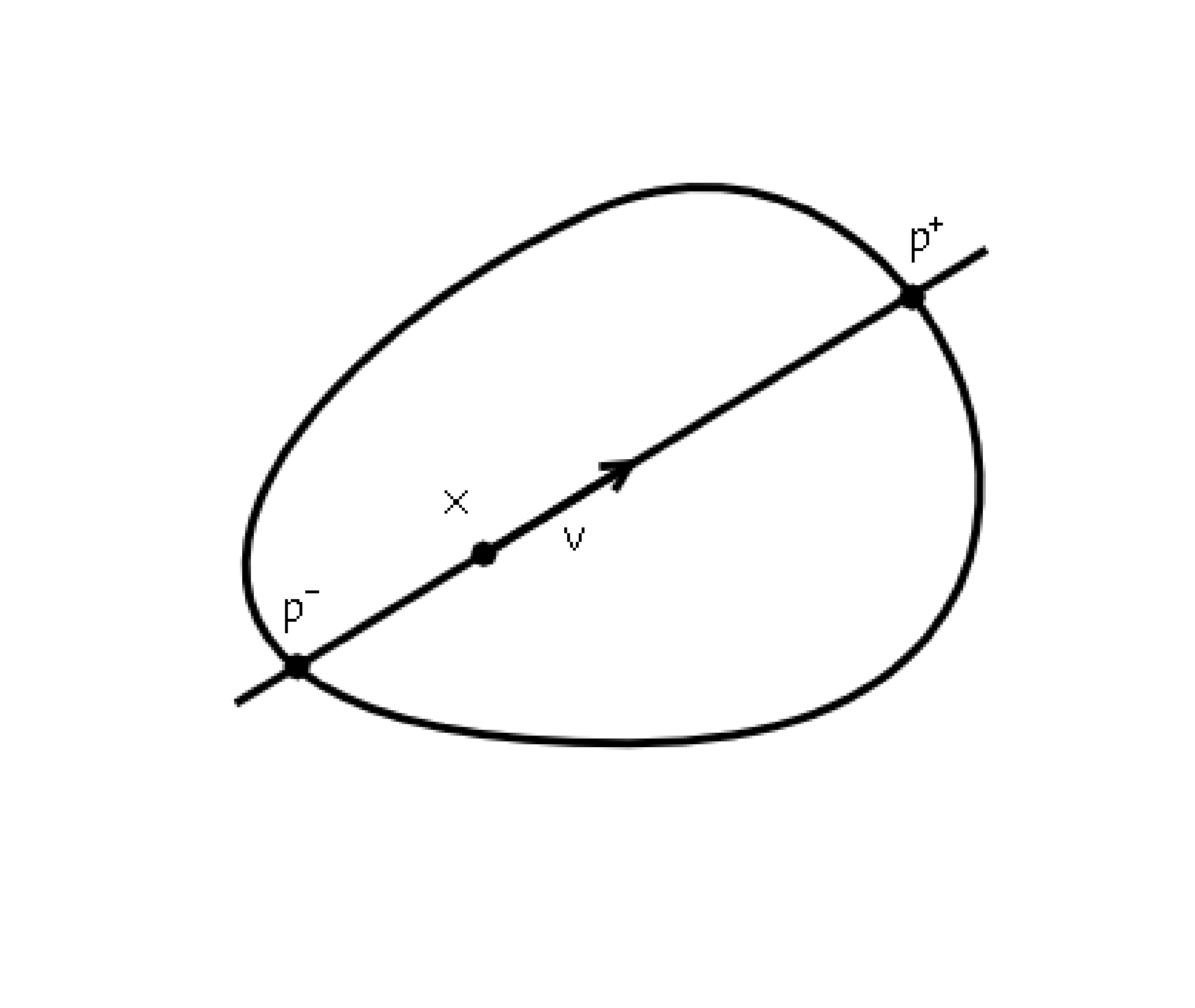}
\caption{La m\'etrique de Hilbert} \label{met}
\end{center}
\end{figure}

\begin{fait}
Soient $\O$ un ouvert proprement convexe de $\PP^n$ et $A$ un ouvert
affine qui contient $\overline{\O}$,
\begin{itemize}
\item la distance induite par la m\'etrique finsl\'erienne $\|\cdot \|
_{\cdot}$ est la distance $d_{\O}$.

\item Autrement dit on a les formules suivantes:
\begin{itemize}
\item $\|v\|_x = \frac{d}{dt}|_{t=0} d_{\O}(x,x+tv)$, o\`u $v \in
A$, $t\in \R$ assez petit.

\item $d_{\O}(x,y) = \inf \int_0^1 \| \sigma'(t) \|_{\sigma(t)}
dt$, o\`u $l'\inf$ est pris sur les chemins $\sigma$ de classe
$\C^1$ tel que $\sigma(0)=x$ et $\sigma(1)=y$.
\end{itemize}
\end{itemize}
\end{fait}

\begin{rem}
La quantit\'e $\|v\|_x$ est donc ind\'ependante du choix de $A$ et de $\| \cdot
\|_{A}$.
\end{rem}

\subsection{Mesure sur un ouvert proprement convexe (dite mesure de Busemann)}

Nous allons construire une mesure bor\'elienne $\mu_{\O}$ sur $\O$,
de la m\^eme façon que l'on construit une mesure bor\'elienne sur une
vari\'et\'e riemanienne.

Soit $\O$ un ouvert proprement convexe de $\PP^n$, on note:
\begin{itemize}
\item $B_x(1) = \{ v \in T_x \O \, | \, \|v\|_x < 1 \}$

\item $\textrm{Vol}$ est la mesure de Lebesgue sur $A$ normalis\'ee
pour avoir $\textrm{Vol}(\{ v \in A \, | \, \|v\| < 1 \})=1$.
\end{itemize}

On peut \`a pr\'esent d\'efinir la mesure $\mu_{\O}$. Pour tout bor\'elien
$\mathcal{A} \subset \O \subset A$, on pose:

$$\mu_{\O} (\mathcal{A})= \int_{\mathcal{A}} \frac{dVol(x)}{\textrm{Vol}(B_x(1))}$$

La mesure $\mu_{\O}$ est ind\'ependante du choix de $A$ et de $\| \cdot \|$,
car c'est la mesure de Hausdorff de $(\O,d_{\O})$ (Exemple 5.5.13
\cite{BBI}). (Pour une introduction aux mesures de Hausdorff, on
pourra regarder \cite{BBI}). La mesure $\mu_{\O}$ est donc
$\Aut(\O)$-invariante.

\subsection{Un r\'esultat de comparaison}

Dans la proposition suivante, il y a deux ouverts en jeu, on
ajoute donc aux notations introduites pr\'ec\'edemment le symbole de
l'ouvert auxquelles elles correspondent (ex: $\|v\|_x^{\O}$,
$p^-_{\O}$, $B^{\O}_x(1)$, etc...).

\begin{prop}\label{compa}
Soient $\O_1$ et $\O_2$ deux ouverts proprement convexes de $\PP^n$
tels que $\O_1 \subset \O_2$, alors:
\begin{itemize}
\item Les m\'etriques finsl\'eriennes de $\O_1$ et $\O_2$
v\'erifient: $\|v\|_x^{\O_2} \leqslant \|v\|_x^{\O_1}$ pour tout $x
\in \O_1$ et tout $v \in T_x \O_1 = T_x \O_2$, l'\'egalit\'e ayant
lieu si et seulement si $p^+_{\O_1}=p^+_{\O_2}$ et
$p^-_{\O_1}=p^-_{\O_2}$.

\item $\forall x,y \in \O_1$, on a $d_{\O_2}(x,y) \leqslant
d_{\O_1}(x,y)$.

\item $\forall x \in \O_1$, on a $B^{\O_1}_x(1) \subset
B^{\O_2}_x(1)$  avec \'egalit\'e si et seulement si $\O_1=\O_2$.

\item Pour tout bor\'elien $\mathcal{A}$ de $\O_1$, on a
$\mu_{\O_2}(\mathcal{A}) \leqslant \mu_{\O_1}(\mathcal{A})$.
\end{itemize}
\end{prop}

\subsection{Quelques r\'esultats en g\'eom\'etrie de Hilbert plane}

\subsubsection{Un r\'esultat sur les ouverts proprement convexes de $\P$}

On souhaite montrer la proposition suivante:

\begin{prop}\label{mubord}
Soient $\O$ un ouvert proprement convexe de $\P$ et $s$ un point du bord $\partial \O$ de
$\O$, alors, pour tout voisinage $V$ de $s$ dans $\overline{\O}$, on a $\mu_{\O}(V \cap \O) = \infty$.
\end{prop}

Soient $\O$ un ouvert proprement convexe de $\P$ et un point $x\in \O$, on notera $D_x^{\O}(\varepsilon)$ le disque de centre $x \in \O$ et rayon $\varepsilon>0$: $D_x^{\O}(\varepsilon)= \{ y \in \O \, |\, d_{\O}(x,y) < \varepsilon \}$. L'id\'ee est de construire une infinit\'e de disques disjoints de rayon constant inclus dans $V \cap \O$. La d\'emonstration se passe en deux \'etapes. On commence par montrer (lemme \ref{voldisc}) que le volume des disques de rayon $\varepsilon$ est uniform\'ement minor\'e. Ensuite, on construit une suite de disques disjoints inclus dans $V \cap \O$.

Nous allons avoir besoin du th\'eor\`eme suivant dû \`a Benz\'ecri (\cite{Ben}). On munit l'ensemble $\mathcal{E} = \{ (\O,x) \, | \, \O \textrm{ est un ouvert proprement}$ $\textrm{convexe de: }$ $ \P \textrm{ et } x \in \O \}$ de la topologie de Hausdorff pointée.

\begin{theo}[Benz\'ecri]\label{ben}
L'action de $\s$ sur l'ensemble $\mathcal{E} = \{ (\O,x) \, | \, \O $ est
un ouvert proprement convexe de $\P \textrm{ et } x
\in \O \}$ est propre et cocompacte.
\end{theo}

On obtient la proposition suivante:

\begin{lemm}\label{voldisc}
Le volume minimum d'un disque de rayon $\varepsilon$ d'un ouvert proprement convexe de $\P$ est strictement positif. Autrement dit:
$$\underset{(\O, x) \in \mathcal{E}}{\inf} \mu_{\O}(D_x^{\O}(\varepsilon)) > 0$$
\end{lemm}

\begin{proof}
La fonction qui a $(\O, x) \in \mathcal{E}$ associe $\mu_{\O}(D_x^{\O}(\varepsilon))$ est continue et $\s$-invariante. Par cons\'equent, le th\'eor\`eme \ref{ben}  montre que l'infimum de cette fonction est atteint sur $\mathcal{E}$. C'est ce qu'il fallait montrer.
\end{proof}

\`{A} pr\'esent, nous allons chercher \`a \'evaluer la taille euclidienne des disques de $\O$.

\begin{lemm}\label{barriere}
Soient $\O$ un ouvert proprement convexe de $\P$ et un point $s \in \partial \O$, on choisit une carte affine $A$ contenant $\overline{\O}$. On se donne $D_0$ une droite de $\P$ passant par $s$ tel que $\O \cap D_0 = \varnothing$. On consid\`ere un point $x \in \O$. On note $D_1$ la droite parall\`ele \`a $D_0$ (dans la carte $A$) passant par $x$. On note $D_{\infty}$ une droite parall\`ele (dans la carte $A$) \`a $D_0$ qui ne rencontre pas $\O$. Enfin, on note $D_e$ la droite parall\`ele (dans la carte $A$) \`a $D_0$ et tel que le birapport du quadruplet de droites $(D_0,D_1,D_e,D_{\infty})$ est \'egale \`a $e=\exp(1)$. Pour terminer, on note $B$ la composante connexe de $\O - D_e$ qui contient $x$. Alors, $D_x^{\O}(1)$ est inclus dans $B$.
\end{lemm}

\begin{proof}
La figure \ref{bar1} peut aider \`a suivre cette d\'emonstration.

\begin{figure}[!h]
\begin{center}
\includegraphics[trim=0cm 12cm 0cm 0cm, clip=true, width=10cm]{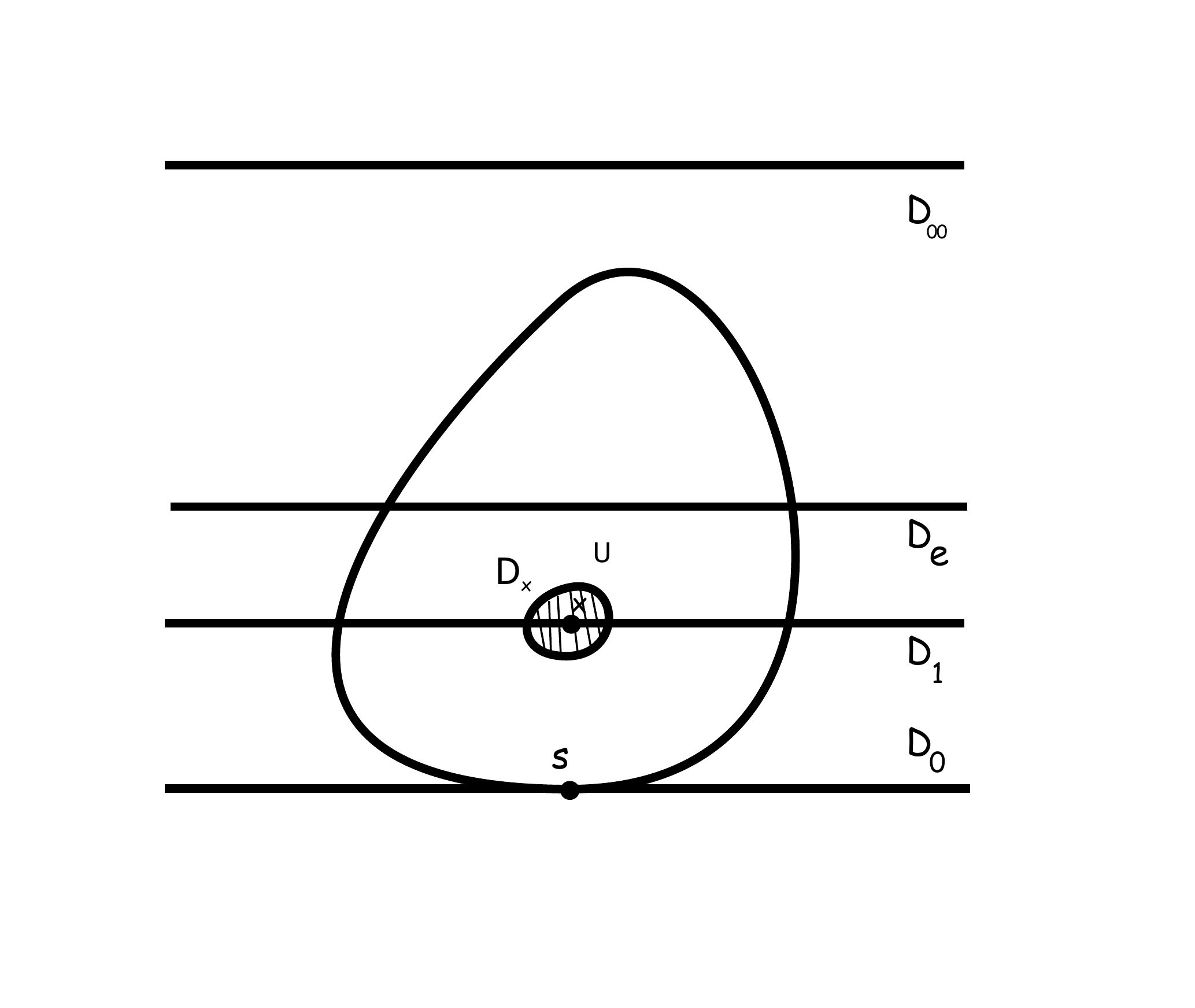}
\caption{D\'emonstration du lemme \ref{barriere}}\label{bar1}
\end{center}
\end{figure}

Pour montrer ce lemme, il faut utiliser la proposition \ref{compa}. On choisit pour $\O_2$ n'importe quel rectangle contenant $\O$, d\'elimit\'e par les droites $D_0$ et $D_{\infty}$ et tel que les c\^ot\'es donn\'es par $D_0$ et $D_{\infty}$ sont des c\^ot\'es oppos\'es de $\O_2$. La proposition \ref{compa}  montre que $D_x^{\O}(1) \subset D_x^{\O_2}(1)$. Par cons\'equent, il suffit de montrer que $D_x^{\O_2}(1)$ est inclus entre les droites $D_e$ et $D_0$. Mais $\O_2$ est un rectangle, par cons\'equent, comme la distance de Hilbert est d\'efini en terme de birapport, le disque $D_x^{\O_2}(1)$ est inclus dans un rectangle dont les c\^ot\'es sont parall\`eles \`a ceux de $\O_2$ et inclus entre les droites $D_e$ et $D_0$.
\end{proof}

On peut \`a pr\'esent montrer la proposition \ref{mubord} lorsque $s$ poss\`ede un voisinage $V$ dans $\P$ tel que $V \cap \partial \O$ ne contient aucun segment non trivial.

\begin{lemm}
Soient $\O$ un ouvert proprement convexe de $\P$ et un point $s$ du bord $\partial \O$ qui poss\`ede un voisinage $V$ dans $\P$ tel que $V \cap \partial \O$ ne contient aucun segment non trivial; alors, pour tout voisinage $V$ de $s$ dans $\overline{\O}$, on a $\mu_{\O}(V \cap \O) = \infty$.
\end{lemm}

\begin{proof}
Nous allons construire une infinit\'e de disque de rayon 1 dans $V \cap \O$. On se donne un point $x$ de $V \cap \O$ et on consid\`ere le segment $[x,s[$ inclus dans $\O$. Ce segment  fournit une g\'eod\'esique $\lambda$ de longueur infini que l'on param\`etre par la longueur d'arc pour la m\'etrique de Hilbert. Par cons\'equent, si on note $x_n = \lambda(3n)$ alors les disques $D_n = D_{x_n}^{\O}(1)$ sont disjoints. Il  reste \`a comprendre pourquoi ils sont inclus dans $V \cap \O$ pour $n$ assez grand.

On se donne $D_s$ une droite de $\P$ passant par $s$ tel que $\O \cap D_s = \varnothing$. On note $D_n$ la droite parall\`ele \`a $D_s$ passant par $\lambda(3n-1)$ et on note $B_n$ la composante connexe de $\O - D_n$ qui contient $x_n$. Si $V$ est un voisinage de $s$ dans $\P$ tel que $V \cap \partial \O$ ne contient aucun segment non trivial alors $B_n$ est inclus dans $V$ pour $n$ assez grand.

De plus, le lemme \ref{barriere}  montre que les disques de centre $x_n$ et de rayon 1 sont inclus dans $B_n$ pour $n$ assez grand. L'ensemble $V \cap \O$ contient donc une infinit\'e de disques disjoints et leur volume est uniform\'ement minor\'e par le lemme \ref{voldisc}. L'ensemble $V \cap \O$ est donc de volume infini.
\end{proof}

Il faut \`a pr\'esent traiter le cas contraire: s'il existe un voisinage de $s$ dans $\P$ qui contient un segment du bord de $\O$ alors il existe un point $s' \in  V \cap \partial \O $ tel que $s'$ appartient \`a l'int\'erieur d'un segment de $\O$. Il  reste donc \`a montrer le lemme suivant pour terminer la d\'emonstration de la proposition \ref{mubord}.

\begin{lemm}\label{nons}
Soient $\O$ un ouvert proprement convexe de $\P$ et un point $s$ du bord $\partial \O$, on suppose que $s$ est sur l'int\'erieur d'un segment $S$ du bord de $\O$; alors, pour tout voisinage $V$ de $s$ dans $\overline{\O}$, on a $\mu_{\O}(V \cap \O) = \infty$.
\end{lemm}

\begin{proof}
La figure \ref{nonstricconv} peut aider \`a suivre la d\'emonstration.

\begin{figure}[!h]
\begin{center}
\includegraphics[trim=0cm 14cm 0cm 0cm, clip=true, width=10cm]{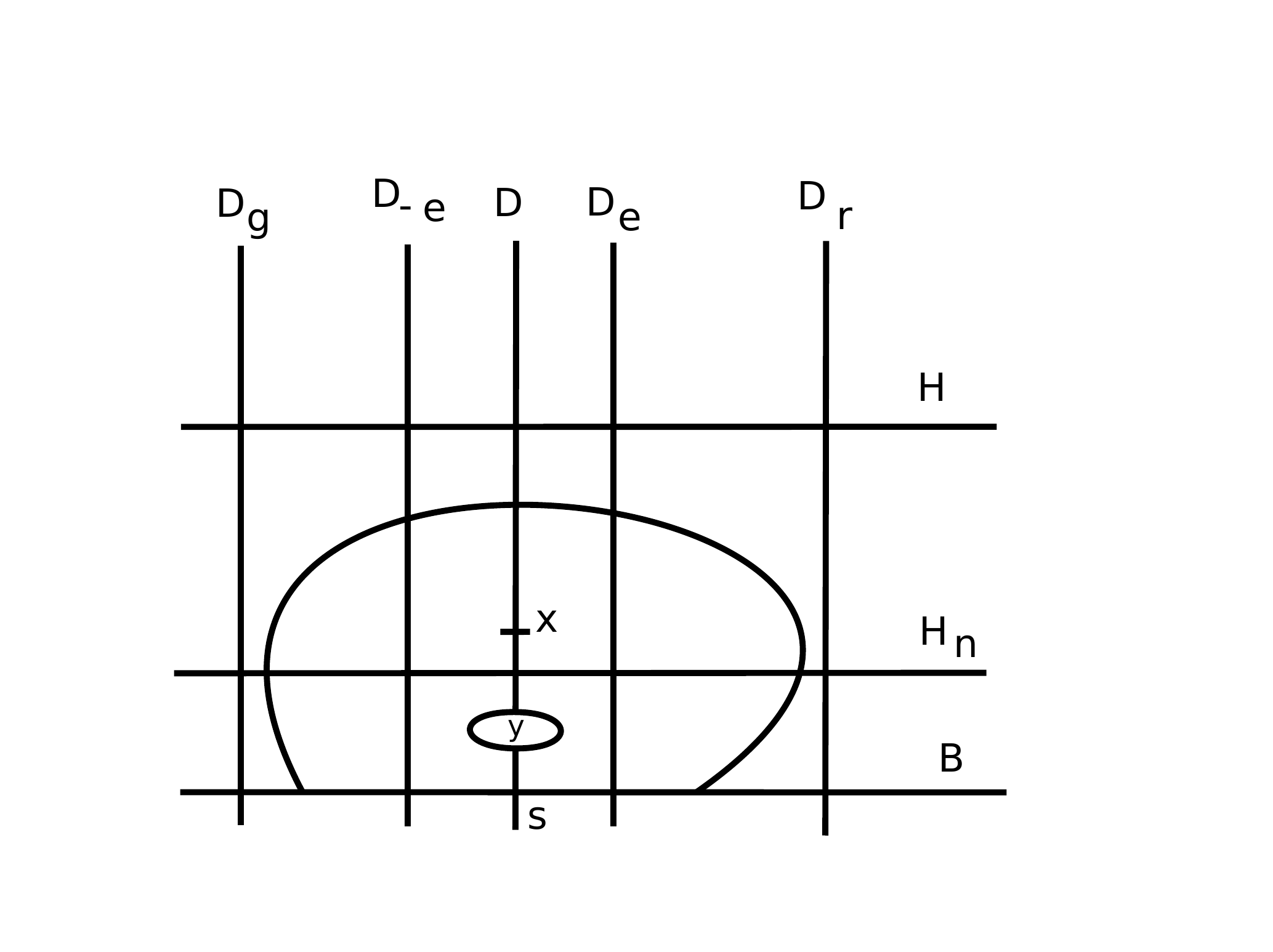}
\caption{D\'emonstration du lemme \ref{nons}}\label{nonstricconv}
\end{center}
\end{figure}

Dans ce cas, on ne peut pas trouver une infinit\'e de disques de rayon 1, mais nous allons construire une infinit\'e de disque de rayon $\varepsilon$ dans $V \cap \O$, avec $\varepsilon$ assez petit. On se donne un point $x$ de $V \cap \O$ et on consid\`ere le segment $[x,s[$ inclus dans $\O$. Ce segment  fournit une g\'eod\'esique $\lambda$ de longueur infini que l'on param\`etre par la longueur d'arc pour la m\'etrique de Hilbert. On d\'efinit la suite $x_n = \lambda(3n)$. Si $\varepsilon < 1$ alors les disques $D_n = D_{x_n}^{\O}(\varepsilon)$ sont disjoints.

Il  reste \`a comprendre pourquoi ils sont inclus dans $V \cap \O$ pour $n$ assez grand, si $\varepsilon$ est assez petit. Le lemme \ref{barriere} montre que ces disques tendent vers un sous-segment $S'$ de $S$. Comme $S'$ peut ne pas \^etre trivial il faut prendre $\varepsilon$  assez petit pour que les disques $D_n$ soient dans $V \cap \O$ pour $n$ assez grand. On va utiliser la m\^eme id\'ee que pour le lemme \ref{barriere} mais cette fois-ci, au lieu de pousser les $D_n$ vers le bord de $\O$, nous allons contrôler leur taille dans le sens "parall\`ele" au segment $S$ du bord de $\O$.

Pour cela, on note $D$ la droite passant par $x$ et $s$. On se donne une carte affine $A$ contenant $\overline{\O}$. On consid\`ere $D_g$, $D_{\varepsilon}$, $D_{-\varepsilon}$ et $D_r$ quatre droites parall\`eles (dans la carte affine $A$) \`a $D$ et tel que:
\begin{itemize}
\item Les droites $D_r$ et $D_g$ ne rencontre pas $\O$.

\item Le birapport de $(D_g,D,D_{\varepsilon},D_r)$ est \'egale \`a $e^{\varepsilon}$

\item Le birapport de $(D_g,D_{-\varepsilon},D,D_r)$ est \'egale \`a $e^{-\varepsilon}$.
\end{itemize}

Il faut aussi une droite $B$ passant par $s$ et n'intersectant pas $\O$, ainsi qu'une droite $H$ parall\`ele \`a $B$ et n'intersectant pas $\O$. L'ouvert $\O$ est donc inclus dans le quadrilat\`ere $\O_2$ d\'elimit\'e par les droites $H,D_r,B,D_g$. Comme la distance de Hilbert est d\'efini en terme de birraport, il vient que tout disque $D_y^{\O_2}(\varepsilon)$ de $\O_2$ de centre $y \in ]x,s[$ et de rayon $\varepsilon$ est inclus dans le quadrilat\`ere d\'elimit\'e par les droites $H,D_{\varepsilon},B,D_{-\varepsilon}$.

Il est essentiel de remarquer que lorsque $\varepsilon$ tend vers $0$, les droites $D_{\varepsilon}$ et $D_{-\varepsilon}$ converge vers $D$.

Pour conclure, notons $H_n$ une droite parall\`ele \`a $H$ et passant par $\lambda(3n-1)$. La proposition \ref{compa} et le lemme \ref{barriere}  montre que les disques $D_n$ de $\O$ de centre $x_n$ et de rayon $\varepsilon$ sont inclus dans le quadrilat\`ere d\'elimit\'e par les droites $H_n,D_{\varepsilon},B,D_{-\varepsilon}$. Par cons\'equent, si l'on choisit $\varepsilon$ assez petit alors les disques $D_n$ sont inclus dans $V \cap \O$, pour $n$ assez grand.

Leur volume est uniform\'ement minor\'e par le lemme \ref{voldisc}. L'ensemble $V \cap \O$ est donc de volume infini.
\end{proof}

\begin{rem}
Le théorème de Benzécri \ref{ben} est en fait valable en toute dimension. Le lemme \ref{voldisc} peut donc s'énoncer en toute dimension. L'énoncé de la proposition \ref{mubord} se généralise aussi en dimension quelconque et la méthode de démonstration peut aussi s'adapter au prix de quelques complications techniques.
\end{rem}

\subsubsection{Un r\'esultat sur les pics}

Soient $\O$ un ouvert proprement convexe de $\P$ et un point $p \in \partial \O$, l'ensemble des droites de $\P$ concourantes en $p$ et tel que $D \cap \O = \varnothing$ est un segment ferm\'e $E_p$ de $\Pd$.
On a la dichotomie suivante:
\begin{itemize}
\item L'ensemble $E_p$ est un singleton dans ce cas $\partial \O$ est $\Cc^1$
en $p$, et l'unique droite de $E_p$ est la \emph{tangente} au bord $\partial \O$ de
$\O$ en $p$.

\item Sinon, $\partial \O$ n'est pas $\Cc^1$ en $p$, et les points extr\'emaux du segment $E_p$
sont les \emph{demi-tangentes} \`a $\partial \O$ en $p$.
\end{itemize}

\begin{defi}
Soient $0<\alpha<1$, $\O$ un ouvert proprement convexe de $\P$ et
$x$ un point $\Cc^1$ du bord $\partial \O$ de $\O$, on dit que
\emph{$x$ est $\Cc^{1,\alpha}$} lorsqu'il existe $K>0$ et $V$ un
voisinage de $p$ dans $\partial \O$ tel que $\forall y \in V$,
$d_{\mathbb{E}}(y, T_x \O) \leqslant K
d_{\mathbb{E}}(x,y)^{\alpha}$, o\`u $d_{\mathbb{E}}$ est une
distance euclidienne sur un ouvert affine
contenant $\overline{\O}$.
\end{defi}

\begin{defi}
Soit $\O$ un ouvert proprement convexe de $\P$, un \emph{pic} $P$
est un triangle ouvert de $\O$ qui poss\`ede un et un seul sommet sur le
bord $\partial \O$ de $\O$, on appelle ce sommet \emph{le sommet \`a l'infini de $P$}.
\end{defi}

\begin{theo}[CVV]\label{mupointe}
Soit $\O$ un ouvert proprement convexe de $\P$.
\begin{itemize}
\item Tout pic de $\O$ dont le sommet \`a l'infini est un point
non $\Cc^1$ du bord $\partial \O$ de $\O$ est de volume infini.

\item Tout pic de $\O$ dont le sommet \`a l'infini est un point
$\Cc^{1,\alpha}$ avec $\alpha > 0$ du bord $\partial \O$ de $\O$
est de volume fini.
\end{itemize}
\end{theo}

\begin{rem}
Nous n'utiliserons pas le second point. Mais il illustre bien l'importance de la r\'egularit\'e dans l'estimation du volume des pics, qui sera essentielle dans ce texte. On pourra trouver une d\'emonstration dans \cite{CVV1} du premier et du second point. Nous donnons une d\'emonstration du premier point pour la commodit\'e du lecteur.
\end{rem}

\begin{rem}
La généralisation du théorème \ref{mupointe}  en dimension supérieure, où les pics seraient des simplexes avec un seul point sur le bord n'est pas du tout évidente. C'est encore une question ouverte.
\end{rem}

Nous allons utiliser la proposition \ref{compa} et le lemme suivant.

\begin{lemm}
Soit $\O_0$ l'ouvert proprement convexe de $\R^2$ d\'efini par $\O_0= \{ (x,y)\in \R^2 \,|\, x,y > 0 \}$, tout pic dont le sommet \`a l'infini est l'origine $(0,0)$ est de volume infini pour $\mu_{\O_0}$.
\end{lemm}

\begin{proof}
L'homoth\'etie $\g$ de rapport $\frac{1}{2}$ et de centre l'origine pr\'eserve $\O_0$. Soit $P$ un pic de $\O_0$ dont le sommet \`a l'infini est l'origine. On note $D$ la droite engendr\'e par le c\^ot\'e oppos\'e \`a l'origine de $P$, et $A$ le quadrilat\`ere ferm\'e d\'elimit\'e par les c\^ot\'es de $P$ contenant l'origine, la droite $D$ et la droite $\g D$. Ainsi, $P = \underset{n \in \N} \bigcup \g A$. Par cons\'equent, comme $\g$ pr\'eserve la mesure $\mu_{\O}$ et que $\mu_{\O}(A) > 0$. Il vient que $\mu_{\O}(P) = \infty$.
\end{proof}

\begin{proof}[D\'emonstration du premier point du th\'eor\`eme \ref{mupointe}]
Il  reste simplement \`a remarquer que si $p$ est un point non $\Cc^1$ de $\partial \O$ alors $\O$ est inclus dans un triangle dont l'un des sommets est $p$. Comme l'ouvert $\{ (x,y) \in \R^2 \,|\, x,y > 0 \}$ est un triangle, la proposition \ref{compa} et le lemme pr\'ec\'edent montre que tout pic de $\O$ dont le sommet \`a l'infini est le point $P$ est de volume infini.
\end{proof}

\subsubsection{Minoration de l'aire des triangles id\'eaux}

\begin{defi}
Soit $\O$ un ouvert proprement convexe de $\P$, un \emph{triangle
id\'eal de $\O$} est un triangle ouvert dont les sommets appartiennent au
bord de $\O$.
\end{defi}

\begin{theo}[CVV]\label{airetri}
Il existe une constante strictement positive $C_{\P}$ tel que pour tout ouvert $\O$ proprement convexe de $\P$ et tout triangle id\'eal $\Delta$ de $\O$ on ait $\mu_{\O}(\Delta) \geqslant C_{\P} >0$.
\end{theo}

\begin{proof}
On note $(p_i)_{i=1...3}$ les sommets de $\Delta$. On consid\`ere $(D_i)_{i=1...3}$ trois droites de $\P$ tel que $D_i \cap \O = \varnothing$ et $p_i \in D_i$. Les droites $(D_i)_{i=1...3}$ d\'efinissent quatre triangles ouverts de $\P$. Un seul d'entre eux contient l'ouvert $\O$, on le note T. La proposition \ref{compa}  montre que $\mu_{\O}(\Delta) \geqslant \mu_{\mathrm{T}}(\Delta)$. Le lemme \ref{tritri} qui suit conclut la d\'emonstration.
\end{proof}

\begin{lemm}\label{tritri}
Soit T un triangle de $\P$, il existe une constante strictement positive $C_{\P}$ tel que pour tout triangle id\'eal $\Delta$ de T on ait $\mu_{\mathrm{T}}(\Delta) \geqslant C_{\P}$.
\end{lemm}

\begin{proof}
La figure \ref{mintri} peut aider \`a suivre la d\'emonstration.

\begin{figure}[!h]
\begin{center}
\includegraphics[trim=0cm 12cm 0cm 0cm, clip=true, width=8cm]{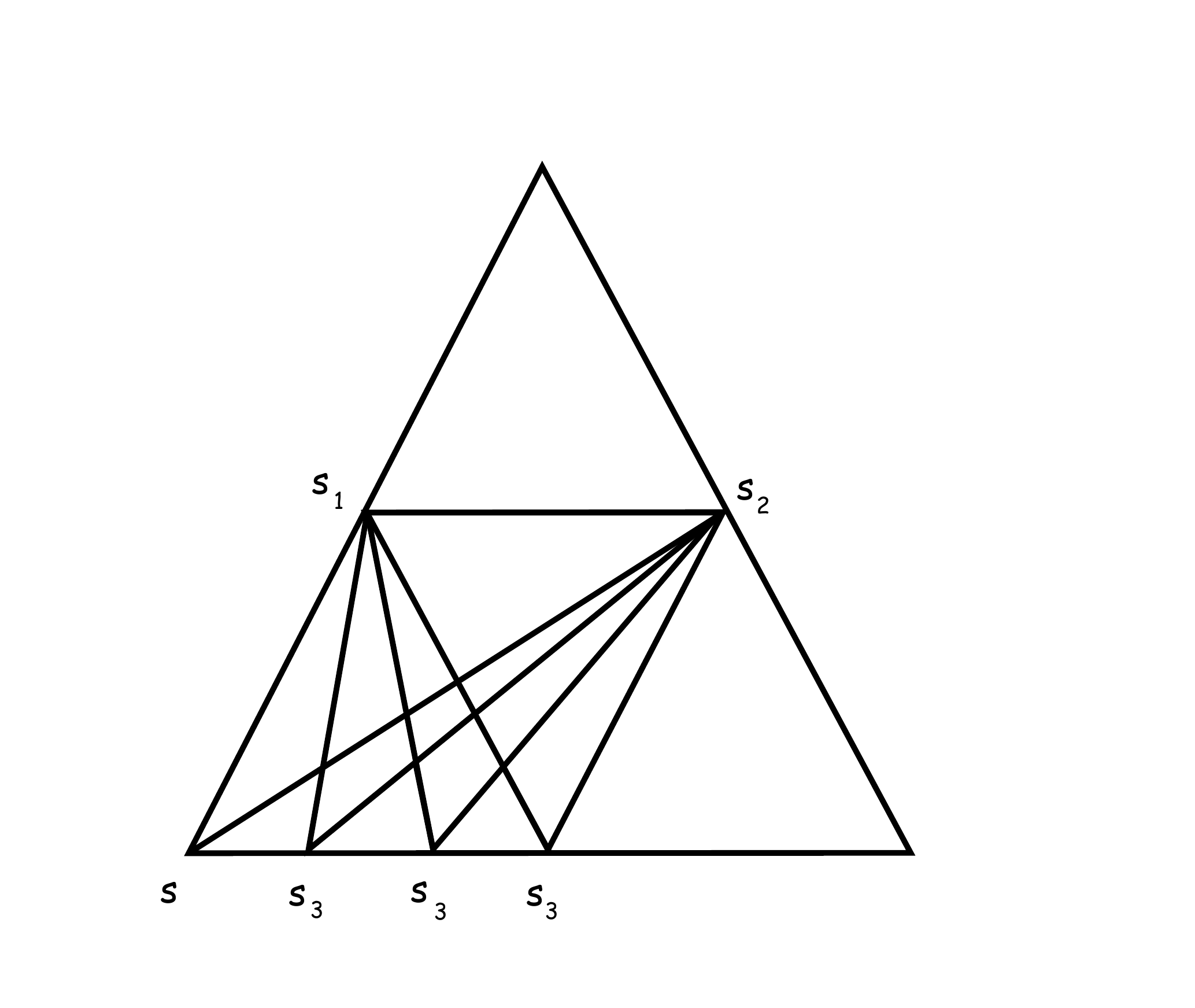}
\caption{D\'emonstration du lemme \ref{tritri}}\label{mintri}
\end{center}
\end{figure}

On peut supposer que T$=\{ [x:y:z] \in \P \,|\, x,y,z \textrm{ sont de m\^eme signe}\}$. Le groupe $D$ des matrices diagonales \`a diagonale strictement positive pr\'eserve T. Commençons par remarquer que la proposition \ref{mubord}  montre que tout triangle id\'eal qui poss\`ede deux sommets sur le m\^eme cot\'e de $T$ est de volume infini. De m\^eme, tout triangle id\'eal qui poss\`ede un sommet en commun avec les sommets de T est aussi de volume infini d'apr\`es le th\'eor\`eme \ref{mupointe}. On peut donc supposer que les sommets de $\Delta$ sont sur chacun des trois c\^ot\'es ouverts de T.

Le groupe $D$ agit transitivement sur chaque c\^ot\'e ouvert de T. Mieux, le stabilisateur d'un point de l'int\'erieur d'un c\^ot\'e ouvert de T agit encore transitivement sur chacun des deux autres c\^ot\'es ouverts. On peut donc supposer que les sommets $s_1$, $s_2$ et $s_3$ du triangle id\'eal $\Delta$ ont pour coordonn\'ees $s_1=[1:1:0]$, $s_2=[0:1:1]$ et enfin $s_3=[x:0:1]$, o\`u $x>0$. On note $\Delta_x$ ce triangle.

Comme $\mu_T(\Delta_x)$ d\'epend de façon continu de $x$, il  suffit de montrer que le volume de $\Delta_x$ admet une limite infinie lorsque $x$ tend vers $0$, pour conclure notre d\'emonstration. Le point $s_3$ tend vers le point $s=[0:0:1]$ lorsque $x$ tend vers 0. On note $\Delta'$ le triangle inclus dans $\Delta$ et de sommet $s,s_1,s_2$, et $\Delta'_x$ le triangle intersection $\Delta_x' =\Delta_x \cap \Delta'$. La famille des triangles $\Delta'_x$ croit lorsque $x$ d\'ecroit vers 0 et elle converge vers le triangle $\Delta'$. Le th\'eor\`eme \ref{mubord}  montre que le volume de $\Delta'$ est infini et le lemme de Fatou  montre que le volume de $\Delta'_x$ tend vers l'infini lorsque $x$ tend vers 0. Mais les $\Delta'_x$ sont inclus dans $\Delta_x$. Ce qui conclut la démonstration.
\end{proof}

\begin{rem}
On peut montrer que l'aire minimale pour un triangle id\'eal du triangle T est atteinte pour le triangle $\Delta_x$ avec $x=1$, et son aire est (avec les normalisations que l'on a choisi) $\frac{\pi^3}{24}$, voir \cite{CVV1}.
\end{rem}

\begin{rem}
La généralisation du théorème \ref{airetri} en dimension supérieure est fausse. En effet, on peut construire sans peine une suite de tétraèdres idéaux (non dégénérés) de l'espace hyperbolique dont le volume tend vers 0.
\end{rem}

\section{Dynamique}\label{dyna}

Dans cette partie nous allons \'etudier les diff\'erentes dynamiques
possibles pour un \'el\'ement $\g$ de $\s$ qui pr\'eserve un ouvert
proprement convexe.

\begin{rem}
On munit l'ensemble des ferm\'es de $\P$ de la topologie de
Hausdorff h\'erit\'ee de la topologie de $\P$. Ainsi, toutes les
convergences de suites de ferm\'es utilis\'ees dans le reste du texte
sont au sens de la topologie de Hausdorff.
\end{rem}

\subsection{Le cas de $\PP^1$ et $\mathrm{PSL}_2(\R)$}

On rappelle sans d\'emonstration la proposition suivante:

\begin{prop}
Soit $\g \in \mathrm{PSL}_2(\R)$, l'\'el\'ement $\g$ fait partie de l'une des quatre familles suivantes:

\begin{itemize}
\item La famille des \'el\'ements dits $\underline{hyperboliques}$ qui sont conjugu\'es \`a une matrice de la forme suivante:

$$
\begin{array}{cc}
\left(\begin{array}{cc}
\lambda &   0        \\
0          & \mu \\
\end{array}
\right) &
\begin{tabular}{l}
o\`u $\lambda \mu = 1$ et $\lambda > \mu > 0$\\
\end{tabular}
\end{array}
$$

\item La famille des \'el\'ements dits $\underline{paraboliques}$ qui sont conjugu\'es \`a la matrice:

$$
\begin{array}{cc}
\left(\begin{array}{cc}
1 & 1 \\
0 & 1 \\
\end{array}
\right) &
\end{array}
$$

\item La famille des \'el\'ements dits $\underline{elliptiques}$ qui sont conjugu\'es \`a une matrice de la forme suivante:

$$
\begin{array}{cc}
\left(\begin{array}{cc}
\cos(\theta) & -\sin(\theta) \\
\sin(\theta) &  \cos(\theta) \\
\end{array}
\right) &
\begin{tabular}{l}
o\`u, $0< \theta <2\pi$\\
\end{tabular}
\end{array}
$$

\item La famille compos\'ee uniquement de l'identit\'e.
\end{itemize}
\end{prop}

\begin{defi}
Soient $\g \in \s$ et $D$ une droite de $\P$ stable par $\g$, on
dira que \emph{l'action de $\g$ sur $D$ est de type hyperbolique (resp. parabolique, resp. elliptique)} lorsque $\g$ restreint \`a cette
droite est un \'el\'ement hyperbolique (resp. parabolique, resp.
elliptique).
\end{defi}

\subsection{Classification}

La proposition suivante est d\'emontr\'ee dans \cite{Choi}. On
reproduit ici une d\'emonstration pour la commodit\'e du lecteur.

\begin{prop}\label{classi}
Soit $\O$ un ouvert proprement convexe et un \'el\'ement $\g \in \Aut(\O)$,
$\g$ fait partie de l'une des six familles
suivantes:

\begin{itemize}
\item La famille des \'el\'ements dits $\underline{hyperboliques}$ qui sont conjugu\'es \`a une matrice de la forme suivante:

$$
\begin{array}{cc}
\left(\begin{array}{ccc}
\lambda^+ &    0      & 0  \\
 0        & \lambda^0 & 0  \\
 0        &    0      & \lambda^-\\
\end{array}
\right)
&
\begin{tabular}{l}
o\`u, $\lambda^+ > \lambda^0 > \lambda^- > 0$\\
et $\lambda^+ \lambda^0 \lambda^- = 1$.
\end{tabular}
\end{array}
$$

\item La famille des \'el\'ements dits $\underline{planaires}$ qui sont conjugu\'es \`a une matrice de la forme suivante:

$$
\begin{array}{cc}
\left(\begin{array}{ccc}
\alpha &   0        & 0   \\
 0     & \alpha     & 0    \\
 0     &   0        & \beta\\
\end{array}
\right) &
\begin{tabular}{l}
o\`u, $\alpha, \beta > 0$, $\alpha^2 \beta = 1$\\
et $\alpha,\, \beta \neq 1$.
\end{tabular}
\end{array}
$$

\item La famille des \'el\'ements dits
$\underline{quasi-hyperboliques}$ qui sont conjugu\'es \`a une matrice de la forme suivante:

$$
\begin{array}{cc}
\left(\begin{array}{ccc}
\alpha &  1     &  0    \\
 0     & \alpha &  0   \\
 0     &  0     & \beta\\
\end{array}
\right)
 &
\begin{tabular}{l}
o\`u, $\alpha, \, \beta > 0$, $\alpha^2 \beta = 1$\\
et $\alpha, \,\beta \neq 1$.
\end{tabular}
\end{array}
$$

\item La famille des el\'ements dits $\underline{paraboliques}$ qui sont conjugu\'es \`a la matrice suivante:

$$
\left(\begin{array}{ccc}
1 & 1 & 0\\
0 & 1 & 1\\
0 & 0 & 1\\
\end{array}
\right)
$$

\item La famille des el\'ements dits \underline{elliptiques} qui sont conjugu\'es \`a une matrice de la forme suivante:

$$
\begin{tabular}{cc}
$\left(\begin{array}{ccc}
1      &   0          & 0\\
0      & \cos(\theta) & -\sin(\theta)\\
0      & \sin(\theta) & \cos(\theta)\\
\end{array}
\right)$ &

O\`u, $0 < \theta < 2\pi$.

\end{tabular}
$$

\item La famille compos\'ee uniquemenent de \underline{l'identit\'e}:
\end{itemize}
\end{prop}

\begin{proof}
Soit $\g \in \s$ on notera $\Sp(\g)$ le spectre de $\g$, pour montrer ce lemme il suffit de montrer 3 points:
\begin{itemize}
\item Si $\Sp(\g) \subset \R$ et $\Sp(\g) \not\subset \{1, \, -1\}$
alors on a $\Sp(\g) \subset \R_+^*$.

\item Si $\Sp(\g) \subset \{ 1, \, -1\}$ alors $\g$ est
diagonalisable ou parabolique.

\item Si $\Sp(\g) \not\subset \R$ alors l'unique valeur propre
r\'eelle de $\g$ est 1.
\end{itemize}

Commençons par montrer le premier point. On consid\`ere
une valeur propre $\lambda$ de $\g$ de valeur absolue maximale, comme $\Sp(\g)
\not\subset \{1, \, -1\}$ on a $|\lambda| > 1$. On note $E$ la
r\'eunion des points et droites stables de $\g$. L'ensemble $\P-E$
est un ouvert dense, par cons\'equent il existe un point $x \in
(\P-E) \cap \O$. Si $\lambda < 0$,  le segment
$[\g^{2n}(x),\g^{2n+1}(x)]$ converge vers une droite de $\P$, ce
qui est absurde car $\O$ est proprement convexe, donc $\lambda
> 0$. En appliquant ce raisonnement \`a $\g^{-1}$ on obtient que les
valeurs propres de valeur absolu maximale et minimale sont positives
et par cons\'equent, comme $\g \in \s$, on a $\Sp(\g) \subset \R_+^*$.

Montrons \`a pr\'esent le deuxi\`eme point. Il s'agit de montrer que
$\g$ ne peut pas \^etre conjugu\'e \`a l'une des 2 matrices suivantes:
$$
\left(\begin{array}{ccc}
u & 1 &  \\
  & u &  \\
  &   & 1\\
\end{array}
\right)
$$
o\`u, $u=1$ ou $u=-1$. Supposons qu'il existe un tel $\g \in
\Aut(\O)$. Quitte \`a travailler avec $\g^2$ on peut supposer que
$u=1$. Alors, il existe une unique droite $D$ compos\'ee de points
fixes pour $\g$. Et, $\g$ poss\`ede un unique point fixe $v \in D$
tel que l'action de $\g$ sur toute droite $D'$ passant par $v$ est
parabolique si $D' \neq D$. L'\'el\'ement $\g$ ne pr\'eserve donc aucun convexe
proprement convexe de $\P$.

Enfin, montrons le dernier point. Si le spectre de $\g$ n'est pas
r\'eel alors $\g$ est conjugu\'e \`a une matrice de la forme
suivante:

$$
\begin{tabular}{ccc}
$\left(\begin{array}{ccc}
r^{-2} &               &\\
       & r\cos(\theta) & -r\sin(\theta)\\
       & r\sin(\theta) & r\cos(\theta)\\
\end{array}
\right)$ &
o\`u, $0 < \theta < 2\pi$ et $\theta \neq \pi$.
\end{tabular}
$$

L'\'el\'ement $\g$ poss\`ede un unique point fixe $v$ et $\P- \{ v \}$ est une
r\'eunion d'ellipses disjointes permut\'ees par $\g$. Par cons\'equent $\g$
pr\'eserve un convexe proprement convexe si et seulement si les
ellipses sont globalement pr\'eserv\'ees par $\g$, c'est \`a dire si et
seulement si l'unique valeur propre r\'eelle de $\g$ est 1.
\end{proof}

\subsection{R\'esultat \'el\'ementaire sur la dynamique}

Les r\'esultats suivants ont d\'ej\`a \'et\'e pr\'esent\'e dans \cite{Gold1} ou
\cite{Choi}.

\subsubsection{Dynamique hyperbolique}

Soit $\g$ un \'el\'ement hyperbolique de $\s$, l'\'el\'ement $\g$ est conjugu\'e \`a la
matrice

$$
\begin{array}{cc}
\left(\begin{array}{ccc}
\lambda^+ &   0       & 0  \\
 0        & \lambda^0 & 0  \\
 0        &   0       & \lambda^-\\
\end{array}
\right) &
\begin{tabular}{l}
O\`u, $\lambda^+ > \lambda^0 > \lambda^- > 0$\\
et $\lambda^+ \lambda^0 \lambda^- = 1$.
\end{tabular}
\end{array}
$$

On note:
\begin{itemize}
\item $p^+_{\g}$ le point propre de $\P$ associ\'e \`a la valeur
propre $\lambda^+$.

\item $p^0_{\g}$ le point propre de $\P$ associ\'e \`a la valeur
propre $\lambda^0$.

\item $p^-_{\g}$ le point propre de $\P$ associ\'e \`a la valeur
propre $\lambda^-$.

\item $D^{+,-}_{\g}$ la droite stable de $\P$ associ\'ee aux valeurs
propres $\lambda^+,\lambda^-$.

\item $D^{+,0}_{\g}$ la droite stable de $\P$ associ\'ee aux valeurs
propres $\lambda^+,\lambda^0$.

\item $D^{-,0}_{\g}$ la droite stable de $\P$ associ\'ee aux valeurs
propres $\lambda^-,\lambda^0$.
\end{itemize}


\begin{prop}
Soient $\O$ un ouvert proprement convexe et un \'el\'ement hyperbolique $\g \in \Aut(\O)$, alors, on a $p^+_{\g} ,\, p^-_{\g} \in \partial \O$ et $p^0_{\g} \notin \O$.
\end{prop}

\begin{proof}
Tout d'abord les points $p^+_{\g}, \, p^0_{\g}, \, p^-_{\g} \notin
\O$ puisque $\Aut(\O)$ agit proprement sur $\O$. Ensuite pour tout $x \in \O, \, \underset{n \rightarrow + \infty}{\lim} \g^n x =
p^+_{\g}$ et $\underset{n \rightarrow - \infty}{\lim} \g^n x = p^-_{\g}$
donc $p^+_{\g}, \, p^-_{\g} \in
\partial \O$.
\end{proof}

On peut \`a pr\'esent d\'efinir l'axe d'un \'el\'ement hyperbolique qui agit sur un ouvert proprement convexe.

\begin{defi}
Soient $\O$ un ouvert proprement convexe et un \'el\'ement hyperbolique $\g \in \Aut(\O)$, \emph{l'axe de $\g$} que l'on notera $\Ax(\g)$ est le segment ouvert de la droite $(p^-_{\g} p^+_{\g})$ qui est inclus dans $\overline{\O}$ et dont les extr\'emit\'es sont $p^-_{\g}$ et $p^+_{\g}$.
\end{defi}

\begin{defi}\label{axesecond}
Soient $\O$ un ouvert proprement convexe de $\P$ et un \'el\'ement hyperbolique $\g$ de $\Aut(\O)$. On suppose que $p^0_{\g} \in \partial \O$. Alors, on note $[p^+_{\g},p^0_{\g}]$ (resp. $[p^-_{\g},p^0_{\g}]$) le segment ouvert de la droite $(p^+_{\g} p^0_{\g})$ (resp.$(p^-_{\g} p^0_{\g})$) inclus dans $\overline{\O}$. On les appelle les \emph{axes secondaires} de $\g$.
\end{defi}

\begin{defi}
Soit $c: \mathbb{S}^1 \rightarrow \P$ une courbe simple continue, on dit
que $c$ est une \emph{courbe convexe} lorsque la composante
connexe orientable de $\P - c(\mathbb{S}^1)$ est un ouvert
proprement convexe.
\end{defi}

Les figures \ref{ffhyp} et \ref{dessin_hyp} illustrent la dynamique d'un \'el\'ement hyperbolique.

\begin{prop}\label{hyp}
Soient $\O$ un ouvert proprement convexe et un \'el\'ement hyperbolique $\g \in \Aut(\O)$.
\begin{itemize}
\item[(A)] Si $p^0_{\g} \notin \overline{\O}$ et $\Ax(\g) \subset \O$
alors
      \begin{itemize}
      \item $\partial \O$ est $\Cc^1$ en $p^+_{\g}$ et en $p^-_{\g}$.

      \item $p^0_{\g} = T_{p^+_{\g}} \partial \O \cap T_{p^-_{\g}} \partial \O$.
      \end{itemize}

\item[(B)]  Si $p^0_{\g} \notin \overline{\O}$ et $\Ax(\g) \subset
\partial \O$ alors
      \begin{itemize}
      \item $\partial \O$ n'est pas $\Cc^1$ en $p^+_{\g}$ et en $p^-_{\g}$.
      \item Les demi-tangentes \`a $\partial \O$ en $p^+_{\g}$ sont $(p^+_{\g} p^-_{\g})$ et $(p^+_{\g} p^0_{\g})$.
      \item Les demi-tangentes \`a $\partial \O$ en $p^-_{\g}$ sont $(p^+_{\g} p^-_{\g})$ et $(p^-_{\g} p^0_{\g})$.
      \end{itemize}

\item[(C)]  Si $p^0_{\g} \in \overline{\O}$ et $\Ax(\g) \subset \O$ alors
      \begin{itemize}
      \item $[p^+_{\g},p^0_{\g}]$ et $[p^-_{\g},p^0_{\g}] \subset \partial \O$.
      \item $\partial \O$ est $\Cc^1$ en $p^+_{\g}$, $p^-_{\g}$ et $p^0_{\g} = T_{p^+_{\g}} \partial \O \cap T_{p^-_{\g}} \partial
      \O$.
      \end{itemize}

\item[(D)]  Si $p^0_{\g} \in \overline{\O}$ et $\Ax(\g) \subset \partial
\O$ alors
      \begin{itemize}
      \item $[p^+_{\g},p^0_{\g}]$ et $[p^-_{\g},p^0_{\g}] \subset \partial \O$.
      \item $\O$ est un
triangle dont les sommets sont $p^+_{\g}, \, p^0_{\g}, \, p^-_{\g}$.
      \end{itemize}
\end{itemize}
\end{prop}

\begin{figure}[h!]
\begin{center}
\includegraphics[trim=0cm 7cm 0cm 0cm, clip=true, width=8cm]{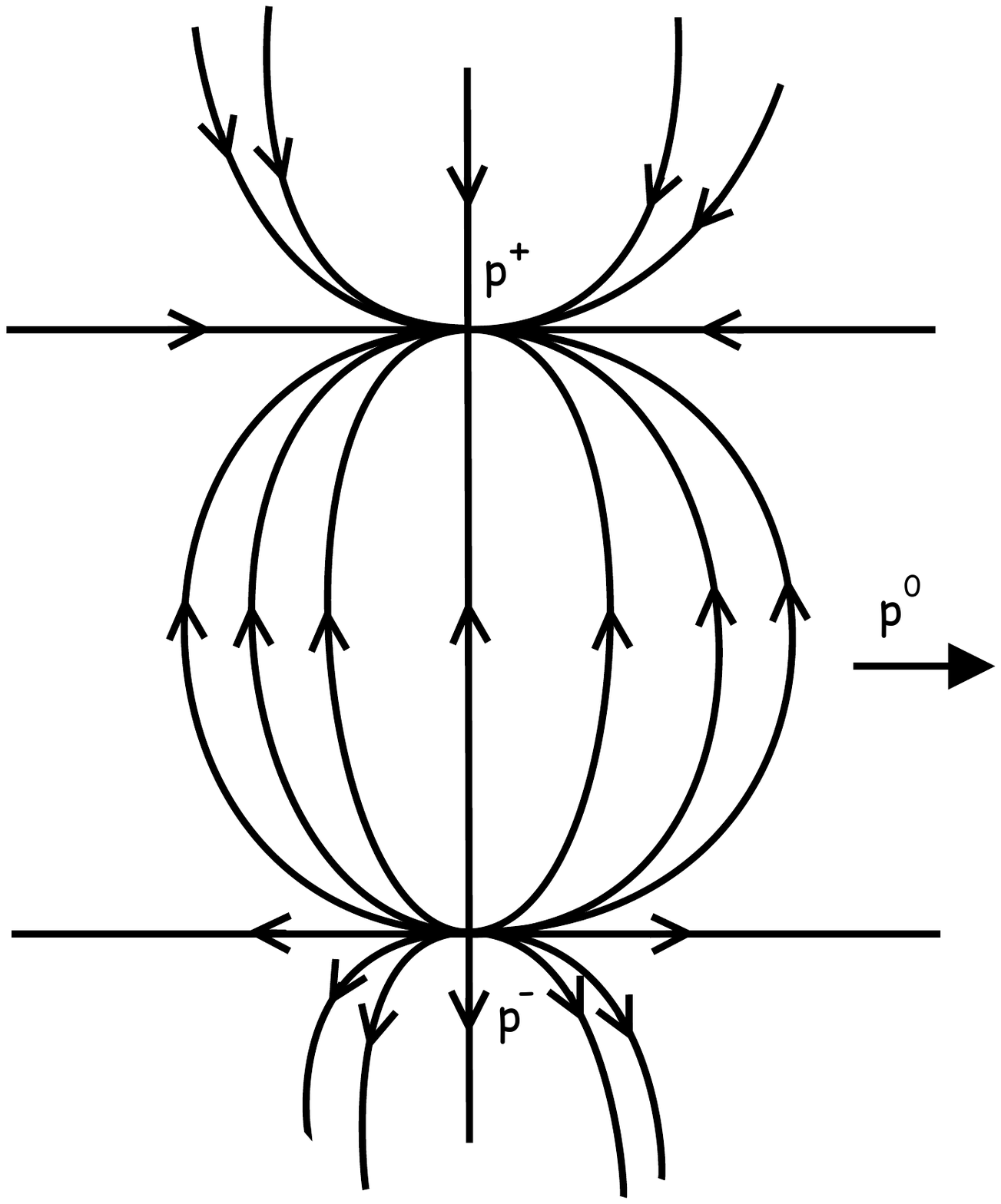}
\caption{Dynamique d'un \'el\'ement hyperbolique sur $\P$}\label{ffhyp}
\end{center}
\end{figure}

\begin{proof}
Les points $p^+_{\g}$, $p^0_{\g}$ et $p^-_{\g}$ d\'efinissent un pavage de
$\P$ en quatre triangles ferm\'es. Soient $T_1$ et $T_2$ les deux triangles de cette partition tels
que $T_1 \cap T_2  = \Ax(\g)$, et $x_1 \in \overset{\circ}{T_1}$ et $x_2 \in \overset{\circ}{T_2}$, la courbe
obtenue en concat\'enant la courbe $(\g^t x_1)_{t \in \R}$ et la courbe $(\g^{-s} x_2)_{s \in \R}$ et en ajoutant les points limites $p^+_{\g}$ et $p^-_{\g}$ d\'efinit une courbe $\C$ convexe analytique
en dehors des points $p^+_{\g}$ et $p^-_{\g}$. $\C$ est $\Cc^{1,\alpha^+}$
en $p^+_{\g}$ et $\Cc^{1,\alpha^-}$ en $p^-_{\g}$, o\`u $\alpha^+ =
\frac{\ln(\lambda^+) -\ln(\lambda^0)}{\ln(\lambda^0)
-\ln(\lambda^-)} > 0$ et $\alpha^- = (\alpha^+)^{-1}$. On tire facilement de tout ceci les conclusions de la proposition.
\end{proof}

\begin{center}
\begin{figure}[h!]
  \centering
  \subfloat[][]{\label{pliage1}\includegraphics[trim=0cm 10cm 5cm 0cm, clip=true, width=5cm]{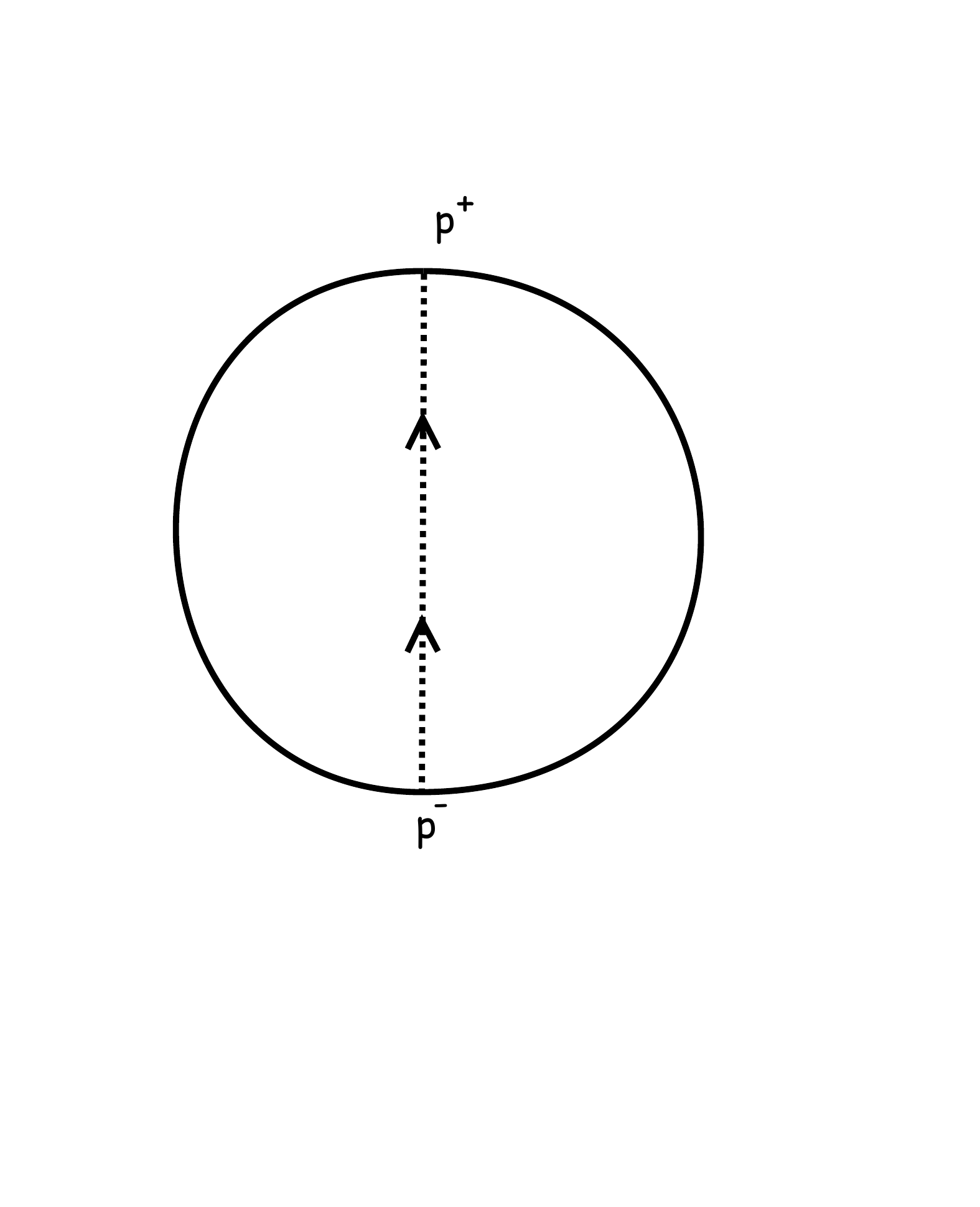}}                
  \subfloat[][]{\label{pliage2} \includegraphics[trim=0cm 10cm 5cm 0cm, clip=true, width=5cm]{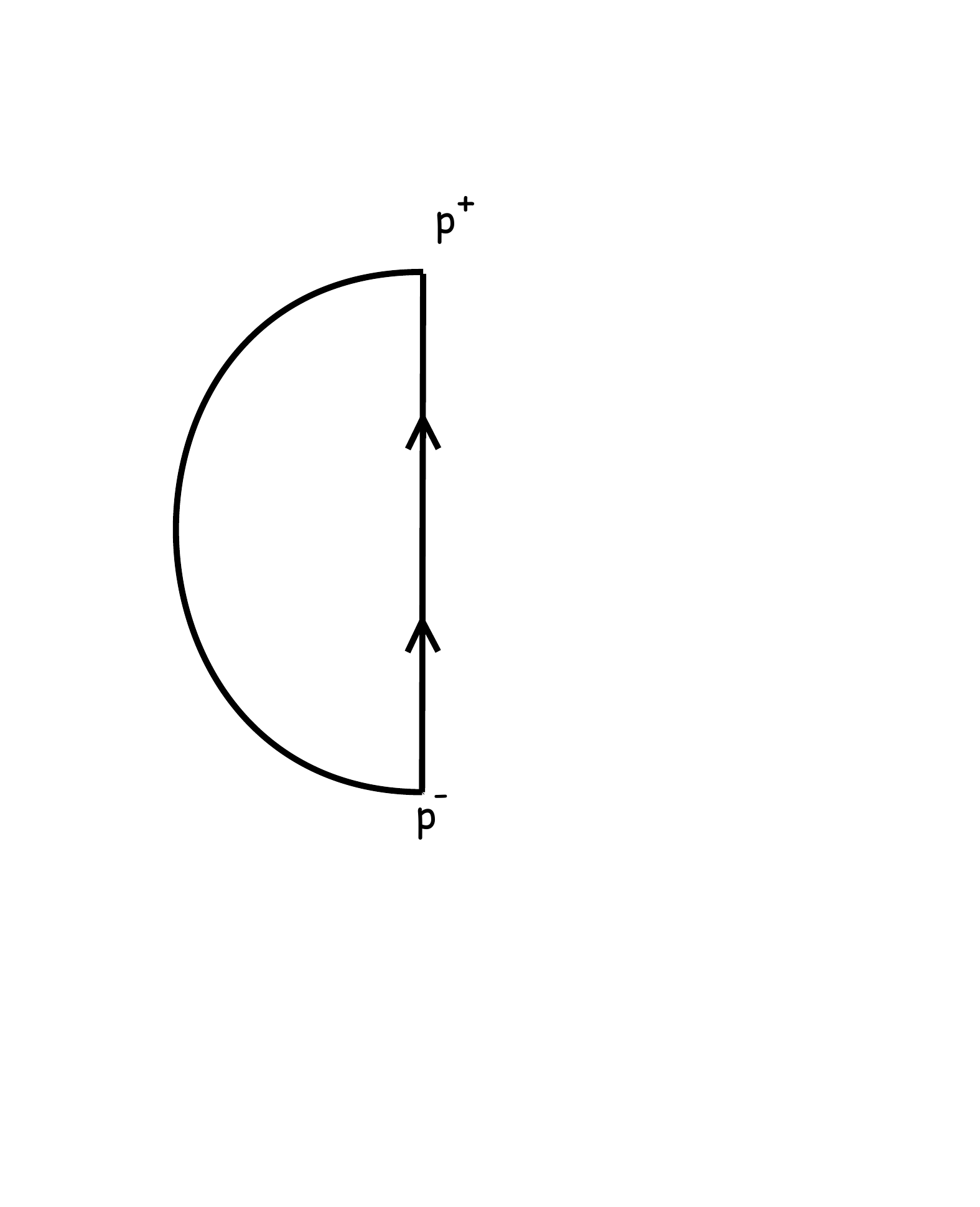}}
\\
  \subfloat[][]{\label{pliage3}\includegraphics[trim=0cm 10cm 5cm 0cm, clip=true, width=5cm]{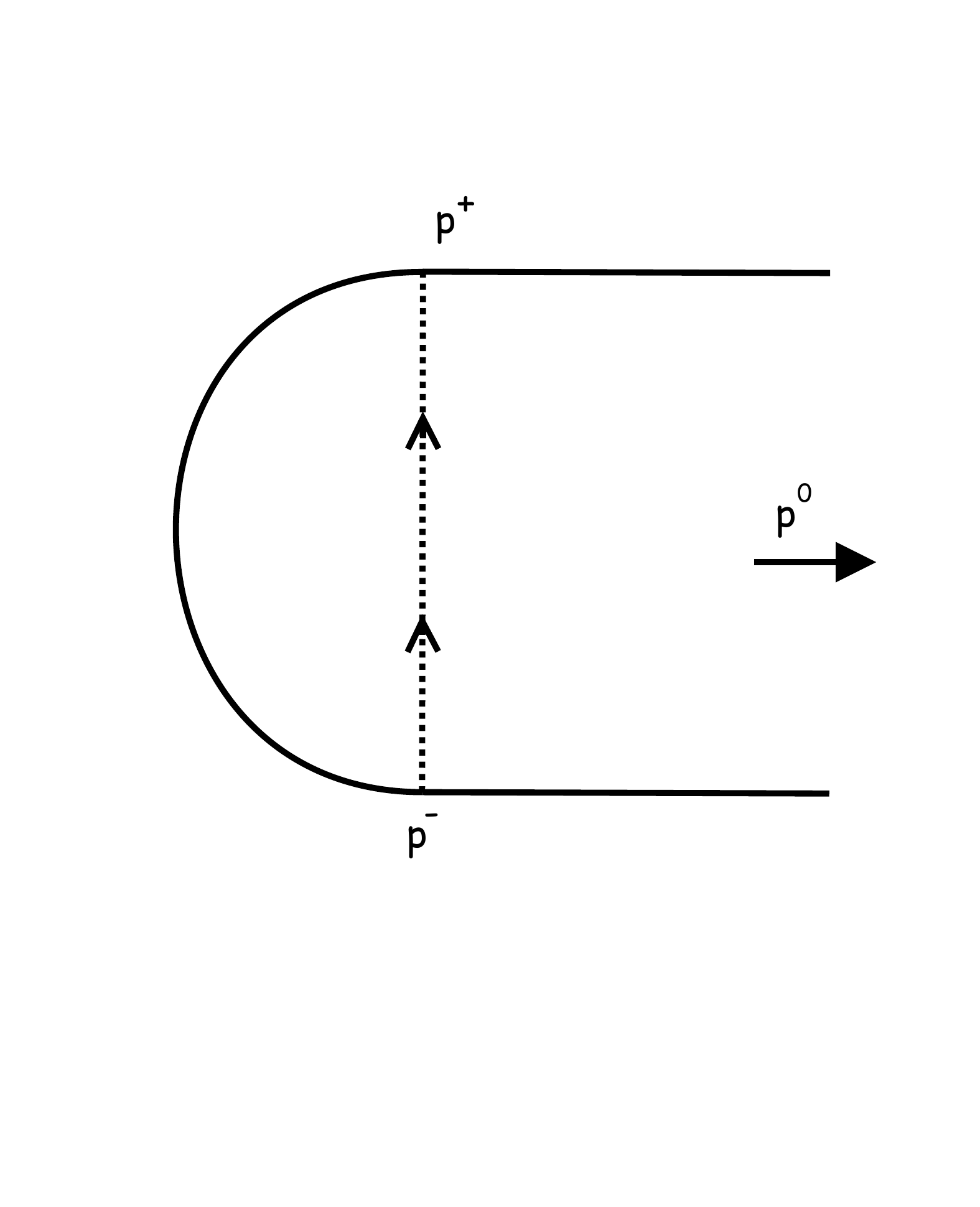}}                
  \subfloat[][]{\label{pliage4} \includegraphics[trim=0cm 10cm 5cm 0cm, clip=true, width=5cm]{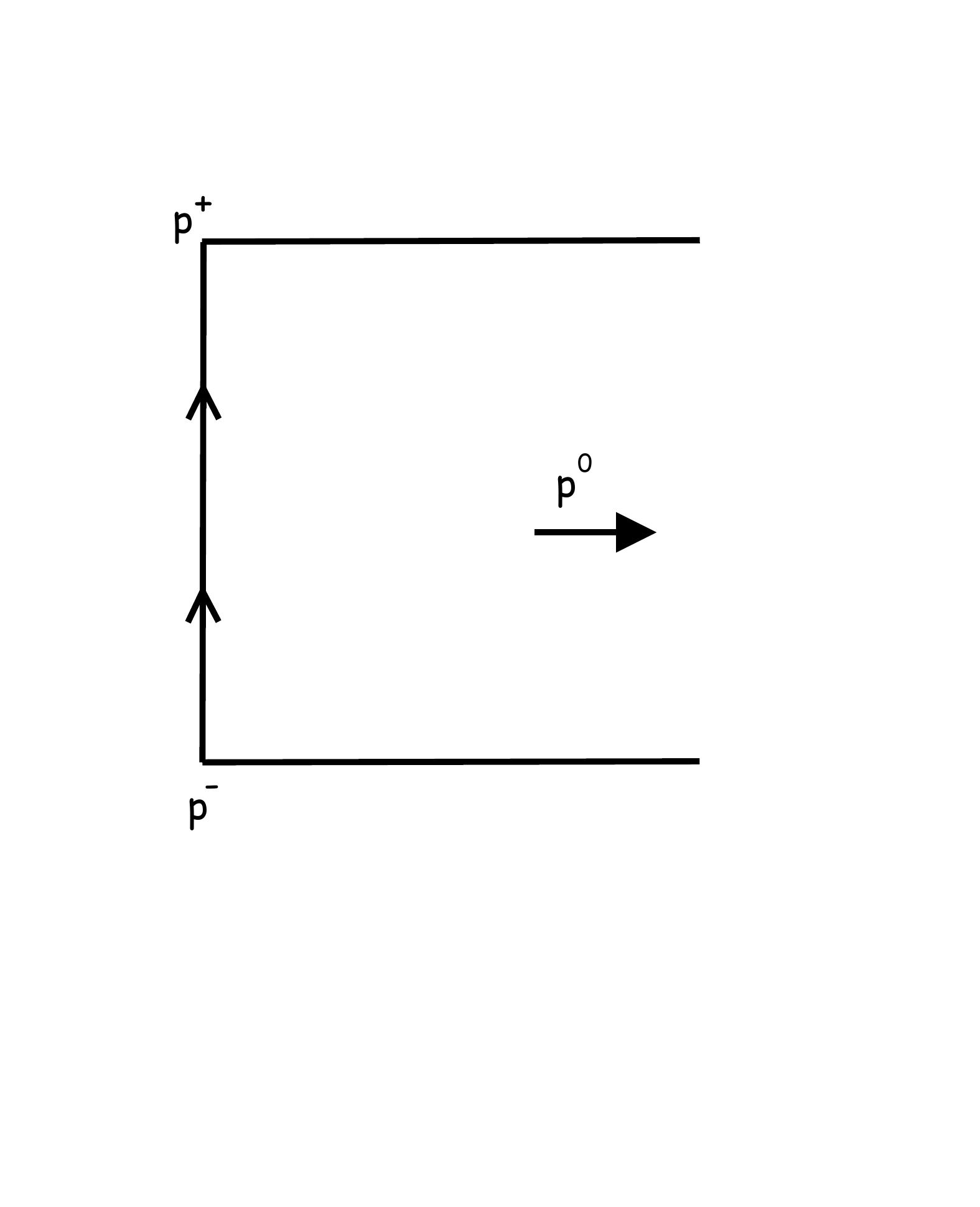}}
\caption{Les différentes configurations de la proposition \ref{hyp}}
\label{dessin_hyp}
\end{figure}
\end{center}
\subsubsection{Dynamique planaire}

Soit $\g$ un \'el\'ement planaire de $\s$, l'\'el\'ement $\g$ est conjugu\'e \`a la
matrice

$$
\begin{array}{cc}
\left(\begin{array}{ccc}
\alpha &  0        & 0   \\
 0     & \alpha    & 0   \\
 0     &  0        & \beta\\
\end{array}
\right) &
\begin{tabular}{l}
o\`u, $\alpha, \beta > 0$, $\alpha^2 \beta = 1$\\
et $\alpha,\, \beta \neq 1$.
\end{tabular}
\end{array}
$$

On note:
\begin{itemize}
\item $p_{\g}$ le point propre de $\P$ associ\'e \`a la valeur propre
$\beta$.

\item $D_{\g}$ la droite stable de $\P$ associ\'ee \`a la valeur
propre $\alpha$.
\end{itemize}

La dynamique des \'el\'ements planaires \'etant extrêmement simple, on
obtient facilement la propostion suivante.

\begin{prop}\label{planaire}
Soit $\O$ un ouvert proprement convexe et un
\'el\'ement planaire $\g \in \Aut(\O)$, alors, $\O$ est un triangle dont l'un des
sommets est $p_{\g}$ et le c\^ot\'e de $\O$ oppos\'e \`a $p_{\g}$ est
inclus dans la droite $D_{\g}$.
\end{prop}

\subsubsection{Dynamique quasi-hyperbolique}

Soit $\g$ un \'el\'ement quasi-hyperbolique de $\s$, l'\'el\'ement $\g$ est conjugu\'e
\`a la matrice

$$
\begin{array}{cc}
\left(\begin{array}{ccc}
\alpha &  1     &  0   \\
 0     & \alpha &  0   \\
 0     &  0     & \beta\\
\end{array}
\right)
 &
\begin{tabular}{l}
O\`u, $\alpha, \, \beta > 0$, $\alpha^2 \beta = 1$\\
et $\alpha, \,\beta \neq 1$.
\end{tabular}
\end{array}
$$

On note:
\begin{itemize}
\item $p^1_{\g}$ le point propre de $\P$ associ\'e \`a la valeur
propre $\beta$.

\item $p^2_{\g}$ le point propre de $\P$ associ\'e \`a la valeur
propre $\alpha$.

\item $D_{\g}$ la droite stable de $\P$ associ\'ee \`a la valeur
propre $\alpha$.
\end{itemize}

On peut d\'efinir l'axe d'un \'el\'ement quasi-hyperbolique qui agit sur un ouvert proprement convexe de la m\^eme façon que pour un \'el\'ement hyperbolique. La m\^eme d\'emonstration que dans le cas hyperbolique donne la proposition suivante:

\begin{prop}
Soient $\O$ un ouvert proprement convexe et un
\'el\'ement quasi-hyperbolique $\g \in \Aut(\O)$, alors, on a $p^1_{\g} ,\, p^2_{\g} \in \partial \O$.
\end{prop}

\begin{defi}
Soient $\O$ un ouvert proprement convexe et un
\'el\'ement quasi-hyperbolique $\g \in \Aut(\O)$, \emph{ l'axe de $\g$} que l'on notera $\Ax(\g)$ est le segment ouvert de la droite $(p^1_{\g}
 p^2_{\g})$ qui est inclus dans $\overline{\O}$ et dont les extr\'emit\'es sont $p^1_{\g}$ et $p^2_{\g}$.
\end{defi}

La figure \ref{qhyp} illustre la dynamique d'un \'el\'ement quasi-hyperbolique

\begin{figure}[!h]
\begin{center}
\includegraphics[trim=0cm 12cm 0cm 0cm, clip=true, width=8cm]{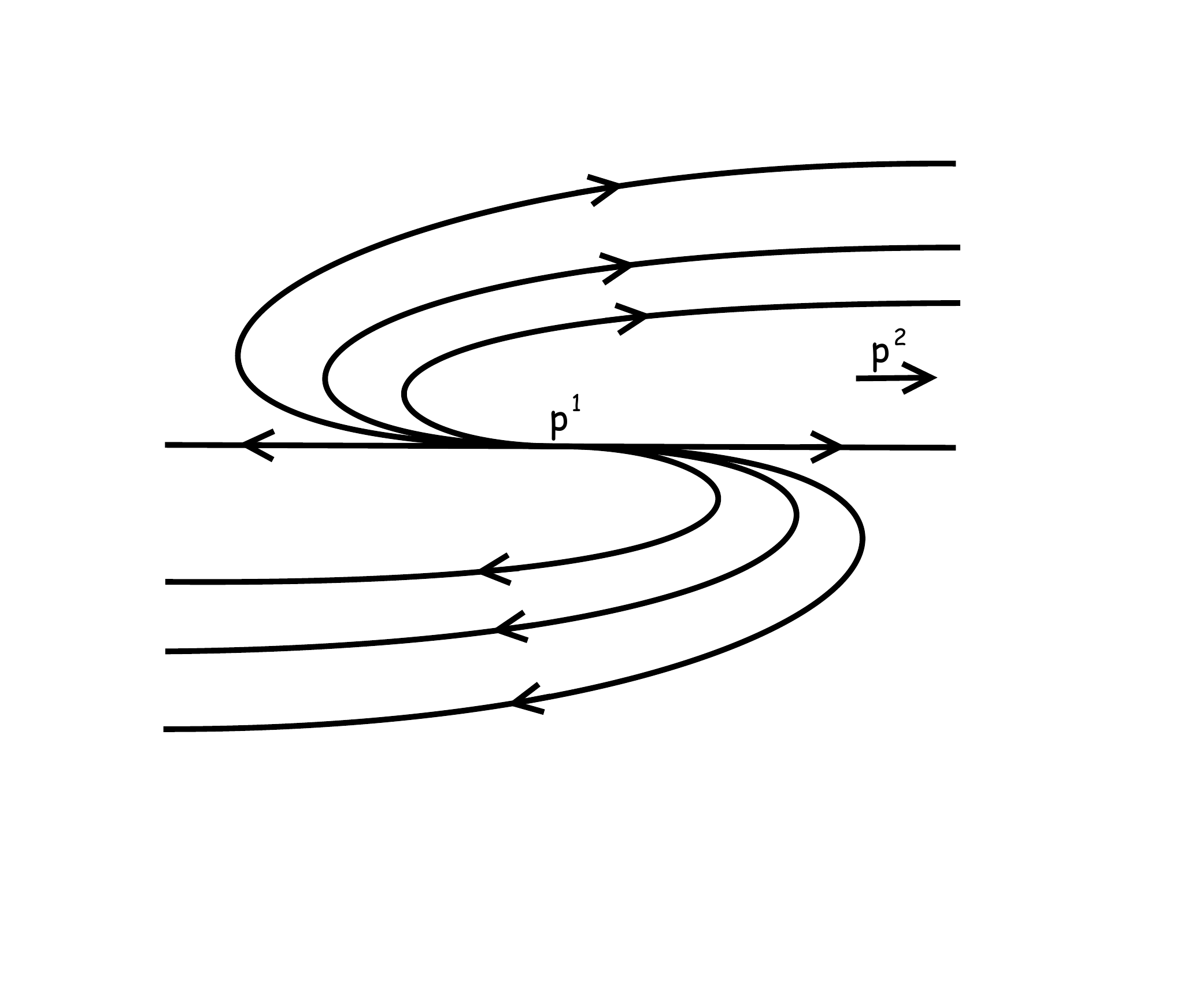}
\caption{Dynamique d'un \'el\'ement quasi-hyperbolique}\label{qhyp}
\end{center}
\end{figure}

\begin{prop}\label{quasihyp}
Soient $\O$ un ouvert proprement convexe et un
\'el\'ement quasi-hyperbolique $\g \in \Aut(\O)$. Alors,

\begin{itemize}
\item $\Ax(\g) \subset \partial \O$.

\item $\partial \O$ n'est pas $\Cc^1$ en $p^2_{\g}$ et les
demi-tangentes \`a $\partial \O$ en $p^2_{\g}$ sont $(p^1_{\g} p^2_{\g})$ et $D$.

\item $\partial \O$ est $\Cc^1$ en $p^1_{\g}$ et $T_{p^1_{\g}} \partial \O =
(p^1_{\g} p^2_{\g})$.
\end{itemize}
\end{prop}

\begin{proof}
On proc\`ede comme pour l'\'etude de la dynamique d'un \'el\'ement
hyperbolique. La droite $(p^1_{\g} p^2_{\g})$ et la droite $D$ d\'efinissent un
pavage en deux parties de $\P$. L'action de $\g$ sur $D$ est
parabolique et l'action de $\g$ sur $(p^1_{\g} p^2_{\g})$ est hyperbolique. Les points $p^1_{\g}$ et $p^2_{\g}$ d\'efinissent deux segments $S_1$ et $S_2$ de la droite $(p^1_{\g} p^2_{\g})$ dont les extr\'emit\'es sont $p^1_{\g}$ et $p^2_{\g}$.
Soit $x \in \P - (D \cup (p^1_{\g} p^2_{\g}))$, la courbe $(\g^t x)_{t \in \R}$ a pour limite les points $p^1_{\g}$ et $p^2_{\g}$ lorsque $t$ tend vers $\pm\infty$. Si on ajoute le segment $S_1$ ou bien $S_2$ \`a cette courbe on obtient une courbe $\C$ convexe et analytique en dehors des points $p^1_{\g}$ et $p^2_{\g}$. La courbe $\C$ n'est pas $\Cc^1$ en $p^2_{\g}$ et admet comme
demi-tangentes en $p^2_{\g}$, les droites $D$ et $(p^1_{\g} p^2_{\g})$. La courbe $\C$ est
$C^1$ en $p^1_{\g}$ et sa tangente est la droite $(p^1_{\g} p^2_{\g})$. Il faut
aussi remarquer que si $y \in \P - (D \cup (p^1_{\g}p^2_{\g}))$ est dans
l'autre composante connexe de $\P - (D \cup (p^1_{\g} p^2_{\g}))$, alors la
courbe obtenue par le m\^eme proc\'ed\'e, mais en ajoutant l'autre segment est une courbe convexe analytique en dehors des points
$p^1_{\g}$ et $p^2_{\g}$. On tire facilement de ceci les conclusions de la proposition.
\end{proof}

\subsubsection{Dynamique parabolique}

Soit $\g$ un \'el\'ement parabolique de $\s$, l'\'el\'ement $\g$ est conjugu\'e \`a la
matrice

$$
\left(\begin{array}{ccc}
1  & 1    & 0 \\
0  & 1    & 1\\
0  & 0    & 1\\
\end{array}
\right)
$$

On note:
\begin{itemize}
\item $p_{\g}$ l'unique point fixe de $\g$ sur $\P$.

\item $D_{\g}$ l'unique droite fixe de $\g$ sur $\P$.
\end{itemize}

La figure \ref{fpara} illustre la dynamique d'un \'el\'ement parabolique.

\begin{figure}[!h]
\begin{center}
\includegraphics[trim=0cm 10cm 0cm 0cm, clip=true, width=9cm]{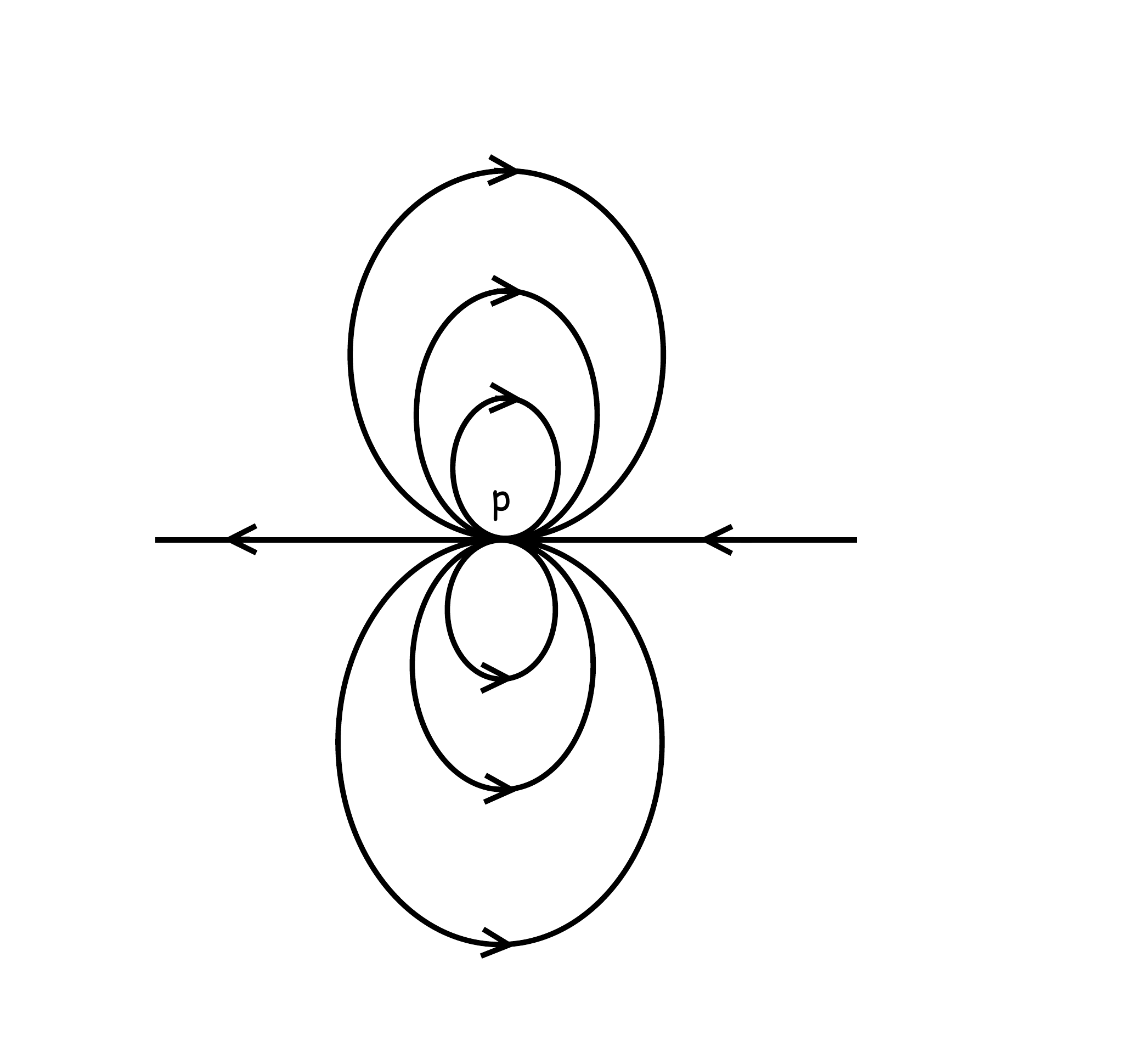}
\caption{Dynamique d'un \'el\'ement parabolique}\label{fpara}
\end{center}
\end{figure}

\begin{prop}\label{para}
Soient $\O$ un ouvert proprement convexe et un \'el\'ement parabolique $\g \in \Aut(\O)$ alors

\begin{itemize}
\item $p \in \partial \O$.

\item $\partial \O$ est $\Cc^1$ en $p$.

\item $T_p \partial \O = D$.

\item Le point $p$ n'appartient pas \`a un segment non trivial du bord de $\O$.
\end{itemize}
\end{prop}

\begin{proof}
On proc\`ede comme pour l'\'etude des \'el\'ements hyperboliques et
quasi-hyperboliques. Soit $x \notin D$, la courbe $\C$ obtenue en
ajoutant le point $p$ \`a la courbe $(\g^t x)_{t \in \R}$ d\'efinit
une courbe convexe analytique (la courbe $\C$ est en fait une ellipse).
\end{proof}

\subsubsection{Dynamique elliptique}

On ne se soucie pas des \'el\'ements elliptiques car le lemme de Selberg  permet de s'en d\'ebarrasser sans frais. On rappelle ici un \'enonc\'e de celui-ci.

\begin{lemm}[Selberg]
Tout sous-groupe de type fini de $\mathrm{GL}_n(\mathbb{C})$ est virtuellement sans torsion, c'est \`a dire poss\`ede un sous-groupe d'indice fini sans torsion.
\end{lemm}

\subsection{Calcul du centralisateur d'un \'el\'ement de $\G$ dans $\Aut(\O)$}

\begin{lemm}\label{para-para}
Soient $\g, \delta$ des \'el\'ements de $\s$ qui v\'erifient que:
\begin{itemize}
\item les \'el\'ements $\g, \delta$ sont hyperboliques tels que $\{p^+_{\g},p^-_{\g},p^0_{\g} \} = \{p^+_{\delta},p^-_{\delta},p^0_{\delta}\}$.

\item Ou bien, les \'el\'ements $\g, \delta$ sont paraboliques tels que $p_{\g} = p_{\delta}$ et $D_{\g} = D_{\delta}$.

\item Ou bien, les \'el\'ements $\g$ et $\delta$ sont quasi-hyperboliques tels que $p^1_{\g} = p^1_{\delta}$ et $p^2_{\g} = p^2_{\delta}$.
\end{itemize}

Si le groupe $<\g, \, \delta>$ est discret et pr\'eserve un ouvert
proprement convexe qui n'est pas un triangle alors le groupe $<\g, \, \delta>$ est cyclique infini.
\end{lemm}

\begin{proof}
Dans les trois cas, un calcul simple montre que le commutateur $\g \delta \g^{-1} \delta^{-1}$ de
$\g$ et $\delta$ est la matrice identit\'e ou une matrice conjugu\'ee \`a la matrice
suivante:

$$
\left(\begin{array}{ccc}
1  & 1  & 0  \\
0  & 1  & 0 \\
0  & 0  & 1 \\
\end{array}
\right)
$$
La proposition \ref{classi} montre que le dernier cas est
impossible donc $\g$ et $\delta$ commutent. Pour conclure, on peut consid\'erer la composante connexe de l'adh\'erence de Zariski $A$ du groupe $<\g, \, \delta>$ qui est un groupe de Lie ab\'elien. On a donc deux possibilit\'es:
\begin{itemize}
\item $A$ est isomorphe \`a $\R^2$.
\begin{itemize}
\item Si les \'el\'ements $\g,\delta$ sont hyperboliques alors $A$ est conjugu\'e au groupe suivant:
$$
\left\{ \left(\begin{array}{ccc}
\alpha  & 0  & 0  \\
0  & \beta  & 0  \\
0  & 0  & \gamma \\
\end{array}
\right) | \, \alpha,\beta,\gamma > 0 \textrm{ et } \alpha\beta\gamma=1 \right\}
$$

\item Si les \'el\'ements $\g,\delta$ sont quasi-hyperboliques alors $A$ est conjugu\'e au groupe suivant:
$$
\left\{ \left(\begin{array}{ccc}
\alpha  & \beta   & 0  \\
0       & \alpha  & 0  \\
0       & 0       & \gamma \\
\end{array}
\right) | \, \alpha,\gamma > 0 \, , \, \beta \in \R \textrm{ et } \alpha^2\gamma=1 \right\}
$$

\item Si les \'el\'ements $\g,\delta$ sont paraboliques alors $A$ est conjugu\'e au groupe suivant:
$$
\left\{ \left(\begin{array}{ccc}
1  & \alpha  & \beta  \\
0  & 1       & \alpha  \\
0  & 0       &   1 \\
\end{array}
\right) | \, \alpha,\beta \in \R \right\}
$$
\end{itemize}

\item $A$ est isomorphe \`a $\R$
\end{itemize}
Dans le second cas, comme le groupe $<\g, \, \delta>$ est un sous-groupe discret de $A$, le groupe $<\g, \, \delta>$
est cyclique. Pour conclure, il suffit donc de montrer que le premier cas est absurde. Pour cela, remarquons que A agit simplement transitivement sur:
\begin{itemize}
\item Chacune des quatre composantes connexes de $\P-({D_{\g}^{+,-}\cup D_{\g}^{+,0} \cup D_{\g}^{0,-}})$ si $\g$ est hyperbolique.

\item Chacune des deux composantes connexes de $\P-(D_{\g} \cup (p^1_{\g} p^2_{\g}))$ si $\g$ est quasi-hyperbolique.

\item $\P-D_{\g}$ si $\g$ est parabolique.
\end{itemize}
On note $X$ l'une de ces composantes connexes. Comme le groupe $\G=<\g, \, \delta>$ est Zariski-dense dans $A$,  c'est un r\'eseau cocompact de $A$, car $A$ est un groupe de Lie ab\'elien. L'enveloppe convexe de toute orbite d'un point de $X$ est alors \'egale \`a $X$. Par cons\'equent, si $\g$ et $\delta$ sont hyperboliques alors tout ouvert proprement convexe pr\'eserv\'e par $\G$ est un triangle, ce qui est absurde. Et sinon $\G$ ne peut pas pr\'eserver d'ouvert proprement convexe, ce qui est absurde.
\end{proof}

\begin{defi}
Soient $\O$ un ouvert proprement convexe de $\P$ et $\g, \delta \in \Aut(\O)$, on dira que $\g$ et $\delta$ ont \emph{les m\^emes caract\'eristiques g\'eom\'etriques} s'ils font partie d'un m\^eme groupe \`a un param\`etre de $\s$. Ceci entraine qu'ils ont les m\^emes points fixes et droites fixes.
\end{defi}

\begin{prop}\label{centra}
Soient $\O$ un ouvert proprement convexe qui n'est pas un triangle et $\g \in \Aut(\O)$, le
centralisateur d'un \'el\'ement hyperbolique (resp.
quasi-hyperbolique, resp. parabolique, resp. elliptique d'ordre diff\'erent de 2) $\g$ dans $\Aut(\O)$ est le
sous-groupe des \'el\'ements hyperboliques (resp. quasi-hyperboliques, resp. paraboliques, resp. elliptiques) de $\Aut(\O)$ qui ont les m\^emes caract\'eristiques g\'eom\'etriques que $\g$.
\end{prop}

\begin{proof}
Soit $\g \in \Aut(\O)$, rappellons que, comme $\O$ n'est pas un triangle, $\Aut(\O)$ ne contient pas d'\'el\'ement planaire. Raisonnons au cas par cas.
\begin{itemize}
\item Si $\g$ est hyperbolique alors le centralisateur de $\g$
dans $\s$ est l'ensemble des matrices diagonalisables dans la m\^eme
base que $\g$, et le lemme \ref{para-para} permet de conclure.

\item Si $\g$ est elliptique alors comme $\g$ n'est pas d'ordre 2, le centralisateur de $\g$
dans $\Aut(\O)$ est l'ensemble des \'el\'ements elliptiques de $\G$ qui ont les m\^emes espaces stables. Or, l'ensemble des \'el\'ements elliptiques de $\s$ qui pr\'eserve un plan $P$ et une droite $D$ avec $D \not\subset P$ est un groupe de Lie de dimension 1. C'est ce qu'il fallait d\'emontrer.

\item Si $\g$ est quasi-hyperbolique alors un simple calcul montre que le centralisateur de $\g$ dans $\Aut(\O)$ est l'ensemble des \'el\'ements quasi-hyperboliques $\delta$ qui v\'erifient $p^1_{\g} = p^1_{\delta}$ et $p^2_{\g} = p^2_{\delta}$. Le lemme \ref{para-para} conclut la d\'emonstration.

\item Si $\g$ est parabolique alors un simple calcul montre que le centralisateur de $\g$ dans $\Aut(\O)$ est l'ensemble des \'el\'ements paraboliques qui v\'erifient $p_{\g} = p_{\delta}$ et $D_{\g} = D_{\delta}$. Par cons\'equent le lemme \ref{para-para} conclut la d\'emonstration une nouvelle fois.
\end{itemize}
\end{proof}

\section{Irr\'eductibilit\'e et adh\'erence de Zariski}

\subsection{Irr\'eductibilit\'e}

On reproduit pour la commodit\'e du lecteur la d\'emonstration de la
proposition suivante due \`a Goldman dans \cite{Gold1}.

\begin{prop}[Goldman]\label{irreducgold}
Soit $\G$ un sous-groupe discret de $\s$ qui pr\'eserve un ouvert
proprement convexe $\O$ de $\P$, si $\G$ n'est pas virtuellement
ab\'elien alors $\G$ est irr\'eductible.
\end{prop}

\begin{proof}
Supposons que $\G$ n'est pas irr\'eductible, nous allons montrer que $\G$ est virtuellement ab\'elien. Alors le groupe $\G$ fixe un point ou une droite de $\P$. Par dualit\'e, on peut supposer que $\G$ fixe un
point $p \in \P$. Il faut distinguer 2 cas.
\begin{itemize}
\item Si $p \in \O$ alors, comme $\G$ agit proprement sur $\O$,
$\G$ est fini.

\item Si $p \notin \O$ alors $\G$ ne contient aucun \'el\'ement elliptique, par cons\'equent $\G$ est sans torsion. \`{A} pr\'esent, le faisceau $\mathcal{F}$ des droites concourantes en $p$ est pr\'eserv\'e par $\G$. La projection de $\mathcal{F}$ sur la surface $S=\O/_{\G}$ est un feuilletage de dimension 1 sans point singulier. La surface $S$ est donc un cylindre ou un tore. Le groupe $\G$ est donc ab\'elien.
\end{itemize}
\end{proof}

\begin{lemm}
Soit $\G$ un sous-groupe discret de $\s$ qui pr\'eserve un ouvert
proprement convexe $\O$ de $\P$. Si $\Gamma$ est virtuellement ab\'elien et $\mu(\Quo) < \infty$ alors $\O$ est un triangle, $\G$ contient une copie de $\Z^2$ d'indice fini et $\O/_{\G}$ est compact.
\end{lemm}

\begin{proof}
On peut supposer que $\Gamma$ est ab\'elien. L'espace $\Quo$ est de volume fini par cons\'equent $\G$ contient un \'el\'ement d'ordre infini. Pour faciliter la discussion, il est commode de distinguer le cas o\`u $\O$ est un triangle, du cas o\`u $\O$ n'est pas un triangle.

Dans le premier cas, le groupe $\Aut(\O)$ est un groupe de Lie ab\'elien isomorphe \`a $\R^2$ qui agit simplement transitivement sur $\O$. Par cons\'equent, $\mu(\Quo) < \infty$ si et seulement si $\G$ est un r\'eseau de $\Aut(\O)$. Ceci entraine que $\Quo$ est compact et que $\G$ contient
une copie de $\Z^2$ d'indice fini.

Enfin il faut montrer que le second cas est absurde. On peut utiliser la proposition \ref{centra}. Celle-ci montre
que le centralisateur dans $\Aut(\O)$ d'un \'el\'ement hyperbolique (resp. parabolique, resp. quasi-hyperbolique) est un \'el\'ement hyperbolique (resp. parabolique, resp.
quasi-hyperbolique) qui poss\`ede les m\^emes caract\'eristiques g\'eom\'etriques. Il vient que $\G$ est isomorphe \`a $\Z$.

\`{A} pr\'esent, nous allons construire un domaine fondamental convexe $F$ pour l'action de $\G$ sur $\O$. Pour cela, on consid\`ere un g\'en\'erateur $\g$ de $\G$ et on note un point fixe de $\g$: $p \in \partial \O$. On consid\`ere une droite $D$ passant par $p$ et tel que $\O \cap D \neq \varnothing$. On d\'efinit $F$ comme l'adh\'erence d'une composante connexe de $\O - \underset{n \in \Z}{\bigcup} \g^n D$. L'ensemble $F$ est un domaine fondamental pour l'action de $\G$ sur $\O$ et il contient un voisinage d'un point du bord de $\O$. Il vient que $F$ est de volume infini par le th\'eor\`eme \ref{mubord}.
\end{proof}

\begin{coro}\label{irre}
Soit $\G$ un sous-groupe discret de $\s$ qui pr\'eserve un ouvert
proprement convexe $\O$ de $\P$. Si $\mu(\Quo) < \infty$ et $\O$
n'est pas un triangle alors $\G$ est irr\'eductible.
\end{coro}

\subsection{Adh\'erence de Zariski}

\begin{prop}
Soit $\G$ un sous-groupe discret de $\s$, si $\G$ est infini et
irr\'eductible alors l'adh\'erence de Zariski de
$\G$ est:
\begin{itemize}
\item $\s$ ou

\item Un conjugu\'e de $\mathrm{SO}_{2,1}(\R)$.
\end{itemize}
\end{prop}

\begin{proof}
Tout sous-groupe Zariski-ferm\'e et irr\'eductible de $\s$ est:
\begin{itemize}
\item $\s$ ou

\item un conjugu\'e de $\mathrm{SO}_{3}(\R)$ ou

\item un conjugu\'e de $\mathrm{SO}_{2,1}(\R)$.
\end{itemize}
Par cons\'equent, comme l'adh\'erence de Zariski de $\G$ est un sous-groupe de $\s$
Zariski-ferm\'e, irr\'eductible et non born\'e. On obtient le r\'esultat
voulu.
\end{proof}

\begin{coro}\label{Zariski}
Soit $\G$ un sous-groupe discret de $\s$ qui pr\'eserve un ouvert
proprement convexe $\O$ de $\P$, si $\mu(\Quo) < \infty$ et $\O$
n'est pas un triangle alors l'adh\'erence de Zariski de $\G$ est:
\begin{itemize}
\item $\s$ ou

\item un conjugu\'e de $\mathrm{SO}_{2,1}(\R)$.
\end{itemize}
\end{coro}

\section{Existence d'un domaine fondamental convexe}

L'existence d'un domaine
fondamental convexe et localement fini pour l'action d'un
sous-groupe discret de $\ss$ sur un ouvert proprement convexe de
$\PP^n$ est dû \`a Jaejeong Lee (\cite{JL}). Nous donnons ici une courte d\'emonstration de ce r\'esultat.

\subsection{Fonction caract\'eristique d'un c\^one convexe}

Pour montrer ce r\'esultat, nous aurons besoin de nous placer
dans un cadre vectoriel. Rappelons donc quelques d\'efinitions.

\begin{defi}
Un \emph{c\^one} de $\R^{n+1}$ est une partie invariante par les
homoth\'eties lin\'eaires de rapport positifs. Un c\^one convexe est dit
\emph{proprement convexe} s'il ne contient pas de droite affine.
\end{defi}

Soit $\Cc$ un c\^one ouvert proprement convexe, on note $\Cc^* = \{f \in
(\R^{n+1})^*\, | \, \forall \, v \in \overline{\Cc}-\{ 0 \}, \, f(v)
> 0 \}$ le \emph{c\^one dual} de $\Cc$. C'est un c\^one ouvert proprement
convexe de $(\R^{n+1})^*$. Les points 1 \`a 5 du lemme suivant sont tir\'es d'un article
de Vinberg \cite{Vin}.

\begin{lemm}\label{vinberg}
Soit $\Cc$ un c\^one proprement convexe de $\R^n$, on
consid\`ere l'application suivante appel\'ee \emph{fonction caract\'eristique}
de $\Cc$.

$$
\begin{array}{cccc}
\varphi_{\Cc} :& \Cc & \rightarrow & \R \\
           & M  & \mapsto     & \int_{\Cc^*} e^{-f(M)}df
\end{array}
$$

\begin{enumerate}
\item La fonction $\varphi_{\Cc}$ est bien d\'efinie.

\item La fonction $\varphi_{\Cc}$ est analytique.

\item La fonction $\varphi_{\Cc}$ est une submersion.

\item Le hessien de $\varphi_{\Cc}$ est d\'efini positif.

\item $\underset{M \rightarrow M_{\infty} \in
\partial \Cc}{\lim}  \varphi_{\Cc}(M) =
+ \infty $.

\item $\forall M \in \Cc$, $\forall v \in \overline{\Cc}$, $\underset{\lambda
\rightarrow +\infty}{\lim} \varphi_{\Cc}(M+\lambda v) = 0$.
\end{enumerate}
\end{lemm}

\begin{proof}
\item
\begin{enumerate}
\item Soit $M \in \Cc$ fix\'e, on note $\O_M = \{ f \in
\Cc^* \,|\, f(M)=1 \}$, $E$ l'espace vectoriel engendr\'e par $\O_M$ et $\textrm{Vol}_E$ la mesure de Lebesgue
canonique du sous-espace $E$ de $(\R^{n+1})^*$. La propre convexit\'e de $\Cc$ entraine que $\O_M$ est une partie compacte de $E$, le calcul suivant conclut.
$$\int_{\Cc^*} e^{-f(M)}df = \int_{\O_M}
 \Big(\int_{\R^*_+} e^{-\lambda} d\lambda \Big) d\textrm{Vol}_{E}  =
\textrm{Vol}_{E}(\O_M) < +\infty.$$

\item C'est clair.

\item Le calcul de $d\varphi_{\Cc}$ est imm\'ediat et donne:
$$\begin{array}{cccc}
d\varphi_{\Cc}: & \Cc & \rightarrow & (\R^{n+1})^*\\
            &  M & \mapsto     & \int_{\Cc^*} fe^{-f(M)}df
\end{array}$$
Donc $\varphi_{\Cc}$ est une submersion.

\item Le calcul de $d^2\varphi_{\Cc}$ est lui aussi imm\'ediat et donne:
$$\begin{array}{cccc}
d^2\varphi_{\Cc}: & \Cc & \rightarrow & \textrm{Sym}(\R^{n+1} \times \R^{n+1},\R)\\
            &  M & \mapsto     & (u,v) \mapsto \int_{\Cc^*} f(u)f(v)e^{-f(M)}df
\end{array}$$
o\`u $\textrm{Sym}(\R^{n+1} \times \R^{n+1},\R)$ d\'esigne l'espace
des formes bilin\'eaires sym\'etriques sur $\R^{n+1}$. Donc le hessien
de $\varphi_{\Cc}$ est d\'efini positif.

\item Soit $M_{\infty} \in \partial \Cc$, il existe $f \in
\partial \Cc ^*$ tel que $f(M_{\infty})=0$. On consid\`ere un compact
d'int\'erieur non vide $K$ inclus dans $\Cc^*$ et on d\'efinit le
sous-ensemble $L = K + \{ \lambda f \}_{\lambda > 0} \subset \Cc^*$.
Enfin, on note $c = \underset{f \in K}{\sup} \{f(M_{\infty})\}$. Le calcul suivant permet de conclure.

$$\lim_{M \rightarrow M_{\infty} \in
\partial \Cc} \varphi_{\Cc}(M) \geqslant \varphi_{\Cc}(M_{\infty}) \geqslant \int_{L} e^{-f(M_{\infty})}df \geqslant \int_{L} e^{-c} df = +\infty$$

La premi\`ere in\'egalit\'e est une cons\'equence du lemme de Fatou, les
autres sont triviales.

\item Soient $M \in \Cc$ et $v \in \overline{\Cc}$, $\varphi_{\Cc}(M+\lambda v) =
\int_{\Cc^*} e^{-\lambda f(v)} e^{-f(M)}df$, l'int\'egrant est domin\'e
par $M \mapsto e^{-f(M)}$ qui est int\'egrable et il tend vers $0$
lorsque $\lambda \rightarrow +\infty$. Le th\'eor\`eme de convergence
domin\'ee entraine que  $\underset{\lambda \rightarrow
+\infty}{\lim} \varphi_{\Cc}(M+\lambda v) = 0$.
\end{enumerate}
\end{proof}

\begin{defi}
Soit $\Sigma$ une hypersurface de $\R^{n+1}$, on dit que
$\Sigma$ est \emph{localement convexe} (resp. \emph{localement
strictement convexe}) lorsque tout point de $\Sigma$ poss\`ede un
voisinage dans $\Sigma$ qui est une partie du bord d'un convexe
(resp. d'un convexe strictement convexe) de $\R^{n+1}$.
\end{defi}

\begin{rem}
Soit $\Sigma$ une hypersurface de $\R^{n+1}-\{ 0 \}$ localement
strictement convexe, pour tout point $p$ de $\Sigma$, et tout
ouvert $V$ de $\R^{n+1}$ suffisament petit contenant $p$,
$V-\Sigma$ poss\`ede deux composantes connexes. La stricte convexit\'e
 permet de d\'efinir la composante int\'erieure et la composante
ext\'erieure. En particulier, pour tout point $p$ de $\Sigma$, on
peut donner un sens \`a la phrase "le vecteur $\overrightarrow{0p}$
pointe vers l'ext\'erieur (resp. l'int\'erieur) de $\Sigma$".
\end{rem}

\begin{defi}
Soit $\Sigma$ une hypersurface localement strictement convexe de $\R^{n+1}-\{ 0 \}$, on dit que $\Sigma$ est \emph{radiale} si pour point $p$ de $\Sigma$ le vecteur $\overrightarrow{0p}$ pointe vers l'int\'erieur de $\Sigma$.
\end{defi}

\begin{rem}
Toute hypersurface de $\R^{n+1}-\{ 0 \}$ localement strictement
convexe, radiale et propre est le bord d'un ouvert strictement
convexe de $\R^{n+1}-\{ 0 \}$.
\end{rem}

\begin{defi}
Soient $\Sigma$ une hypersurface localement strictement convexe,
radiale et propre de $\R^{n+1}-\{ 0 \}$ et $\Cc$ un c\^one ouvert
convexe de $\R^{n+1}$, on dit que $\Sigma$ est asymptote au c\^one
$\Cc$ lorsque:
\begin{itemize}
\item Le c\^one $\Cc$ contient $\Sigma$.

\item Toute demi-droite affine incluse dans $\Cc$ intersecte
$\Sigma$.
\end{itemize}
\end{defi}

\begin{rem}
On peut remarquer que la derni\`ere condition est \'equivalente au
fait que l'hyperplan \underline{affine} tangent \`a $\Sigma$ en un point $x$
converge vers un hyperplan \underline{vectoriel} tangent \`a
$\partial \Cc$ le bord de $\Cc$, lorsque la droite engendr\'ee par $x$
converge dans $\PP^n$ vers une droite incluse dans $\partial \C$.
\end{rem}

Soit $\pi:\R^{n+1}-\{ 0 \} \rightarrow \mathbb{P}^n$ la projection
naturelle, le lemme \ref{vinberg}  donne le corollaire
suivant.

\begin{coro}
Soit $\Cc$ un c\^one ouvert proprement convexe de $\R^{n+1}$, pour tout $m > 0$, l'hypersurface $\varphi_{\Cc}^{-1}(m)$ est une hypersurface ferm\'ee de $\R^{n+1}$, strictement convexe, radiale, propre et asymptote au c\^one $\Cc$. De
plus, toute application lin\'eaire de d\'eterminant 1 qui pr\'eserve $\Cc$, pr\'eserve
$\varphi_{\Cc}$ et donc aussi les hypersurfaces $(\varphi_{\Cc}^{-1}(m))_{m
\in \R^*_+}$.
\end{coro}

\subsection{Existence d'un domaine fondamental convexe}

\begin{defi}
Soient $X$ un espace topologique et $\G$ un groupe qui agit sur
$X$ par hom\'eomorphisme, on dit qu'une partie ferm\'ee $D \subset X$ est un \emph{domaine
fondamental pour l'action de $\G$ sur $X$} lorsque:
\begin{itemize}
\item $\underset{\g \in \G}{\bigcup} \g D = X$.

\item $\forall \g \neq 1$, $\g \overset{\circ}{D} \cap
\overset{\circ}{D} = \varnothing$.
\end{itemize}
De plus, un domaine fondamental $D$ pour l'action de $\G$ sur $X$
est dit \emph{localement fini} lorsque:
\begin{itemize}
\item $\forall K$ compact de $X$, $\{\g \in \G \,|\, \g D \cap K
\neq \varnothing \}$ est fini.
\end{itemize}
\end{defi}

Nous allons construire un domaine fondamental convexe et
localement fini pour l'action de $\G$ sur $\O$. Pour cela on
introduit les objets suivants. On note $\Cc$ l'une des deux
composantes connexes de $\pi^{-1}(\O)$. On pose $\varphi=\varphi_{\Cc}$ la
fonction caract\'eristique de $\Cc$. Le groupe $\G$ agit sur $\Cc$ en pr\'eservant
les lignes de niveau de $\varphi$. On note $\Sigma =
\varphi^{-1}(1)$, c'est une hypersurface pr\'eserv\'ee par $\G$.

Enfin, on d\'efinit $\psi_X$ pour $X \in \Sigma$ la forme lin\'eaire sur $\R^{n+1}$
qui donne l'\'equation de l'hyperplan vectoriel tangent \`a $\Sigma$
en $X$ et qui v\'erifie $\psi_X(X)=1$. Autrement dit, on a $\psi_X =
\frac{d\varphi_X}{d\varphi_X(X)}$.

\begin{defi}
On reprend les notations introduites. Soit $X_0 \in \Sigma$ dont
le stabilisateur dans $\G$ est trivial, \emph{le domaine de
Dirichlet-Lee pour l'action de $\G$ sur $\Sigma$ bas\'e en $X_0$}
est, l'ensemble:
$$
D_{X_0} = \{ X \in \Sigma \,|\, \forall \g \neq 1, \, \psi_{X_0}(X) \leqslant \psi_{X_0}(\g X) \}
$$
\vspace{.005cm}

Si on note $x_0 = \pi(X_0)$, on a une d\'efinition naturelle du
\emph{domaine de Dirichlet-Lee pour l'action de $\G$ sur $\O$ bas\'e
en $x_0$}, il s'agit de l'ensemble $\Delta_{x_0} = \pi(D_{X_0})$.
\end{defi}

\begin{theo}[Lee]\label{Lee}
Soient $\G$ un sous-groupe discret de $\ss$ qui pr\'eserve un ouvert
proprement convexe $\O$ de $\mathbb{P}^n$ et un point $x_0$ dont le stabilisateur dans $\G$ est trivial. Le domaine de
Dirichlet-Lee pour l'action de $\G$ sur $\O$ bas\'e en $x_0$ est un
domaine fondamental convexe et localement fini pour l'action de
$\G$ sur $\O$.
\end{theo}

Pour montrer le th\'eor\`eme de Lee, nous aurons besoin de deux lemmes.

\begin{lemm}\label{existence}
Pour tout points $X_0, X \in \Sigma$, $\underset{\g \rightarrow \infty}{\lim}
\psi_{X_0}(\g X) = +\infty$.
\end{lemm}

\begin{proof}
Le lemme \ref{vinberg} montre que l'hypersurface localement strictement convexe, radiale et propre $\Sigma$ est asymptote au c\^one proprement convexe $\Cc$. Par cons\'equent, l'intersection du c\^one $\Cc$ avec tout demi-espace de la forme $\{ \psi_{X_0} \leqslant c \}$, o\`u $c \in \R$ est une partie born\'ee de $\R^{n+1}$. Mais, le point $\g X$ tend vers l'infini lorsque $\g$ tend vers l'infini. Ceci conclut la d\'emonstration du lemme.
\end{proof}

Pour le second lemme, il faut introduire plusieurs objets. Soit $X \in \Sigma$, on note $T_X$ l'hyperplan affine
tangent \`a $\Sigma$ en $X$, il est donn\'e par l'\'equation $\psi_X=1$.
Si $X \in \Sigma$ et $\g \in \G$, alors on note $\mu^{\g}_{X}$ la
forme lin\'eaire $\psi_{X}-\psi_{X} \circ \g$. Et on note
$H^{\g}_{X}$ l'hyperplan vectoriel $\mu^{\g}_{X}=0$, et si
$x=\pi(X)$, on note $M^{\g}_{x}$ son image dans $\P$.

Nous aurons aussi besoin d'une d\'efinition.

\begin{defi}
Soient $\Cc$ un ouvert convexe de $\R^{n+1}$ et $X \in \partial \Cc$, un \emph{hyperplan d'appui} \`a $\Cc$ en $X$ est un hyperplan affine de $\R^{n+1}$ contenant le point $X$ mais ne rencontrant pas $\Cc$.

Soient $\O$ un ouvert convexe de $\PP^n$ et $x \in \partial \O$, un \emph{hyperplan d'appui} \`a $\O$ en $x$ est un hyperplan projectif de $\PP^n$ contenant le point $x$ mais ne rencontrant pas $\O$.
\end{defi}

\begin{lemm}\label{lemessen}
Soient $(\g_p)_{p \in \N} \in \G^{\N}$ tel que $\underset{p
\rightarrow \infty}{\lim} \g_p= \infty$ et $x_0 \in \O$, on
suppose que la suite $\g_p x_0$ converge vers un point $x_{\infty}
\in
\partial \O$. Alors, la suite des hyperplans $M^{\g_p}_{x_0}$
converge vers un hyperplan d'appui $M_{\infty}$ \`a $\O$ en $x_{\infty}$.
\end{lemm}

\begin{proof}
La figure \ref{lemdom} peut aider \`a suivre la d\'emonstration.

\begin{figure}[!h]
\begin{center}
\includegraphics[trim=0cm 10cm 0cm 0cm, clip=true, width=10cm]{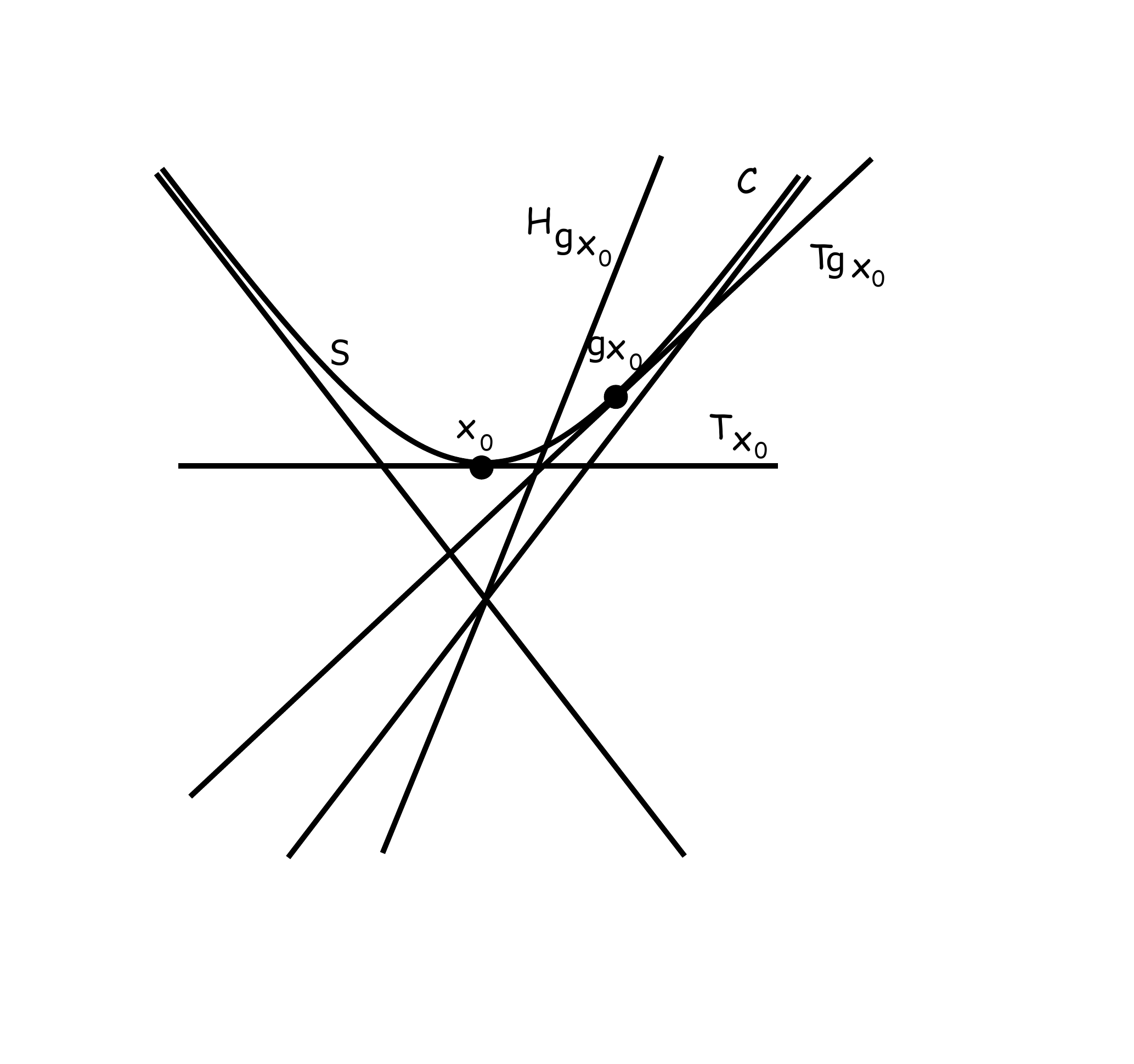}
\caption{D\'emonstration du lemme \ref{lemessen}}\label{lemdom}
\end{center}
\end{figure}

On note $X_0$ le point de $\Sigma$ tel que $\pi(X_0)=x_0$. Tout
d'abord, l'hypersurface $\Sigma$ est asymptote au c\^one convexe $\Cc$ par cons\'equent les hyperplans affines $T_{\g_p X_0}$
convergent vers un hyperplan vectoriel $T$ de $\R^{n+1}$ qui est un hyperplan d'appui \`a $\Cc$ contenant la droite $x_{\infty}$. On veut montrer que la suite d'hyperplan vectoriel $\pi^{-1}(M^{\g_p}_{x_0})=H^{\g_p}_{X_0}$ converge aussi vers $T$.

Commençons par remarquer que $\pi^{-1}(M^{\g_p}_{x_0})=H^{\g_p}_{X_0} =$Vect$(T_{X_0} \cap T_{\g_p X_0})$. En effet, l'intersection $T_{X_0} \cap T_{\g_p X_0}$ est incluse dans $H^{\g_p}_{X_0}$ et les hyperplans affines $T_{X_0}$ et $T_{\g_p X_0}$ ne sont pas parall\`eles donc $\dim(T_{X_0} \cap T_{\g_p X_0})=n-1$.

Ainsi, la suite $(H^{\g_p}_{X_0})_{p \in \N}$ a donc la m\^eme limite que la suite
$(T_{\g_p X_0})_{p \in \N}$ et puisque l'hyperplan affine $(T_{\g_p X_0})_{p \in \N}$ converge vers l'hyperplan \underline{vectoriel} $T$, on a donc d\'emontr\'e que la suite $(H^{\g_p}_{X_0})_{p \in \N}$ converge vers un hyperplan d'appui au c\^one convexe $\Cc$. Ce qui conclut la d\'emonstration.
\end{proof}

On obtient le corollaire suivant:

\begin{coro}\label{finitude}
Pour tout point $x_0 \in \O$, la famille des hyperplans
$\{M^{\g}_{x_0}\}_{\g \in \G}$ est localement finie dans $\O$.
\end{coro}

\begin{proof}[D\'emonstration du th\'eor\`eme \ref{Lee}]
L'ensemble $\Delta_{x_0}$ est une partie convexe de
$\O$ puisqu'elle est obtenue comme intersection de demi-espaces de
$\O$. Pour montrer que c'est un domaine fondamental il suffit de
montrer que l'$\underset{\g \in \G}{\inf}\{ \psi_{X_0}(\g X)\}$
est atteint pour tout point $X \in \Sigma$. Ceci est une cons\'equence
directe du lemme \ref{existence}. La partie $\Delta_{x_0}$ est donc un
domaine fondamental convexe, il est localement fini car la famille
des hyperplans $M^{\g}_{x_0}$ est localement finie dans $\O$
(corollaire \ref{finitude}).
\end{proof}

\begin{prop}\label{proprete}
Soient un espace localement compact $X$ et un
groupe discret $\G$ qui agit par hom\'eomorphisme sur $X$. Supposons qu'il existe un domaine fondamental $D$ localement fini pour l'action de $\G$ sur $X$. L'application naturelle $p:D \rightarrow X/_{\G}$ est propre.
\end{prop}

\begin{proof}
On va montrer que l'image r\'eciproque de tout suite convergente est
une suite incluse dans un compact de $D$. Soient une suite de points $(s_n)_{n\in \N} \in
(X/_{\G})^{\N}$ telle que $\underset{n \rightarrow \infty}{\lim} s_n= s
\in X/_{\G}$ et une suite $t_n \in D$ tel que $p(t_n)= s_n$, il existe des \'el\'ements $\g_n
\in \G$ tel que la suite $(\g_n t_n)_{n \in \N}$ converge vers un
point $t \in D$ qui v\'erifie $p(t)=s$. Par cons\'equent si on consid\`ere $K$ un voisinage compact de $t$ alors pour $n$
assez grand on a $\g_n D \cap K \neq \varnothing$. L'ensemble $D$ est un
domaine fondamental localement fini, il n'y a donc qu'un nombre
fini de $\g \in \G$ qui v\'erifient $\g D \cap K \neq
\varnothing$. Par cons\'equent, la suite $(\g_n)_{n\in \N}$ est finie et la suite $(t_n)_{n\in \N}$ est incluse dans un compact de $X$.
\end{proof}

\subsection{Locale finitude \`a l'infini en dimension 2}

\begin{rem}
Dans cette partie on se place en dimension 2. De plus, \`a partir de maintenant si $A$ est une partie de $\O$ on d\'esignera par $\overline{A}$ son adh\'erence dans $\P$ et $\overline{A} \cap \O$ son adh\'erence dans $\O$. De cette façon, on \'evitera toute ambiguït\'e.
\end{rem}

\begin{defi}\label{secteur}
Soient $\O$ un ouvert proprement convexe de $\P$ et un \'el\'ement $\g$ de $\Aut(\O)$, un \emph{secteur} de $\g$ est l'adh\'erence dans $\O$ de l'enveloppe convexe dans $\O$ de l'orbite d'une partie compacte non vide de $\O$ sous l'action de $\g$.
\end{defi}

La forme des secteurs est tr\`es variable suivant la dynamique de l'\'el\'ement $\g$. On \'etudie dans la proposition suivante de façon exhaustive la forme de ces derniers. La d\'emonstration de cette proposition est une simple cons\'equence de l'\'etude de la dynamique des \'el\'ements de $\Aut(\O)$ faite dans la partie \ref{dyna}.

La figure \ref{formesecteur} illustre la proposition suivante.

\begin{figure}[!h]
\begin{center}
\includegraphics[trim=0cm 17cm 5cm 3cm, clip=true, width=10cm]{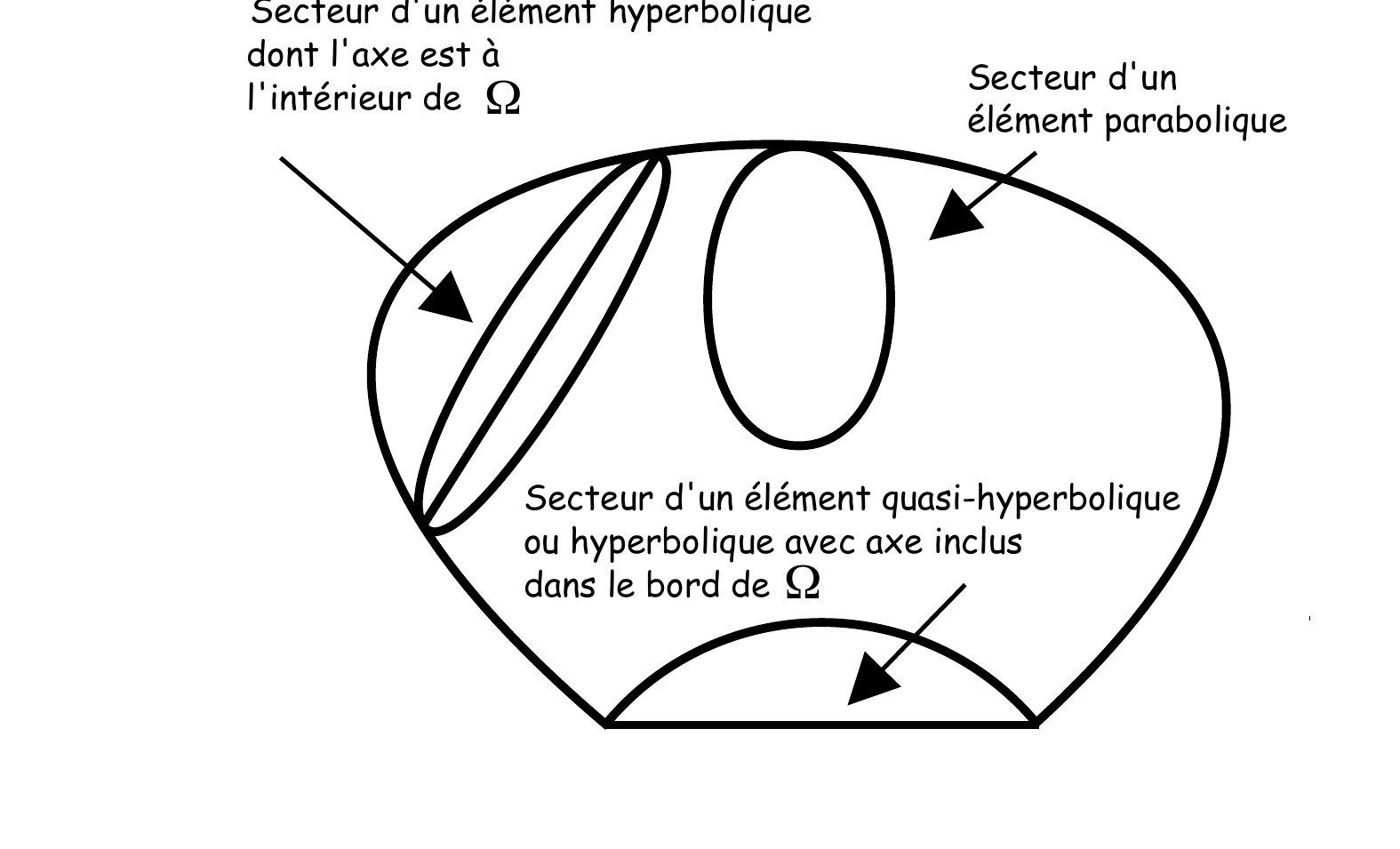}
\caption{Forme des secteurs}\label{formesecteur}
\end{center}
\end{figure}

\begin{prop}
Soient $\O$ un ouvert proprement convexe de $\P$ et $\G$ un
sous-groupe discret de $\s$ qui pr\'eserve $\O$ et un \'el\'ement $\g$ de $\G$.

\begin{itemize}
\item Si $\g$ est hyperbolique et $\Ax(\g) \subset \O$ alors tout secteur de $\g$ est un ferm\'e $\F$ de $\O$, $\g$-stable, convexe et tel que $\F \supset \Ax(\g)$.

\item Si $\g$ est hyperbolique, $\Ax(\g) \subset \partial \O$ et $p^0_{\g} \notin \partial \O$ alors tout secteur de $\g$ est un ferm\'e $\F$ de $\O$, $\g$-stable, convexe tel que $\overline{\F} \cap \partial \O =
\overline{\Ax(\g)}$.

\item Si $\g$ est quasi-hyperbolique alors tout secteur de $\g$ est un ferm\'e $\F$ de $\O$, $\g$-stable, convexe et tel que $\overline{\F} \cap \partial \O = \overline{\Ax(\g)}$.

\item Si $\g$ est parabolique alors tout secteur de $\g$ est un ferm\'e $\F$ de $\O$, $\g$-stable, convexe et tel que $\overline{\F} \cap \partial \O = \{ p_{\g} \}$.

\item Si $\g$ est elliptique alors tout secteur de $\g$ est un ferm\'e de $\O$, $\g$-stable et convexe.
\end{itemize}
\end{prop}

La proposition suivante d\'ecrit la propri\'et\'e essentielle des secteurs.

\begin{prop}
Soient $\O$ un ouvert proprement convexe de $\P$ et un \'el\'ement $\g$ de $\Aut(\O)$, si $\F\subset \F'$ sont deux secteurs de $\g$ alors $<\g>$ agit cocompactement sur $\overline{\F-\F'}\cap \O$.
\end{prop}

\begin{proof}
Une \'etude exhaustive en distinguant les cas:
\begin{itemize}
\item l'\'el\'ement $\g$ est hyperbolique avec $\Ax(\g) \subset \O$,

\item l'\'el\'ement $\g$ est quasi-hyperbolique ou hyperbolique avec $\Ax(\g) \subset \partial \O$,

\item l'\'el\'ement $\g$ est parabolique,

\item l'\'el\'ement $\g$ est elliptique,
\end{itemize}
et en regardant la figure \ref{formesecteur} rend cette proposition claire.
\end{proof}

\begin{prop}\label{inftyfini}
Soit $\G$ un sous-groupe discret de $\s$ qui pr\'eserve un ouvert
proprement convexe $\O$ de $\P$, on se donne un domaine Dirichlet-Lee $D$ pour l'action de $\G$ sur $\O$.
Soient $\g \in \G$ et $\F$ un secteur de $\g$, alors le domaine $D$ ne rencontre qu'un nombre fini d'images $(\delta \F)_{\delta \in \G}$.
\end{prop}

Nous allons avoir besoin de plusieurs lemmes pr\'eliminaires.

\begin{lemm}\label{rajout1}
Soit $\G$ un sous-groupe discret de $\s$ qui pr\'eserve un ouvert
proprement convexe $\O$ de $\P$ qui n'est pas un triangle. Si le groupe $\G$ contient un élément hyperbolique avec $\Ax(\g) \subset \partial \O$ alors il existe un ouvert proprement convexe $\O'$ préservé par $\G$ tel que $\Ax(\g) \subset \O'$.
\end{lemm}

\begin{proof}
Le lemme \ref{hyp} montre qu'alors le point attractif et le point répulsif de $\g$ ne sont pas $\Cc^1$. On considère le triangle $T_{\g}$ de sommets $p^+_{\g}, p^-_{\g}, p^0_{\g}$ qui n'est pas inclus dans $\O$ et dont l'un des côtés est $\Ax(\g)$. On prend pour $\O'$ la réunion de $\O$ et de l'orbite de $T_{\g} \cup \Ax(\g)$ sous $\G$. Le convexe $\O'$ est proprement convexe car le convexe $\O$ n'est pas un triangle.
\end{proof}

\begin{lemm}\label{rajout2}
Soit $\G$ un sous-groupe discret de $\s$ qui pr\'eserve un ouvert
proprement convexe $\O$ de $\P$ qui n'est pas un triangle. Si le groupe $\G$ contient un élément parabolique ou un élément quasi-hyperbolique ou un élément hyperbolique avec $\Ax(\g) \subset \partial \O$ alors le quotient $\Quo$ n'est pas compact.
\end{lemm}

\begin{proof}
Commençons par supposer que le groupe $\G$ contient un élément parabolique ou un élément quasi-hyperbolique $\g$. On va exhiber une suite de points $(x_n)_{n \in \N}$ tel que $d_{\O}(x_n, \g (x_n)) \to 0$ lorsque $n \to \infty$. L'existence d'une telle suite sur une variété compacte est impossible, puisque le rayon d'injectivité d'un point d'une variété finslérienne est une fonction continue et positive.

On considère $\Delta$ une droite passant par $p= p_{\g}$ si $\g$ est parabolique et par $p =p^{1}_{\g}$ si $\g$ est quasi-hyperbolique. On suppose de plus que $\Delta \cap \O \neq \varnothing$. On considère une suite de points $(x_n)_{n \in \N}$ de $\Delta$ qui tend vers $p$. On suppose que la droite $(x_n \g(x_n))$ rencontre $\partial \O$ en $a_n$ et $b_n$ et que $x_n$ est entre $a_n$ et $\g(x_n)$.

La quantité $d_{\O}(x_n, \g (x_n))$ est égale à la moitié du logarithme du birraport des droites $\Delta^+_n, \Delta, \g(\Delta)$ et $\Delta^-_n$, où $\Delta^+_n$ (resp. $\Delta^-_n$) est la droite passant par $p$ et le point $a_n$ (resp. $b_n$). Le point $p$ est un point $\Cc^1$ de la courbe $\partial \O$ et le point $\g(x_n)$ converge vers $p$, par suite les droites $\Delta_n^+$ et $\Delta^-_n$ convergent vers la tangente à $\partial \O$ en $p$. En particulier, la quantité $d_{\O}(x_n, \g (x_n))$ tend vers $0$.

Supposons à présent que le groupe $\G$ contient un élément hyperbolique avec $\Ax(\g) \subset \partial \O$. On va exhiber une distance $d'$ sur $\Quo$ qui induit la même topologie mais qui fait de $(\Quo,d')$ un espace métrique incomplet. Ce qui est impossible sur une variété compacte.

Le lemme \ref{rajout1} montre qu'il existe alors un ouvert proprement convexe $\O'$ préservé par $\G$ tel que $\Ax(\g) \subset \O'$.  On munit $\Quo$ de la distance induite par $d_{\O'}$. Les composantes connexes $\O' \setminus \O$ sont permutées par $\g$ et le stabilisateur d'une composante connexe est engendré par une racine de $\g$ à indice 2 près (lemme \ref{centra}). En particulier, on peut trouver une suite de points de $\Quo$ qui converge vers un point de la géodésique $\Ax(\g)/_{<\g>}$ dans $\O'/_{\G}$. L'espace métrique $(\Quo,d')$ n'est donc pas complet.
\end{proof}


\begin{lemm}\label{chiant}
Soit $\G$ un sous-groupe discret de $\s$ qui pr\'eserve un ouvert proprement convexe $\O$ de $\P$ qui n'est pas un triangle. On se donne un domaine de Dirichlet-Lee $D$ pour l'action de $\G$ sur $\O$. Soit $\g \in \G$, il existe un secteur $\F_0$ de $\g$ tel qu'aucune image $(\delta D)_{\delta \in \G}$ n'est incluse dans $\F_0$.
\end{lemm}

\begin{proof}
Commençons par supposer que l'élément $\g$ est elliptique ou hyperbolique avec $\Ax(\g) \subset \O$. Comme l'élément $\g$ est elliptique ou hyperbolique avec $\Ax(\g) \subset \O$, on peut trouver un secteur  $\F$ de $\g$ tel que le quotient $\F/_{<\g>}$ soit d'aire aussi petite que voulue. Par conséquent, il existe un secteur $\F_0$ de $\g$ qui ne contient aucune image $(\delta D)_{\delta \in \G}$.

Le lemme \ref{rajout2} montre que si le domaine fondamental $D$ est compact alors les éléments de $\G$ sont elliptiques ou hyperboliques avec $\Ax(\g) \subset \O$. On a donc montré le lemme dans le cas $D$ compact.

%

Supposons à présent que le domaine fondamental $D$ n'est pas compact, commençons par montrer ce lemme lorsque l'intersection $\overline{D} \cap \partial \O$ possède au moins deux composantes connexes. On peut supposer que $\g$ est parabolique ou  hyperbolique avec $\Ax(\g) \subset \partial \O$  ou quasi-hyperbolique.



Soit $\F$ un secteur de l'élément $\g$. L'intersection $\overline{\F} \cap \partial \O$ est convexe est un point ou un segment. Par suite, si l'intersection $\overline{D} \cap \partial \O$ possède au moins deux composantes connexes alors le secteur $\F$ ne peut contenir un domaine fondamental $D$.

On suppose \`a pr\'esent que $\overline{D} \cap \partial \O$ poss\`ede une seule composante connexe. Supposons aussi qu'il existe un secteur $\F$ de $\g$ qui contient le domaine fondamental $D$. Nous allons montrer qu'il existe un secteur $\F_0$ de $\g$ inclus dans $\F$ qui ne contient aucune image de $D$. Il faut traiter les cas $\g$ parabolique, $\g$ quasi-hyperbolique, ou $\g$ hyperbolique avec $\Ax(\g) \subset \partial \O$. On ne fait que le cas $\g$ parabolique. Les autres cas sont analogues.

Si un domaine fondamental convexe $D$ est inclus dans la diff\'erence $\F-\F_0$ de deux secteurs d'un \'el\'ement $\g$ alors la convexit\'e de $D$ entra\^ine que $D$ est compact. Donc si un secteur $\F$ de $\g$ contient un domaine fondamental $D$ alors tout secteur $\F_0$ de $\g$ inclus dans $\F$ rencontre le domaine fondamental $D$. Comme l'intersection de l'adh\'erence des secteurs de $\g$ est le point fixe $p_{\g}$ de l'\'el\'ement parabolique $\g$, on a $\overline{D} \cap \partial \O = \{ p_{\g} \}$.


Raisonnons par l'absurde. Si tout secteur $\F_0$ contient une image de $D$, alors ils existent une suite $(\F_n)_{n \in \N}$ de secteurs de $\g$ d\'ecroissante tel que l'intersection $\underset{n \in \N}{\bigcap} \F_n = \varnothing$ et une suite de domaines fondamentaux $(D_n)_{n \in \N}$ avec $D_n \subset \F_n$ et $D_n \not\subset \F_{n+1}$. Les domaines $D_n$ (qui sont tous diff\'erents) v\'erifient tous que $\overline{D_n} \cap \partial \O= \{ p_{\g} \}$. Le domaine fondamental $D_{n+1}$ n'est pas une image du domaine $D_n$ par $<\g>$. 

On peut supposer quitte \`a appliquer une puissance de $\g$ que le domaine $D_{n+1}$ est "entre" le domaine $D_n$ et le domaine $\g D_n$. En particulier, les domaines fondamentaux $(D_n)_{n \in \N}$ sont "entre" les domaines $D_0$ et $\g D_0$. Par cons\'equent, la famille des droites donn\'es par les c\^ot\'es des domaines $(D_n)_{n \in \N}$ n'est pas localement fini, ce qui contredit le lemme \ref{finitude}.
\end{proof}

\begin{proof}[D\'emonstration de la proposition \ref{inftyfini}]
Commençons par faire le cas où l'ouvert $\O$ n'est pas un triangle.

Il existe alors d'apr\`es le lemme \ref{chiant} un secteur $\F_0 \subset \F$ qui ne contient aucun domaine fondamental. Le groupe $<\g>$ agit cocompactement sur $V= \overline{(\F-\F_0)} \cap \O$. On note $E$ un domaine fondamental compact pour cette action.

Si un \'el\'ement $\delta \in \G$ est tel que $\delta^{-1} D \cap \F  \neq \varnothing$ alors il v\'erifie
$\delta^{-1} D \cap V \neq \varnothing$, car sinon on aurait $\delta^{-1} D \subset \F_0$ et ceci contredit notre hypoth\`ese sur $\F_0$.

Il existe $n \in \Z$ tel que $\g^n\delta^{-1} D \cap E \neq \varnothing$. Or, $D$ est un domaine fondamental localement fini donc l'ensemble $\{ \delta \in \G \, | \, \delta^{-1} D \cap E \neq \varnothing \}$ est fini. Notons $\{ \delta_1, ... , \delta_r\}$ ces \'el\'ements.

On vient donc de montrer que si un \'el\'ement $\delta \in \G$ est tel que $D \cap \delta \F  \neq \varnothing$ alors il existe $i=1...r$ et $n \in \Z$ tel que $\delta = \delta_i \g^n$. Or, le secteur $\F$ est $\g$-invariant par cons\'equent $D$ rencontre un nombre fini d'image $(\delta \F)_{\delta \in \G}$.

Enfin, si l'ouvert $\O$ est un triangle alors le groupe $\Aut(\O)$ est virtuellement abélien et par conséquent un secteur à un nombre fini d'images distincts sous $\Aut(\O)$.
\end{proof}

On rappelle la d\'efinition d'\'el\'ement primitif.

\begin{defi}
Soient $\G$ un groupe et $\g \in \G$, on dit que $\g$ est \emph{primitif} lorsque l'existence d'un \'el\'ement $\delta \in \G$ et d'un entier $n \in \N$ tel que $\g = \delta^n$ entraine $\delta=\g$ et $n=1$ ou  $\delta=\g^{-1}$ et $n=-1$.
\end{defi}

Les lemmes suivants seront cruciaux dans la partie suivante pour montrer le lemme \ref{injgeo}.

\begin{lemm}\label{lemmeinj1}
Soit $\G$ un sous-groupe discret de $\s$ qui pr\'eserve un ouvert proprement convexe $\O$ de $\P$. Pour tout \'el\'ement primitif $\g$ de $\G$ et pour tout secteur $\F$ de $\g$, il existe un nombre fini d'\'el\'ements $h_1,...,h_N$ de $\G$ tel que $\{\delta \in \G \, | \, \delta \F \cap \F \neq \varnothing \} \subset \{ \delta \in \G \, | \, \exists n,p \in\Z, \exists i = 1,...,N ,\, \delta = \g^n h_i \g^p \}$.
\end{lemm}

\begin{proof}
On se donne un domaine de Dirichlet-Lee $D$ pour l'action de $\G$ sur $\O$ (th\'eor\`eme \ref{Lee}) qui rencontre $\F$. La proposition \ref{inftyfini}  montre que $D$ rencontre un nombre fini d'image de $\F$. On peut les \'ecrire $g_1 \F, ..., g_r \F$, o\`u $g_i \in \G$ pour $i=1...r$ et $g_1 = Id$.

Comme $\g$ est primitif, on a $\F \subset \underset{n \in \Z, \, i = 1..r}{\bigcup}\g^n g_i^{-1} D$. Par cons\'equent, si $\delta \F \cap \F \neq \varnothing$ alors il existe $n_0 \in \Z$ et un $i_0=1...r$ tel que $g_{i_0} \g^{n_0} \delta \F \cap D \neq \varnothing$, par cons\'equent il existe $n_1 \in \Z$ et $i_1 = 1...r$ tel que $g_{i_0} \g^{n_0} \delta = g_{i_1} \g^{n_1}$. Autrement dit, $\delta = \g^{-n_0} g_{i_0}^{-1} g_{i_1} \g^{n_1}$, c'est ce qu'il fallait montrer.
\end{proof}

\begin{lemm}\label{lemmeinj2}
Soient $\G$ un sous-groupe discret de $\s$ qui pr\'eserve un ouvert proprement convexe $\O$ de $\P$, un \'el\'ement $\g \in \G$ et $\F$ un secteur de $\g$, pour tout \'el\'ement $\delta$ de $\G$ on a l'alternative suivante:
\begin{itemize}
\item $\delta \F \cap \F$ est compacte ou bien

\item $\delta \F = \F$.
\end{itemize}
\end{lemm}

\begin{proof}
Une \'etude exhaustive en distinguant les diff\'erentes dynamiques de l'\'el\'ement $\g$ va rendre ce lemme clair.

Si $\g$ est elliptique alors tout secteur de $\g$ est compact. Il n'y a donc rien à montrer.

Si $\g$ est hyperbolique avec $\Ax(\g) \subset \O$ alors l'intersection de l'adhérence de tout secteur $\F$ de $\g$ avec le bord de $\O$ est égale à deux points, le point attractif et le point répulsif de $\g$. Il vient que pour tout $\delta \in \G$, ou bien
\begin{itemize}
\item $\delta \g \delta^{-1}$ possède les mêmes points fixes que $\g$, alors $\delta \g \delta^{-1}$ est une puissance de $\g$ (proposition \ref{centra}) et $\delta \F = \F$.
\item les points fixes de $\delta \g \delta^{-1}$ et les points fixes de $\g$ forment des ensembles disjoints, et $\delta \F \cap \F$ est compacte.
\end{itemize}

Si $\g$ est parabolique alors l'intersection de l'adhérence de tout secteur $\F$ de $\g$ avec le bord de $\O$ est égale à un point, le point attractif de $\g$. Par conséquent, ou bien 
\begin{itemize}
\item $\delta \g \delta^{-1}$ possède le même point fixe que $\g$, alors $\delta \g \delta^{-1}$ est une puissance de $\g$ (proposition \ref{centra}) et $\delta \F = \F$.
\item le point fixe de $\delta \g \delta^{-1}$ et le point fixe de $\g$ sont différents, et $\delta \F \cap \F$ est compacte.
\end{itemize}

De même, si $\g$ est quasi-hyperbolique ou hyperbolique avec $\Ax(\g) \subset \partial \O$. La même méthode conclut.
\end{proof}

\section{Surface projective convexe d'aire finie}\label{def}

Tout au long de ce texte, une surface est une vari\'et\'e connexe orientable de dimension 2, avec \'eventuellement des bords. Si $S$ est une surface on notera son bord $\partial S$. On note $\mathbb{E}$ un demi-espace affine ferm\'e de $\P$.

\subsection{Structure projective}
\begin{defi}
Soit $S$ une surface, une \emph{structure projective r\'eelle \`a bord g\'eod\'esique} est la donn\'ee d'un atlas maximal
$\varphi_{\mathcal{U}}:\mathcal{U} \rightarrow \mathbb{E}$ sur $S$ tel que:
\begin{itemize} 
\item les fonctions de transitions $\varphi_{\mathcal{U}} \circ \varphi_{\mathcal{V}}^{-1}$ sont des \'el\'ements de $\s$, pour tous ouverts $\mathcal{U}$ et $\mathcal{V}$ de l'atlas de $S$ tel que $\mathcal{U} \cap \mathcal{V} \neq \varnothing$.
\item Pour tout ouvert $\mathcal{U}$ de l'atlas tel que $\mathcal{U} \cap \partial S \neq \varnothing$ et pour toute
composante connexe $B$ de $\mathcal{U} \cap \partial S$, $\varphi_{\mathcal{U}}(B)$ est inclus dans une droite projective de $\P$.
\end{itemize}
\end{defi}

\begin{defi}
Un \emph{isomorphisme} entre deux surfaces munies de structures projectives \`a bord g\'eod\'esique est un diff\'eomorphisme qui, lu dans les cartes, est
donn\'e par des \'el\'ements de $\s$.
\end{defi}

\begin{defi}
Soit $S$ une surface, une \emph{structure projective \`a bord g\'eod\'esique marqu\'ee sur
$S$} est la donn\'ee d'un diff\'eomorphisme $\varphi:S \rightarrow S'$ o\`u
$S'$ est une surface projective \`a bord g\'eod\'esique.

Deux structures projectives \`a bord g\'eod\'esique marqu\'ee sur $S$, $\varphi_1:S
\rightarrow S_1$ et $\varphi_2:S \rightarrow S_2$ sont dites
\emph{isotopiques} lorsqu'il existe un isomorphisme $h:S_1
\rightarrow S_2$ tel que $\varphi_2^{-1} \circ h \circ \varphi_1 : S
\rightarrow S$ est un diff\'eomorphisme isotope \`a l'identit\'e. On
note $\PP(S)$ \emph{l'ensemble des structures projectives \`a bord g\'eod\'esique marqu\'ees
sur $S$ modulo isotopie}.
\end{defi}
\`{A} tout \'el\'ement de $\PP(S)$, on peut associer deux objets:
\begin{itemize}
\item  Un diff\'eomorphisme local $\dev:\widetilde{S}\rightarrow \P$ appel\'ee
\emph{d\'eveloppante}, o\`u $\widetilde{S}$ est le rev\^etement
universel de $S$.

\item Une repr\'esentation $\Hol:\pi_1(S) \rightarrow \s$ appel\'ee
\emph{holonomie}.
\end{itemize}
De plus, la d\'eveloppante est $\pi_1(S)$-\'equivariante, c'est \`a dire que $\forall x \in \widetilde{S}$, $\forall \g \in
\pi_1(S)$ on a $\dev(\g \, x) = \Hol(\g) \dev(x)$. Enfin, le couple $(\dev, \Hol)$ est unique au sens o\`u
si $(\dev',\Hol')$ est une autre telle paire alors il existe un $g
\in \s$ tel que $\dev' = g \circ \dev$ et $\Hol' = g \circ \Hol \circ
g^{-1}$.

On pourra consulter l'article [Gold1] pour avoir plus de d\'etails
sur le couple d\'eveloppante et holonomie.

\begin{rem}
\`{A} partir de maintenant toutes les structures projectives seront
implicitement suppos\'ees marqu\'ees et \`a bord g\'eod\'esique.
\end{rem}

\subsection{Structure projective proprement convexe}

\begin{defi}\label{def}
Soit $S$ une surface, une structure projective est dite 
\emph{proprement convexe} sur $S$ lorsque la d\'eveloppante est un
diff\'eomorphisme sur une partie proprement convexe de $\P$. On note
$\beta(S)$ l'ensemble des structures projectives proprement convexes sur $S$ modulo isotopie.
\end{defi}

\subsection{Structure projective proprement convexe d'aire finie}

Soit $S$ une surface projective proprement convexe, l'application
d\'eveloppante  permet d'identifier le rev\^etement universel $\widetilde{S}$ de $S$ \`a une partie
$\C$ proprement convexe de $\P$. On notera $\pi:\C \rightarrow S$ le rev\^etement universel de $S$. On a construit au paragraphe
\ref{base} sur l'int\'erieur $\O = \overset{\circ}{\C}$ de $\C$ une
distance $d_{\O}$ et une mesure $\mu_{\O}$ qui sont invariantes
sous l'action du groupe fondamental $\pi_1(S)$ de $S$ sur $\O$ . Par cons\'equent, il existe
une unique distance $d_S$ et une unique mesure $\mu_S$ sur
$\overset{\circ}{S}$ l'int\'erieur de $S$ telles que:

\begin{itemize}
\item Pour tout $x,y \in \overset{\circ}{S}$, $d_S(x,y) = \underset{(\widetilde{x},\widetilde{y})\in \mathcal{E}}{\inf} d_{\O}(\widetilde{x},\widetilde{y})$, o\`u $\mathcal{E} = \{ (\widetilde{x},\widetilde{y}) \in \O^2 \, |\, \pi(\widetilde{x}) = x \textrm{ et } \pi(\widetilde{y}) = y \}$

\item $\forall \mathcal{A}$ bor\'elien de $\O$, si $\pi:\O
\rightarrow \overset{\circ}{S}$ restreinte \`a $\mathcal{A}$ est
injective alors $\mu_S(\pi(\mathcal{A}))=\mu_{\O}(\mathcal{A})$.
\end{itemize}

\begin{defi}
Soit $S$ une surface, on dit qu'une structure projective
proprement convexe sur $S$ est de \emph{volume fini} lorsque pour
tout ferm\'e $F$ de l'int\'erieur $\overset{\circ}{S}$ de $S$, on a
$\mu_S(F) < \infty$. On note $\beta_{f}(S)$ l'ensemble des
structures projectives marqu\'ees proprement convexes de volume fini sur $S$
modulo isotopie.
\end{defi}

\subsection{La th\'eorie des bouts d'un espace topologique}

On rappelle quelques d\'efinitions de la th\'eorie des bouts d'un espace topologique.

\begin{defi}
Soit $X$ un espace topologique localement compact, une \emph{base de voisinages d'un bout de $X$} est une suite d\'ecroissante d'ouverts connexes de $X$ qui sort de tout compact. Deux bases de voisinages d'un bout de $X$, $(U_i)_{i \in \N}$ et $(V_i)_{i \in \N}$ sont dites \emph{\'equivalentes} si pour tout $n,m \in \N$, il existe $N,M \in \N$ tel que $U_N \subset V_n$ et $V_M \subset U_m$. Les classes d'\'equivalence de base de voisinages de bout de $X$ forment un ensemble appel\'e \emph{l'espace des bouts de $X$} et dont les \'el\'ements sont \emph{les bouts de $X$}.
\end{defi}

\begin{rem}
Cet ensemble poss\`ede une topologie naturelle. On la construit de la façon suivante. Pour tout compact $K$ de $X$ on d\'efinit $\mathcal{U}_K$ l'ensemble des bouts de $X$ qui sont ultimement inclus dans $X-K$. Les $\mathcal{U}_K$ forment une base de la topologie de l'espace des bouts de $X$. On peut montrer que l'espace des bouts est un espace compact totalement discontinu.
\end{rem}

\subsection{Les lacets d'holonomie parabolique ou quasi-hyperbolique sont \'el\'ementaires}

Le lemme suivant est imm\'ediat.

\begin{lemm}
Soit $S$ une surface projective proprement convexe dont le groupe
fondamental n'est pas virtuellement ab\'elien alors aucun lacet
trac\'e sur $S$ ne poss\`ede une holonomie planaire.
\end{lemm}

\begin{proof}
La dynamique des \'el\'ements planaires  montre que si l'holonomie
d'un \'el\'ement de $\pi_1(S)$ est planaire alors $\O$ est un triangle
(proposition \ref{planaire}). Par cons\'equent, le groupe $\Aut(\O)$ est virtuellement isomorphe \`a $\R^2$ et donc virtuellement ab\'elien. Il vient que $\pi_1(S)$ est virtuellement ab\'elien, ce qui contredit l'hypoth\`ese sur la
topologie de $S$.
\end{proof}

\begin{rem}
Si l'ouvert proprement convexe $\O$ est un triangle alors le groupe $\Aut(\O)$ est isomorphe à $\Z^2$. Par conséquent, si $S$ est une surface projective convexe de caractéristique d'Euler strictement négative alors son revêtement universel $\O$ n'est pas un triangle.
\end{rem}

\begin{defi}\label{elem}
Soit $S$ une surface, on dit qu'un lacet trac\'e sur $S$, $c:\S^1
\rightarrow S$ est \emph{simple} s'il est injectif. On dit qu'un
lacet simple $c$ trac\'e sur $S$ est \emph{\'el\'ementaire} si $S-c$
poss\`ede deux composantes connexes et l'une d'elles est un
cylindre. Lorsque $S$ n'est pas un cylindre, on appellera l'adh\'erence de la composante hom\'eomorphe \`a un cylindre
\emph{la composante \'el\'ementaire associ\'ee \`a $c$}.
\end{defi}

\begin{rem}
Soient $S$ une surface et $c$ un lacet \'el\'ementaire trac\'e sur $S$
alors on a l'alternative suivante:
\begin{itemize}
\item La composante \'el\'ementaire associ\'ee \`a $c$ est un cylindre
avec un bord. On dira alors que $c$ \emph{fait le tour d'un bout}.

\item La composante \'el\'ementaire associ\'ee \`a $c$ est un cylindre
avec deux bords. Dans ce cas $c$ est librement homotope \`a une
composante connexe du bord de $S$.
\end{itemize}
\end{rem}

Les propositions suivantes sont tr\`es classiques en g\'eom\'etrie hyperbolique. On ne montre que la premi\`ere.

\begin{prop}\label{lacethomo}
Soit $S$ une surface, on munit $S$ d'une structure projective proprement convexe. On note $\pi:\C\rightarrow S$ le rev\^etement universel de $S$ donn\'e par la d\'eveloppante de $S$, o\`u $\C$ est une partie proprement convexe de $\P$. L'int\'erieur de $\C$ sera not\'e $\O$. Alors, tout lacet trac\'e sur $S$ dont l'holonomie est hyperbolique avec $\Ax(\g) \subset \O$ est librement homotope \`a une g\'eod\'esique.
\end{prop}

\begin{proof}
On note $\G$ l'image du groupe fondamental de $S$ dans $\s$. On choisit $\g \in \G$ qui repr\'esente $\Hol(c)$ et on note $\widetilde{c}$ le relev\'e correspondant de $c$. On a les convergences suivantes $\underset{t \rightarrow +\infty}{\lim} \widetilde{c}(t) = p^+_{\g}$ et $\underset{t \rightarrow -\infty}{\lim} \widetilde{c}(t) = p^-_{\g}$. Le chemin $\widetilde{c}$ est donc homotope via une homotopie $\g$-\'equivariante \`a une et une seule g\'eod\'esique de $\O$: $\Ax(\g)$. La projection de cette homotopie sur la surface $S$ donne le r\'esultat souhait\'e.
\end{proof}

La m\^eme d\'emonstration que dans le cas hyperbolique donne la proposition suivante:

\begin{prop}\label{lacethomo2}
Soit $S$ une surface, on munit $S$ d'une structure projective proprement convexe. On note $\pi:\C\rightarrow S$ le rev\^etement universel de $S$ donn\'e par la d\'eveloppante de $S$, o\`u $\C$ est une partie proprement convexe de $\P$. L'int\'erieur de $\C$ sera not\'e $\O$. Soient $c_1$ et $c_2$ sont deux lacets simples trac\'es sur $S$ dont l'holonomie est hyperbolique avec $\Ax(\g_1),\, \Ax(\g_2)  \subset \O$, on note $\lambda_1$ (resp. $\lambda_2$) l'unique g\'eod\'esique homotope \`a $c_1$ (resp. $c_2$). Si $c_1$ est simple alors $\lambda_1$ est simple. Si les lacets $c_1$ et $c_2$ ne s'intersectent pas alors les g\'eod\'esiques $\lambda_1$ et $\lambda_2$ ne s'intersectent pas.
\end{prop}

\begin{lemm}\label{injgeo}
Soit $S$ une surface de caractéristique d'Euler strictement négative, on munit $S$ d'une structure projective proprement convexe. Soit $c$ un lacet simple trac\'e sur $S$, il existe un secteur $\F$ de $\Hol(c)$ tel que l'application naturelle de $\F/_{<\Hol(c)>} \rightarrow S$ est une injection.
\end{lemm}

\begin{proof}
On pose $\g= \Hol(c)$. Le lemme \ref{lemmeinj1}  montre qu'il existe des \'el\'ements $g_j$ pour $j=1...N$ tel que si $\delta \F \cap \F \neq \varnothing$ alors il existe $n\in \Z$ et $j_0=1...N$ tel que $\delta \F = \g^n g_{j_0} \F$.

De plus, le lemme \ref{lemmeinj2}  montre que si $g_j \F \cap \F \neq \varnothing$ ($j=1...N$) alors cette intersection est ou bien $\F$ ou une partie compacte de $\O$. Il faut \`a pr\'esent distinguer les deux cas suivants:
\begin{itemize}
\item Si $\g$ est hyperbolique avec $\Ax(\g) \subset \partial \O$ ou quasi-hyperbolique ou parabolique alors comme les $g_j$ sont en nombre fini, on peut trouver un secteur $\F_0$ inclus dans $\F$ tel que $g_j \F_0 \cap \F_0 \neq \varnothing$ si et seulement si $g_j \F_0 = \F_0$.

\item Si $\g$ est hyperbolique avec $\Ax(\g) \subset \O$ alors on peut supposer d'apr\`es le lemme \ref{lacethomo} que $c$ est la g\'eod\'esique simple donn\'e par l'axe de $\g$. Le secteur $\F$ est un voisinage de l'axe de $\g$. De plus, comme $c$ est une g\'eod\'esique simple le lemme \ref{lacethomo2} montre que les relev\'es de $c$ ne s'intersectent pas. Par cons\'equent, comme les $g_j$ sont en nombre fini, on peut trouver un secteur $\F_0$ inclus dans $\F$ tel que $g_j \F_0 \cap \F_0 \neq \varnothing$ si et seulement si $g_j \F_0 = \F_0$.
\end{itemize}
Dans tous les cas, on obtient que pour tout $\delta \in \G$, $\delta \F_0 \cap \F_0 \neq \varnothing$ si et seulement si $\delta \F_0 = \F_0$. Mais, l'ouvert $\O$ n'est pas un triangle par conséquent $\delta \F_0 = \F_0$ si et seulement si $\delta \in <\g>$.
\end{proof}

\begin{prop}\label{fonda1}
Soit $S$ une surface  de caractéristique d'Euler strictement négative, on munit $S$ d'une structure projective proprement convexe. Soit $c$ un lacet simple trac\'e sur $S$, on note $\pi : \C
\rightarrow S$ le rev\^etement universel de $S$ donn\'e par la
d\'eveloppante de $S$. On note $\overset{\circ}{S}$ l'int\'erieur de $S$.
\begin{itemize}
\item Si l'holonomie de $c$ est parabolique alors $c$ fait le tour d'un bout $B$ et l'image de tout secteur de $\Hol(c)$ dans $S$ contient un voisinage du bout $B$.

\item Si l'holonomie de $c$ est quasi-hyperbolique alors $c$ est \'el\'ementaire et l'image de tout secteur de $\Hol(c)$ dans $S$ contient un voisinage du bout correspondant dans la surface $\overset{\circ}{S}$.

\item Si l'holonomie de $c$ est hyperbolique et $\Ax(Hol(c)) \subset \partial \C$ alors $c$ est \'el\'ementaire et l'image de tout secteur de $\Hol(c)$ dans $S$ contient un voisinage du bout correspondant dans la surface $\overset{\circ}{S}$.

\item Si l'holonomie de $c$ est hyperbolique, $\Ax(Hol(c)) \subset \overset{\circ}{\C}$ et $p^0_{Hol(c)} \in \partial \C$ alors $c$ est \'el\'ementaire et l'image dans $S$ de la r\'eunion de n'importe quel secteur de $\Hol(c)$ et de l'unique triangle inclus dans $\C$ d\'efini par les points fixes de $\Hol(c)$ contient un voisinage du bout correspondant dans la surface $\overset{\circ}{S}$.
\end{itemize}
En particulier, si $c$ n'est pas \'el\'ementaire alors l'holonomie de $c$ est
hyperbolique, $\Ax(Hol(c)) \subset \overset{\circ}{\C}$ et $p^0_{Hol(c)} \notin \partial \C$.
\end{prop}

\begin{proof}
On note $\O$ l'int\'erieur de $\C$. On consid\`ere $\g = Hol(c)$ et on note $\F$ un secteur de $\g$ tel que l'application naturelle de $\F/_{<\g>} \rightarrow S$ est une injection. L'existence d'un tel secteur est assur\'e par le lemme \ref{injgeo}.

Commençons par le cas $\g$ est parabolique. Construisons un domaine fondamental pour l'action de $\g$ sur $\F$. On consid\`ere une droite $L$ passant par $p_{\g}$. L'adh\'erence $D'$ dans $\O$ de n'importe quelle composante connexe de $\F - \underset{n \in \Z}{\bigcup} \g^n L$ est un domaine fondamental pour l'action de $\g$ sur $\F$. Le domaine fondamental $D'$ est l'intersection d'un triangle ferm\'e de $\P$ et de $\F$. De plus son adh\'erence dans $\P$ contient le point $p_{\g}$. L'image de $\F/_{<\g>}$ dans $S$ est donc un cylindre qui contient un voisinage d'un bout $B$ de la surface $S$ et le lacet $c$ fait le tour du bout $B$.

Ensuite, on peut traiter en m\^eme temps les cas, ou $\g$ est quasi-hyperbolique, ou hyperbolique avec $\Ax(\g) \subset \partial \C$. On proc\`ede de la m\^eme façon. Construisons un domaine fondamental pour l'action de $\g$ sur $\F$. On note $p_{\g}^+$ le point attractif de $\g$ et $p_{\g}^-$ le point r\'epulsif de $\g$. L'axe de $\g$ est le segment d'extr\'emit\'e $p_{\g}^+$ et $p_{\g}^-$ inclus dans $\partial \O$. On consid\`ere une droite $L$ qui intersecte $\Ax(\g)$ sur son int\'erieur. L'action de $\g$ sur $\partial \O$ v\'erifie que pour tout $x \in \partial \O$ diff\'erent de $p_{\g}^+,p_{\g}^-$, le point $\g x$ appartient \`a la composante connexe $V$ de $\partial \O -\{ p^+_{\g},p^-_{\g} \}$ qui contient $x$. Et, le point $\g x$ appartient \`a la composante connexe de $V-\{ x \}$ qui contient $p_{\g}^+$ dans son adh\'erence. Par cons\'equent, les droites $L$ et $\g L$ ne s'intersectent pas dans $\O$. Il vient donc que
l'adh\'erence $D'$ dans $\O$ de n'importe quelle composante connexe de $\F - \underset{n \in \Z}{\bigcup} \g^n L$ est un domaine fondamental pour l'action de $\g$ sur $\F$. Le domaine fondamental $D'$ contient un sous-segment non trivial de $\Ax(\g)$. L'image de $\F/_{<\g>}$ dans $S$ est donc un cylindre qui contient un voisinage d'un bout $B$ de la surface $\overset{\circ}{S}$ et le lacet $c$ est fait le tour de ce bout.

Enfin, il faut traiter le cas o\`u $\g$ est hyperbolique avec $\Ax(\g) \subset \overset{\circ}{\C}$ et $p^0_{\g} \in \partial \C$. La difficult\'e de ce cas vient du fait que cette fois-ci les secteurs de $\g$ "ne vont pas jusqu'\`a l'infini". On consid\`ere la partie $\widetilde{\F}$ de $\O$ obtenue en ajoutant \`a $\F$ le triangle ouvert inclus dans $\O$ d\'efini par les points $p^+_{\g}, p^-_{\g}, p_{\g}^0$. Pour appliquer le m\^eme raisonnement que pr\'ec\'edemment il faut montrer que l'application naturelle de $\widetilde{\F}/_{<\g>}$ vers $S$ est injective. Comme $\F$ est un voisinage de $\Ax(\g)$ et que pour tout \'el\'ement $\delta \in \G$ si $\delta \F \cap \F \neq \varnothing$ alors $\delta \in <\g>$, l'ensemble $\widetilde{\F}$ v\'erifie aussi que pour tout \'el\'ement $\delta \in \G$ si $\delta  \widetilde{\F} \cap \widetilde{\F} \neq \varnothing$ alors $\delta \in <\g>$, car $\delta  \widetilde{\F} \cap \widetilde{\F} \neq \varnothing$ entraîne $\Ax(\g) \cap \delta \Ax(\g) \neq \varnothing$. Il  reste \`a construire un domaine fondamental pour l'action de $\g$ sur $\widetilde{\F}$. Pour cela on consid\`ere une droite $L$ passant par $p^0_{\g}$. L'adh\'erence $D'$ dans $\O$ de n'importe quelle composante connexe de $\widetilde{\F} - \underset{n \in \Z}{\bigcup} \g^n L$ est un domaine fondamental pour l'action de $\g$ sur $\F$. Le point $p^0_{\g}$ est adh\'erent au domaine fondamental $D'$. L'image de $\F/_{<\g>}$ dans $S$ est donc un cylindre qui contient un voisinage d'un bout $B$ de la surface $\overset{\circ}{S}$ et le lacet $c$ fait le tour de ce bout.
\end{proof}

\subsection{Le groupe fondamental d'une surface de volume fini est de type fini}

\subsubsection{Un peu de topologie des surfaces}

On note $P$ la surface \`a bord obtenue en retirant 3 disques ouverts disjoints \`a la sph\`ere euclidienne $\mathbb{S}^2$.

\begin{defi}
Soit $S$ une surface, un \emph{pantalon de $S$} est une sous-surface \`a bord de $S$ hom\'eomorphe \`a $P$. Un \emph{pantalon non \'el\'ementaire de $S$} est un pantalon de $S$ dont le bord d\'efini 3 lacets non \'el\'ementaires.
\end{defi}

Richards a classifi\'e les surfaces en construisant des invariants \`a l'aide de l'espace des bouts de celles-ci (\cite{Rich}). Il r\'esulte de cette classification la proposition suivante.

\begin{prop}\label{toposurf}
Soit $S$ une surface, on a l'alternative suivante:
\begin{itemize}
\item L'espace des bouts de $S$ est infini.

\item L'espace des bouts de $S$ est fini mais $S$ contient une infinit\'e de pantalons non \'el\'ementaires deux \`a deux disjoints.

\item La surface $S$ est de type fini.
\end{itemize}
\end{prop}

\begin{defi}
Soit $S$ une surface projective proprement convexe, un \emph{pantalon de $S$ \`a bord g\'eod\'esique} est un pantalon $P$ de $S$ tel que le bord de $P$ est d\'efini par trois g\'eod\'esiques non \'el\'ementaires de $S$.
\end{defi}

Les lemmes \ref{lacethomo}, \ref{lacethomo2} et la proposition \ref{fonda1} donne la proposition suivante:

\begin{prop}\label{homopant}
Soit $S$ une surface projective proprement convexe, tout pantalon non \'el\'ementaire de $S$ est homotope \`a un unique pantalon \`a bord g\'eod\'esique. Si $P_1$ et $P_2$ sont deux pantalons non \'el\'ementaires de $S$ disjoints alors les uniques pantalons homotopes \`a $P_1$ et $P_2$ sont aussi disjoints.
\end{prop}

\subsubsection{Minoration de l'aire d'un pantalon projectif proprement convexe}

Pour montrer que l'aire de tout pantalon \`a bord g\'eod\'esique est minor\'ee par une constante universelle, nous allons chercher des triangles id\'eaux, pour cela on utilise des "chemins en spirales".

\begin{prop}\label{minaire}
Il existe une constante universelle $K_{\P}$ tel que pour toute surface $S$ projective proprement convexe et tout pantalon non \'el\'ementaire $P$ \`a bord g\'eod\'esique inclus dans $S$, on ait: $\mu_{S}(P) \geqslant K_{\P}$.
\end{prop}

\begin{proof}
Nous allons montrer que tout pantalon est la r\'eunion de deux triangles id\'eaux. Pour cela on utilise une construction de Goldman (\cite{Gold1}). On part de la figure \ref{genre}.

\begin{figure}[!h]
\begin{center}
\includegraphics[trim=0cm 8cm 0cm 0cm, clip=true, width=10cm]{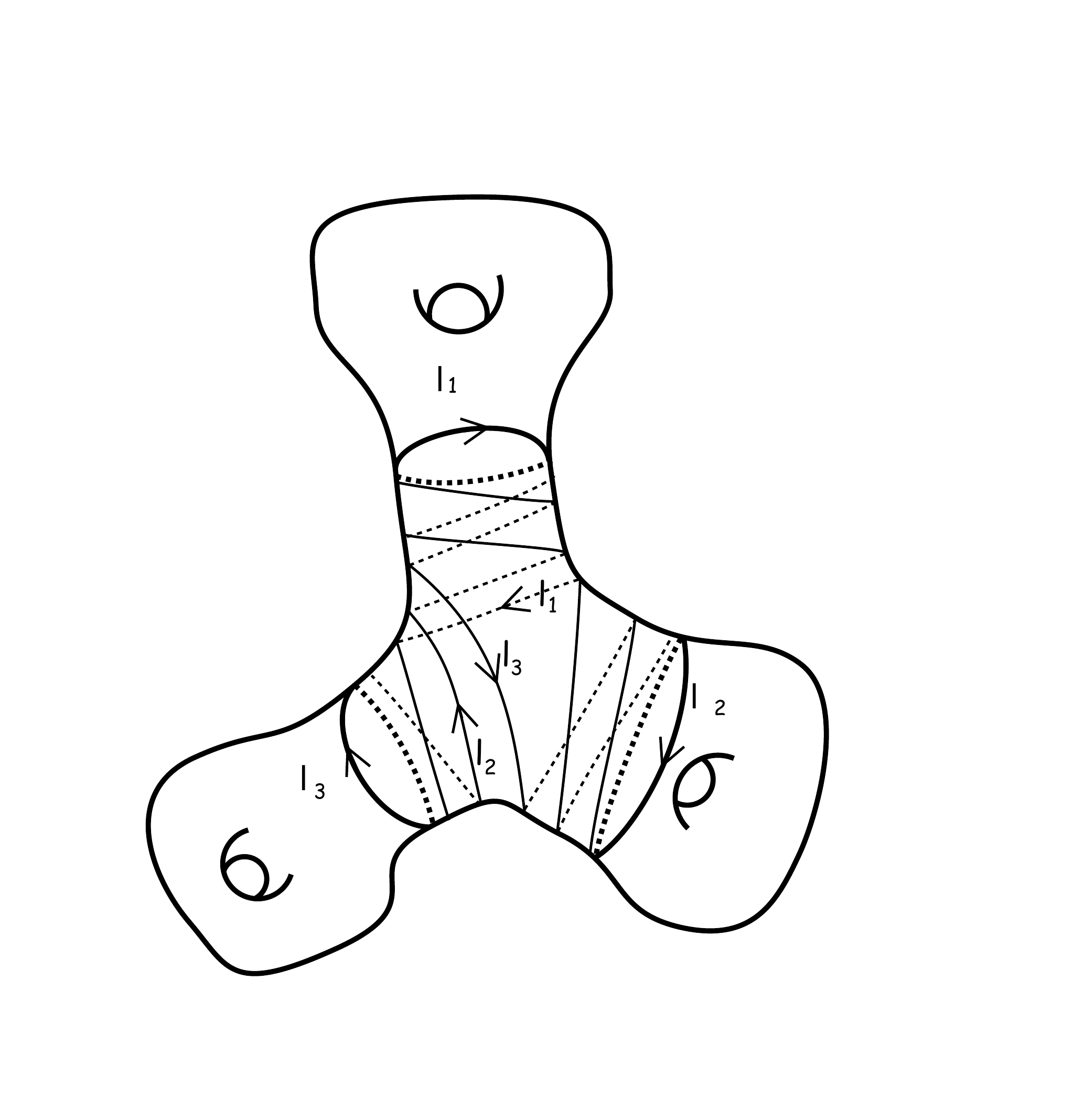}
\caption{D\'emonstration de la proposition \ref{minaire}}\label{genre}
\end{center}
\end{figure}


Commençons par donner une définition précise de la figure \ref{genre} d'un point de vue topologique via la géométrie hyperbolique. On munit la surface $S$ d'une structure hyperbolique, c'est possible car $S$ est de caractéristique d'Euler strictement négative. On note $\lambda^{\mathbb{H}}_1$, $\lambda^{\mathbb{H}}_2$ et $\lambda^{\mathbb{H}}_3$ les géodésiques données par le bord de $P$, et orienté comme sur la figure \ref{genre}. Le chemin $l_i$ est l'unique géodésique qui s'accumule sur les géodésiques $\lambda^{\mathbb{H}}_{i+1}$ et $\lambda^{\mathbb{H}}_{i+2}$ (Les indices sont calculés modulo 3).


Nous allons montrer que dans le cas projectif la situation est analogue. On note $\C$ la partie proprement convexe donné par la développante de $S$ et $\pi:\C \rightarrow S$ le revêtement universel associé. On note $\lambda_1$, $\lambda_2$ et $\lambda_3$ les géodésiques données par le bord de $P$, et orienté comme sur la figure \ref{genre}.

Les chemins $l_1$, $l_2$ et $l_3$ se relèvent en des chemins simples $\widetilde{l_1}$, $\widetilde{l_2}$ et $\widetilde{l_3}$ tracé sur le revêtement universel $\widetilde{P}$ de $P$. Chacun des chemins $\widetilde{l_i}$ vérifient que $\widetilde{P}-\widetilde{l_i}$ possède deux composantes connexes. Ces chemins et leurs images par le groupe fondamental $\G$ de $P$ définissent une triangulation de $\widetilde{P}$ (triangulation de Faray). Deux triangles fermés adjacents de cette triangulation définissent un domaine fondamental pour l'action de $\G$ sur $\widetilde{P}$.

Nous allons montrer que chacun de ces chemins peuvent être supposés géodésiques. On note $\g_1$ (resp. $\g_2$ resp. $\g_3$) les représentants de l'holonomie des lacets $\lambda_1$ (resp. $\lambda_2$ resp. $\lambda_3$) donnés par le choix du point base $x_0$ sur $P$. Ils vérifient la relation $\g_3\g_2\g_1=1$.

On note $\widetilde{\lambda}_1$ (resp. $\widetilde{\lambda}_2$, resp. $\widetilde{\lambda}_3$) un relevé de $\lambda_1$, (resp. $\lambda_2$ resp. $\lambda_3$). On peut supposer que $\widetilde{\lambda}_1$, $\widetilde{\lambda}_2$ et $\widetilde{\lambda}_3$ bordent la même composante connexe $\U$ de $\pi^{-1}(P)$. La partie $\U$ de $\C$ est convexe puisque c'est l'intersection de l'intérieur $\O$ de $\C$ et d'une infinité de demi-espaces définis par les relevés de $\lambda_1$, $\lambda_2$ et $\lambda_3$ qui bordent $\U$. Ainsi, le convexe $\U$ est le revêtement universel du pantalon $P$.

Comme $P$ est un pantalon non élémentaire, la proposition \ref{fonda1} montre que $\g_1$, $\g_2$ et $\g_3$ sont hyperboliques et leurs axes principaux sont inclus dans $\O$. Notre choix d'orientation fait que les points $p_{\g_1}^-,p_{\g_1}^+,p_{\g_2}^-,p_{\g_2}^+,$  $p_{\g_3}^-,p_{\g_3}^+ \in \partial \O$ sont sur $\partial \O$ dans cet ordre. Les axes de $\g_1$, $\g_2$ et $\g_3$ sont inclus dans le bord de $\U$. Tout relevé du chemin $l_3$ converge en $+\infty$ vers le point $p^-_{\g_1}$ et en $-\infty$ vers le point $p^-_{\g_2}$. On peut donc supposer que le chemin $\widetilde{l_3}$ est le segment ouvert d'extrémité $p^-_{\g_1}$ et $p^-_{\g_2}$ qui est inclus dans $\U$. De la même façon, on peut supposer que $\widetilde{l_1}$ est le segment ouvert d'extrémité $p^-_{\g_2}$ et $p^-_{\g_3}$ inclus dans $\U$ et $\widetilde{l_2}$ le segment ouvert d'extrémité $p^-_{\g_3}$ et $p^-_{\g_1}$ inclus dans $\U$. La triangulation topologique de $\widetilde{P}$ peut donc être réalisée dans $\U$ par des géodésiques.

On note $T$ le triangle idéal fermé dans $\O$ de sommets $p_{\g_1}^-,p_{\g_2}^-,p_{\g_3}^-$ et $T' = \g_1 T$. Ces deux triangles sont inclus dans $\U$ car $\U$ est convexe. Le groupe fondamental de $P$ est le groupe $\G'$ engendré par $\g_1,\g_2$ et $\g_3$. Par conséquent, $T \cup T'$ est un domaine fondamental pour l'action de $\G'$ sur l'intérieur de $\U$. La proposition \ref{airetri} conclut notre démonstration.
\end{proof}

On obtient donc la proposition suivante:

\begin{theo}\label{typefini}
Toute surface $S$ projective proprement convexe de volume fini est de type fini (i.e le groupe fondamental $\pi_1(S)$ de $S$ est de type fini).
\end{theo}

\begin{proof}
L'application d\'eveloppante permet d'identifier $S$ au quotient d'une partie proprement convexe $\C$ par un sous-groupe discret $\G$ de $\s$. Commençons par faire le cas o\`u $S$ est une surface sans bord. Le th\'eor\`eme \ref{Lee}  montre que l'action de $\G$ sur l'ouvert $\O=\C$ admet un domaine fondamental convexe et localement fini. Si $\overline{D} \cap \partial \O$ admet un nombre infini de composantes connexes, $D$ contient une infinit\'e de triangles id\'eaux disjoints. La proposition \ref{airetri} montre que l'aire de chacun de ces triangles est minor\'ee par une constante strictement positive. La surface $S$ est donc de volume infini. Par cons\'equent, $\overline{D} \cap \partial \O$ admet un nombre  fini de composantes connexes et la surface $S$ poss\`ede donc un nombre fini de bouts. La proposition \ref{toposurf}  montre qu'il suffit donc de montrer que la surface $S$ ne peut contenir une infinit\'e de pantalons non \'el\'ementaires disjoints. Les propositions \ref{homopant} et \ref{minaire} montre que $S$ contient un nombre fini de pantalons non \'el\'ementaires.

Supposons \`a pr\'esent que $S$ est une surface \`a bord. Le th\'eor\`eme \ref{Lee}  montre que l'action de $\G$ sur l'int\'erieur $\O$ de $\C$ admet un domaine fondamental convexe et localement fini. Commençons par montrer que $\overline{D} \cap \partial \O$ poss\`ede un nombre fini de composante connexe. Sinon, comme dans le cas sans bord, le domaine $D$ contient une infinit\'e de triangles id\'eaux disjoints. Le probl\`eme dans le cas avec bord est que la projection de ces triangles n'est pas un ferm\'e inclus dans l'int\'erieur de $S$. Mais, on peut retirer un petit voisinage de chaque sommet et ainsi on peut construire une infinit\'e d'hexagones disjoints inclus dans $D$ et de volume sup\'erieure \`a $C_{\P}/2$ et ces hexagones se projettent sur un ferm\'e de $\overset{\circ}{S}$. Par cons\'equent, $\overline{D} \cap \partial \O$ admet un nombre  fini de composantes connexes et la surface $\overset{\circ}{S}$ poss\`ede donc un nombre fini de bouts. Son genre est fini pour les m\^emes raisons que dans le cas sans bord.
\end{proof}

\subsubsection{Le domaine fondamental est un poly\`edre fini}

\begin{defi}
Soit $\O$ un ouvert proprement convexe de $\P$, un \emph{poly\`edre} (resp. \emph{poly\`edre fini})de $\O$ est un ferm\'e d'int\'erieur non vide d\'efini comme l'intersection d'une famille (resp. famille finie) de demi-espaces de $\O$.
\end{defi}

\begin{prop}\label{polyfini}
Soient $\O$ un ouvert proprement convexe de $\P$ et $\G$ un sous-groupe discret sans torsion de $\s$ qui pr\'eserve $\O$, on suppose que l'action de $\G$ sur $\O$ est de covolume fini. Alors, tout domaine fondamental de Dirichlet-Lee est un poly\`edre fini.
\end{prop}

\begin{proof}
Le th\'eor\`eme \ref{typefini} montre que la surface $\Quo$ est une surface de type fini, en particulier elle poss\`ede un nombre fini de bout. Soit $D$ un domaine fondamental de Dirichlet-Lee, on sait que $D$ est un poly\`edre. La d\'emonstration se fait en trois \'etapes.

Tout d'abord le nombre de composantes connexes de $\overline{D} \cap \partial \O$ est fini. Sinon on pourrait construire une infinit\'e de triangles id\'eaux disjoints deux \`a deux inclus dans $D$. Et, la proposition \ref{airetri}  montre qu'alors $D$ est de volume infini.

De plus, la proposition \ref{mubord}  montre si l'une des composantes connexes de $\overline{D} \cap \partial \O$ n'est pas r\'eduite \`a un point, alors $D$ est de volume infini.

Comme $D$ est un domaine fondamental localement fini, il ne  reste plus qu'\`a montrer que pour tout point $x_{\infty} \in \overline{D} \cap \partial \O$, il existe un voisinage $V$ de $x_{\infty}$ tel que $V$ ne rencontre qu'un nombre fini de faces de $D$.

Soit $x_{\infty} \in \overline{D} \cap \partial \O$, et $x_0$ un point de $D$, la projection du segment $]x_0,x_{\infty}[$ sur la surface $\Quo$ est une demi-g\'eod\'esique $\lambda$ non born\'ee et simple. On note $\widetilde{\lambda}$ le segment $]x_0,x_{\infty}[$ param\'etr\'e par la longueur d'arc pour la distance de Hilbert. La proposition \ref{proprete} montre qu'il existe une suite $t_n$ tendant vers l'infini et un bout $B$ de $S$ tel que la suite $(\lambda(t_n))_{n \in \N}$ est ultimement incluse dans le bout $B$. Nous allons montrer qu'en fait la demi-géodésique $\lambda$ est ultimement incluse dans le bout $B$. Au bout $B$ de $\Quo$ est associ\'e un lacet \'el\'ementaire $c$, bien d\'efini \`a homotopie pr\`es et \`a orientation pr\`es. On choisit un \'el\'ement $\g$ qui repr\'esente l'holonomie de $c$.

On note $\F$ un secteur de $\g$ tel que l'application naturelle de $\F/_{<\g>}$ vers $S$ est injective (lemme \ref{injgeo}). La proposition \ref{fonda1} montre que $\F/_{<\g>}$ contient un voisinage du bout $B$, sauf si $\g$ est hyperbolique et $p^0_{\g} \in \O$. Dans ce cas, on note $\widetilde{\F}$ l'union de $\F$ et de l'unique triangle inclus dans $\O$ d\'efini par les points $p^+_{\g},p^-_{\g},p^0_{\g}$. On peut donc toujours supposer qu'il existe un ferm\'e convexe $\widetilde{\F}$ tel que l'application naturelle de $\widetilde{\F}/_{<\g>}$ vers $S$ est injective (lemme \ref{injgeo}) et que son image contient un voisinage du bout $B$ de $S$.

Il existe donc un entier $N_0 \in \N$ tel que pour tout $n \geqslant N_0$, il existe un \'el\'ement $\delta_n \in \G$ tel que $\delta_n \widetilde{\lambda}(t_n) \in \widetilde{\F}$. Le domaine $D$ ne rencontre qu'un nombre fini d'images de $\widetilde{\F}$, quitte \`a extraire et \`a conjuguer $\g$, on peut donc supposer que la suite $\widetilde{\lambda}(t_n)$ est ultimement incluse dans $\widetilde{\F}$. Or, $\widetilde{\lambda}$ est une demi-g\'eod\'esique incluse dans $D$ et $\widetilde{\F}$ est convexe donc $\widetilde{\F}$ contient ultimement $\widetilde{\lambda}$.

Le domaine $D$ rencontre un nombre fini d'image du ferm\'e $\widetilde{\F}$. Il existe donc un voisinage de $x_{\infty}$ qui ne rencontre qu'un nombre fini de c\^ot\'es de $D$. C'est ce qu'il fallait montrer.
\end{proof}

\subsection{Holonomie des lacets \'el\'ementaires et volume des composantes \'el\'ementaires associ\'ees}

\subsubsection{Estimation du volume des pics}

\begin{prop}\label{pic}
Soient $\O$ un ouvert proprement convexe de $\P$ et $P$ un pic de $\O$, on suppose que le sommet $p$ \`a l'infini de $P$ est fix\'e par un \'el\'ement $\g$ non trivial de $\Aut(\O)$. Alors,
\begin{itemize}
\item Si $\g$ est parabolique alors $\mu_{\O}(P) < \infty$.


\item Si $\g$ est quasi-hyperbolique alors $\mu_{\O}(P) = \infty$

\item Si $\g$ est hyperbolique et $\Ax(\g) \subset \partial \O$ alors $\mu_{\O}(P) = \infty$.

\item Si $\g$ est hyperbolique et $\Ax(\g) \subset \O$ et $p=p^0_{\g} \in \partial \O$ alors $\mu_{\O}(P) = \infty$.
\end{itemize}
\end{prop}

Nous allons montrer cette proposition \`a l'aide de plusieurs lemmes.

\begin{lemm}\label{compaelli}
Soient $\O$ un ouvert proprement convexe de $\P$ et $P$ un pic de $\O$, on suppose que le sommet $p$ de $P$ qui appartient au bord $\partial \O$ de $\O$ est sur une ellipse $\mathcal{E}$ dont l'int\'erieur (i.e la composante connexe orientable de $\P- \mathcal{E}$) est dans $\O$. Alors, $\mu_{\O}(P) < \infty$.
\end{lemm}

\begin{proof}
On note $E$ l'int\'erieur de $\mathcal{E}$. On peut supposer que $P \subset E$. La proposition \ref{compa}  montre que $\mu_{\O}(P) \leqslant \mu_{E}(P)$. Comme tout ellipsoïde muni de la distance de Hilbert est isométrique au plan hyperbolique, on est ramen\'e \`a montrer que tout pic du plan hyperbolique est d'aire fini. Mais tout pic du plan hyperbolique est inclus dans un triangle idéal, et, tout triangle idéal est d'aire $\pi$ en géométrie hyperbolique.
\end{proof}

\begin{lemm}\label{faiselli}
Soit $\g \in \s$ un \'el\'ement parabolique, l'\'el\'ement $\g$ pr\'eserve un faisceau d'ellipses tangentes \`a la droite $D_{\g}$ au point $p_{\g}$.
\end{lemm}

\begin{proof}
On peut supposer que l'\'el\'ement $\g$ est donn\'e par la matrice:
$$
\left(\begin{array}{ccc}
1  & 1  & 0  \\
0  & 1  & 1 \\
0  & 0  & 1 \\
\end{array}
\right)
$$
Si on utilise les coordonn\'ees $x,y,z$ alors l'\'el\'ement $\g$ pr\'eserve les polyn\^omes $z$ et $y^2-z(y+2x)$. Par cons\'equent, l'\'el\'ement $\g$ pr\'eserve le faisceaux d'ellipses d'\'equation $\{ \lambda z^2 + \mu (y^2-z(y+2x))=0 \}_{\lambda,\mu \in \R}$. Ces ellipses ont pour point commun le point $p_{\g}=[1:0:0]$ et leurs tangentes en $p_{\g}$ est la droite $D_{\g}$ d'\'equation $z=0$.
\end{proof}

\begin{proof}
[D\'emonstration de la proposition \ref{pic} dans le cas parabolique]
La premi\`ere chose \`a remarquer est que la proposition \ref{para}  montre que $\partial \O$ est $\C^1$ en $p$. On note $D$ la tangente \`a $\partial \O$ en $p$. Le lemme \ref{faiselli}  montre que l'\'el\'ement $\g$ pr\'eserve un faisceau d'ellipses  $(\mathcal{E}_i)_{i \in I}$ tangentes \`a la droite $D$ au point $p$. L'unique point commun des ellipses $(\mathcal{E}_i)_{i \in I}$ est le point $p$, leurs tangentes en ce point est la droite $D$. Nous allons montrer que l'int\'erieur de l'une de ces ellipses est inclus dans $\O$, ce qui conclura la d\'emonstration grâce au lemme \ref{compaelli}.

On se place dans une carte affine qui contient $\overline{\O}$, on munit cette carte d'un produit scalaire qui fait du faisceau d'ellipses tangentes $(\mathcal{E}_i)_{i \in I}$ un faisceau de cercles tangents. On note $A$ la droite perpendiculaire \`a la droite $D$ pour le produit scalaire choisi. La droite $A$ est donc perpendiculaire \`a chacun des cercles $(\mathcal{E}_i)_{i \in I}$. On notera l'image de $A$ par $\g$, $A'=\g A$.

Soit $x \in A \cap \O$, on note $T_x$ la tangente en $x$ \`a l'unique cercle $\mathcal{E}_i$ passant par $x$. \`{A} pr\'esent, les droites $A$, $A'$ et $T_x$ d\'efinissent 4 triangles ferm\'es de $\P$, on note $\mathcal{T}_x$ celui qui est born\'e dans la carte $A$.  Si le point $x$ est suffisament proche de $p$ et dans $\O$ alors le triangle $\mathcal{T}_x$ est inclus dans $\O$. On se donne donc un tel point $x$ et on va montrer que l'int\'erieur de l'unique cercle $\mathcal{E}_i$ passant par $x$ est inclus dans $\O$.

Pour cela, il suffit de remarquer que la famille des $(\g^{t} x)_{0 \leqslant t \leqslant 1}$ est incluse dans $\mathcal{T}_x$ et donc dans $\O$. L'ouvert $\O$ est $\g$-invariant, par cons\'equent l'orbite $(\g^{t} x)_{t \in \R}$ est incluse dans $\O$, et cette orbite est l'unique cercle $\mathcal{E}_i$ passant par $x$ priv\'e du point $p$.
\end{proof}

\begin{proof}[D\'emonstration de la proposition \ref{pic} dans les cas hyperboliques et le cas quasi-hyperbolique avec $p=p^2_{\g}$]
Ce cas l\`a est beaucoup plus facile. En effet, si $\g$ est hyperbolique et $\Ax(\g) \subset \partial \O$ la proposition \ref{hyp}  montre que les points $p^+_{\g}$ et $p^-_{\g}$ ne sont pas des points $C^1$ de $\partial \O$. De même, si $\g$ est hyperbolique et $\Ax(\g) \subset \O$ et $p=p^0_{\g} \in \partial \O$, alors le point $p$ n'est pas un point $C^1$ de $\partial \O$. Et, si $\g$ est quasi-hyperbolique alors la proposition \ref{quasihyp}  montre que $\partial \O$ n'est pas $C^1$ en $p^2_{\g}$. Par cons\'equent $\mu_{\O}(P) = \infty$ par le th\'eor\`eme \ref{mupointe}.
\end{proof}

Il  reste le cas o\`u $\g$ est quasi-hyperbolique et $p=p^1_{\g}$. Pour cela, on a besoin de conna\^itre l'allure de l'orbite d'un point $x \in \P$ sous l'action d'un \'el\'ement quasi-hyperbolique. Soit $\g$ un \'el\'ement quasi-hyperbolique, on peut supposer que $\g$ est donn\'e par la matrice suivante:
$$
\left(\begin{array}{ccc}
\alpha  & \alpha       & 0  \\
0       & \alpha  & 0 \\
0       & 0       & \beta \\
\end{array}
\right)
$$
O\`u, $\alpha, \beta > 0$ et $\alpha^2 \beta = 1$.
Nous allons donner les \'equations des orbites d'un point de $\P$ sous l'action du groupe \`a un param\`etre engendr\'e par $\g$. Pour cela, on se place dans la carte affine: $A= \{ [x:y:z] \in \P \,|\, z \neq 0 \}$. L'action de $\g$ dans cette carte est donn\'e par:
$$(X,Y) \mapsto (\frac{\alpha}{\beta} X + \frac{\alpha}{\beta} Y, \frac{\alpha}{\beta} Y)$$
Si $Y_0 \neq 0$, un calcul simple montre que l'\'equation de l'orbite du point $(X_0,Y_0) \in A$ est:
$$\frac{X}{X_0} = \frac{Y}{Y_0} + \frac{Y}{X_0} \frac{\ln(\frac{Y}{Y_0})}{\ln(\frac{\alpha}{\beta})}$$
Si $Y_0 = 0$, l'\'equation de l'orbite du point $(X_0,Y_0) \in A$ est la droite:
$$Y=0$$

Nous allons avoir besoin du lemme ci-dessous pour estimer l'aire d'un pic dont le sommet \`a l'infini est le $p^1_{\g}$ d'un \'el\'ement quasi-hyperbolique.

\begin{lemm}\label{quhyp}
On consid\`ere l'ouvert $\O_0$ de $\R^2$ d\'efinit par $\O_0 = \{ (x,y) \in \R^2 \,|\,  x > 0 \, , \, y >  x \ln(x)  \}$. Tout pic de $\O_0$ dont le sommet \`a l'infini est le point $(0,0)$ est de volume infini pour la mesure de Busemann associ\'ee \`a l'ouvert $\O_0$.
\end{lemm}

\begin{rema}
La courbe $y(x) = x \ln(x)$ est affinement équivalente à la courbe $y'(x') = x' + a x' \ln(x')$ avec $a \neq 0$, via le changement de variable $x = e^{\frac{1}{a}}x'$ et $y = \frac{1}{a}e^{\frac{1}{a}} y'$.
\end{rema}

\begin{proof}
L'ouvert $\O_0$ ne contient aucune droite affine, c'est donc un ouvert proprement convexe de $\P$. Soit $(x,y) \in \O_0$, commençons par \'evaluer le volume de la boule $B^{\O_0}_{(x,y)}(1)$. On identifie l'espace tangent \`a $\O_0$ en $(x,y)$ \`a l'espace vectoriel $\R^2$.

La premi\`ere chose \`a remarquer est que $\O_0$ est inclus dans l'ouvert convexe $\O_0'=\{ (x,y) \in \R^2 \,|\, x > 0 \}$. Par cons\'equent, $B^{\O_0}_{(x,y)}(1)$ est incluse dans la bande $\{ (u,v) \in \R^2 \,|\, -x<u<x\}$.

De plus, un calcul simple montre que le vecteur $u =(0,y- f(x))$ appartient au bord de la boule $B^{\O_0}_{(x,y)}(1)$.

Enfin, la boule $B^{\O_0}_{(x,y)}(1)$ est sym\'etrique par rapport \`a l'origine. Par cons\'equent, les tangentes \`a $\partial B^{\O_0}_{(x,y)}(1)$ en $u$ et $-u$ sont parall\`eles. Il vient que le volume de $B^{\O_0}_{(x,y)}(1)$ v\'erifie:

$$\Vol(B^{\O_0}_{(x,y)}(1)) \leqslant 4x(y-f(x))$$

Le volume d'un pic $P= \{(x,y) \in \R^2 \,|\, 0<x < \epsilon  \, , \, ax < y < bx \}$ avec $0 <a < b$ et $\epsilon > 0$ est donn\'e par:

\[
\begin{array}{lcl}
\mu_{\O_0}(P)= \int_0^{\epsilon} \int_{ax}^{bx} \frac{1}{\Vol(B^{\O_0}_{(x,y)}(1))} dy dx & \geqslant & \int_0^{\epsilon} \int_{ax}^{bx}\frac{1}{4x(y-f(x))}dydx\\
 & \geqslant & \int_0^{\epsilon} \frac{1}{4x} \ln\Big(\frac{b-\ln(x)}{a-\ln(x)} \Big) dx = \infty
\end{array}
\]

La quantit\'e $\frac{b-\ln(x)}{a-\ln(x)}$ tend vers 1 lorsque $x$ tend vers $0$. Les fonctions $x \mapsto \frac{1}{4x} \ln\Big(\frac{b-\ln(x)}{a-\ln(x)}\Big)$ et $x \mapsto -\frac{b-a}{4x\ln(x)}$ sont \'equivalentes et positives au voisinage de $0$. Par cons\'equent, l'int\'egrale $\int_0^{\epsilon} \frac{1}{4x} \ln\Big( \frac{b-\ln(x)}{a-\ln(x)}\Big)$ est infini car l'int\'egrale $\int_0^{\epsilon} -\frac{b-a}{4x\ln(x)}$ est infini.
\end{proof}

\`{A} pr\'esent, nous allons utiliser la proposition \ref{compa} et le lemme \ref{quhyp} pour montrer la proposition \ref{pic} dans le cas quasi-hyperbolique.

\begin{proof}
[D\'emonstration de la proposition \ref{pic} dans le cas quasi-hyperbolique avec $p=p^1_{\g}$]
On rappelle que l'\'el\'ement $\g$ poss\`ede deux droites stables. Ces deux droites d\'efinissent une partition de $\P$ en deux demi-espaces ferm\'es. On a vu au lemme \ref{quasihyp} que $\O$ \'etait inclus dans l'un de ces deux demi-espaces ouverts, on note $A$ celui qui le contient. Nous allons montrer qu'il existe un point $x \in A$ tel que l'orbite de $x$ sous l'action du groupe \`a un param\`etre $H$ engendr\'e par $\g$ n'intersecte pas $\O$.

Commençons par voir pourquoi l'existence d'un tel point  permet de conclure. Soit $x \in A$ tel que l'orbite $\mathcal{O}$ de $x$ sous l'action $H$ n'intersecte pas $\O$. L'union de l'orbite $\mathcal{O}$ et de l'adh\'erence de l'axe de $\g$ forment une courbe convexe qui d\'efinit un ouvert proprement convexe projectivement \'equivalent \`a l'ouvert proprement convexe $\O_0$ du lemme \ref{quhyp}. La proposition \ref{compa} et le lemme \ref{quhyp}  montrent que tout pic de $\g$ dont le sommet \`a l'infini est stabilis\'e par un \'el\'ement quasi-hyperbolique est de volume infini.

Montrons \`a pr\'esent qu'un tel point $x$ existe. Il  suffit de trouver un point $x \in A$ tel que l'ensemble $(\g^t x)_{0 \leqslant t \leqslant 1}$ est inclus dans $A - \O$, puisque l'ensemble $A - \O$ est préservé par $\g$. Pour cela on consid\`ere une droite quelconque $D$ passant par $p^1_{\g}$ et son image $D'$ par $\g$. On se donne un point $x \in D \cap A$, et on consid\`ere le segment $S$ inclus dans $A$ et d\'efinit par les points $x$ et $p^1_{\g}$. La r\'egion $R= \{ y \in A \,|\, \exists t, \, 0 \leqslant t \leqslant 1, \, \exists z \in S,\, \textrm{ tel que } y = \g^t(z) \}$ est une partie convexe dont le bord est form\'e de trois courbes. Une incluse dans $D$, une autre dans $D'$ et la derni\`ere est la courbe $(\g^t x)_{0 \leqslant t \leqslant 1}$. On note $\mathcal{T}_x$ le triangle ferm\'e de $A$ d\'efini par les points $p^1_{\g}, x, \g x$. Si le point $x \in D$ tend vers l'infini, alors le point $\g x \in D'$ tend aussi vers l'infini. Par conséquent, il existe un point $x\in D \subset A$ tel que le c\^ot\'e de $\mathcal{T}_x$ oppos\'e \`a $p^1_{\g}$ soit inclus dans $A - \O$. Mais le segment $[x,\g x]$ est une corde de la courbe convexe $(\g^t x)_{0 \leqslant t \leqslant 1}$. Par cons\'equent, si on se donne un tel $x$ alors la courbe $(\g^t x)_{0 \leqslant t \leqslant 1}$ est incluse dans $A - \O$.
\end{proof}

\subsubsection{L'holonomie des bouts des surfaces projectives proprement convexes de volume fini est parabolique}

\begin{defi}
Soient $\C$ une partie proprement convexe et $\O=
\overset{\circ}{\C}$, on appelle bord r\'eel de $\C$ la partie $\C
\cap \partial \O$, on le note $\partial_r \O$.
\end{defi}

\begin{prop}\label{fonda2}
Soit $S$ une surface de caractéristique d'Euler strictement négative, on munit $S$ d'une structure projective proprement convexe. Soit $c$ un lacet \'el\'ementaire trac\'e sur $S$, on note $\pi : \C \rightarrow S$ le rev\^etement universel de $S$ donn\'e par la
d\'eveloppante de $S$.

\begin{itemize}
\item Si $c$ fait le tour d'un bout et l'holonomie de $c$ est
hyperbolique ou quasi-hyperbolique alors le volume de la
composante \'el\'ementaire associ\'ee \`a $c$ est infini.

\item Si $c$ fait le tour d'un bout et l'holonomie de $c$ est
parabolique alors le volume de la composante \'el\'ementaire associ\'ee
\`a $c$ est fini.

\item Si $c$ est librement homotope \`a une composante connexe du
bord de $S$ alors l'holonomie de $c$ est:
\begin{itemize}
\item ou quasi-hyperbolique avec $\Ax(Hol(c)) \subset \partial_r \C$

\item ou hyperbolique avec $\Ax(Hol(c)) \subset \partial_r \C$

\item ou hyperbolique avec $p^0_{\g} \in \partial \C$ et l'un des axes secondaires de $\g$ est inclus dans $\partial_r \C$.
\end{itemize}
\end{itemize}
\end{prop}

\begin{proof}
On note $\g = Hol(c)$. On sait par le th\'eor\`eme \ref{Lee} qu'il existe un domaine fondamental $D$ convexe et localement fini pour l'action du groupe fondamental $\G$ de $S$ sur l'int\'erieur $\O$ de $\C$. La proposition \ref{fonda1} montre qu'il existe un ferm\'e $\g$-invariant et convexe $\F$ de $\O$ tel l'application naturelle de $\F/_{<\g>}$ vers $S$ est injective et son image est une composante \'el\'ementaire associ\'ee \`a un lacet homotope \`a $c$. Le lemme \ref{inftyfini}  montre que $D$ rencontre un nombre fini d'image de $\F$. On les note $\delta_i \F$ avec $\delta_i \in \G$ et $\delta_1 = Id$. Le volume de $\F/_{<\g>}$ est \'egale au volume de $(\underset{i=1...r}{\bigcup} \delta_i^{-1}D) \cap \F$.

Il  reste donc \`a estimer le volume des $\delta_i^{-1}D \cap \F$ pour $i=1...r$. Il suffit de calculer le volume de $D \cap \F$. Traitons les 3 cas s\'eparement: 
\begin{itemize}
\item Si $c$ fait le tour d'un bout $B$ et l'holonomie de $c$ est
hyperbolique ou quasi-hyperbolique alors on a deux cas \`a distinguer.
\begin{itemize}
\item Les deux faces du poly\`edre $D$ intersectant $\F$ ont une extr\'emit\'e commune $p$ appartenant \`a $\overline{\O}$. Le point $p$ est alors fix\'e par $\g$. Si l'élément $\g$ est hyperbolique alors, il y a trois sous cas: a) $\Ax(\g) \subset \partial \O$,  b) $\Ax(\g) \subset \O$ et $p_{\g}^0 \in \partial \O$, ou bien c) $\Ax(\g) \subset \O$ et $p_{\g}^0 \notin \partial \O$.

Si l'élément $\g$ est quasi-hyperbolique ou hyperbolique cas a) , alors les cas 2 et 3 de la proposition \ref{pic} montre que la composante \'el\'ementaire associ\'ee \`a $c$ est de volume infini.

Si l'élément $\g$ est hyperbolique cas b) alors $p=p_{\g}^0$, car l'ensemble $\F/_{<\g>}$ est un voisinage du bout $B$.
Le cas 4 de la proposition \ref{pic} montre que la composante \'el\'ementaire associ\'ee \`a $c$ est de volume infini.

Enfin, si l'élément $\g$ est hyperbolique cas c) alors comme l'ensemble $\F/_{<\g>}$ est un voisinage du bout $B$, le point $p$ ne peut-être fixé par $\g$. Par conséquent, on ne peut pas être dans ce cas.

\item Les deux faces du poly\`edre $D$ incluse dans $\F$ n'ont pas d'extr\'emit\'es commune appartenant \`a $\overline{\O}$. Dans ce cas, c'est le th\'eor\`eme \ref{mubord} qui  montre que la composante \'el\'ementaire associ\'ee \`a $c$ est de volume infini.
\end{itemize}

\item  Si $c$ fait le tour d'un bout et l'holonomie de $c$ est
parabolique alors les deux faces du poly\`edre $D$ incluse dans $\F$ s'intersectent en $p_{\g}$ le point fixe de $\g$. Dans ce cas la proposition \ref{pic}  montre que la composante \'el\'ementaire associ\'ee \`a $c$ est de volume fini.

\item Si $c$ est librement homotope \`a une composante connexe $L$ du bord de $S$ alors on a vu que l'holonomie de $c$ ne peut-\^etre parabolique (proposition \ref{fonda1}). La surface $S$ est une surface projective \underline{\`a bord g\'eod\'esique} par cons\'equent tout relev\'e de $L$ est un segment $T$ inclus dans le bord de $\C$ et pr\'eserv\'e par $\g$. Si $\g$ est quasi-hyperbolique alors il n'y a rien \`a montrer car $\Ax(\g)$ est l'unique segment pr\'eserv\'e par $\g$ inclus dans $\partial \C$. On suppose donc que $\g$ est hyperbolique. Si le point $p^0_{\g}$ n'appartient pas \`a $\partial \O$ alors $T$ est n\'ecessairement l'axe principal de $\g$, car dans ce cas $\Ax(\g)$ est l'unique segment pr\'eserv\'e par $\g$ inclus dans $\partial \C$. Et, si $p^0_{\g}$ appartient \`a $\partial \O$ alors les axes secondaires de $\g$ sont les seuls segments pr\'eserv\'es par $\g$ et inclus dans $\partial \C$. Ils ne peuvent pas \^etre tous les deux inclus dans $\partial \C$ cas sinon le quotient ne serait pas une surface. Par cons\'equent, $T$ est l'un des deux axes secondaires de $\g$.
\end{itemize}
\end{proof}

On peut \`a pr\'esent montrer le r\'esultat principal de cet article.

\begin{theo}\label{le resultat}
Soit $S$ une surface qui n'est pas un cylindre, on munit $S$ d'une structure projective proprement convexe, cette structure est de volume fini si et seulement si la surface $S$ est de type fini et l'holonomie de tous les lacets qui font le tour d'un bout est parabolique.
\end{theo}

\begin{proof}
On commence par se placer dans le cas où la surface $S$ est de caractéristique d'Euler strictement négative.
Supposons que que l'on s'est donn\'e une structure projective proprement convexe sur $S$ de volume fini. Alors, la surface $S$ est de type fini d'apr\`es le th\'eor\`eme \ref{typefini}, et l'holonomie de chaque bout est parabolique d'apr\`es la proposition \ref{fonda2}.

Supposons \`a pr\'esent que $S$ est de type fini et que l'holonomie de tous les bouts de $S$ est parabolique. La proposition \ref{fonda2}  montre que la surface $S$ est de volume fini.

Enfin, supposons la surface $S$ est de caractéristique d'Euler positive. Si l'intérieur de $S$ est compacte le théorème est évident. Si l'intérieur $S$ n'est pas compact alors $S$ est un cylindre avec ou sans bord. On a exclu ce cas.
\end{proof}

Le corollaire suivant est imm\'ediat.

\begin{coro}\label{caract}
Soit $S$ une surface sans bord, on munit $S$ d'une structure projective proprement convexe. Alors, cette structure est de volume fini si et seulement si la surface $S$ est de type fini et l'holonomie de tous les lacets \'el\'ementaires est parabolique.
\end{coro}

\begin{proof}
Si la surface $S$ n'est pas un cylindre alors ce corollaire est une conséquence direct du théorème \ref{le resultat}. Si la surface $S$ est un cylindre, c'est un cylindre sans bord, il est alors assez facile de voir que l'une des deux composantes élémentaires de $S$ est de volume infini.
\end{proof}

\begin{prop}
Soient $\O$ un ouvert proprement convexe de $\P$ et $\G$ un sous-groupe discret sans torsion de $\s$ qui pr\'eserve $\O$, si $\G$ contient un \'el\'ement quasi-hyperbolique alors l'action de $\G$ sur $\O$ est de covolume infini.
\end{prop}

\begin{proof}
Supposons que le groupe $\G$ contienne un \'el\'ement $\g$ quasi-hyperbolique. L'ouvert $\O$ n'est donc pas un triangle et la surface quotient $\Quo$ est de caractéristique d'Euler négative. Il existe un revêtement fini de $\Quo$ dans lequel $\g$ est représenté par un lacet simple. La proposition \ref{fonda1}  montre que le lacet correspondant \`a $\g$ sur la surface $\Quo$ est \'el\'ementaire. Par cons\'equent, la proposition \ref{fonda2}  montre que $\Quo$ est de volume infini.
\end{proof}

\section{Applications}

\subsection{Stricte convexit\'e}

Le but de cette partie est d'\'etudier la stricte-convexit\'e de l'ouvert $\O$. Nous allons avoir besoin du lemme suivant.

\begin{lemm}\label{lemstrict}
Soient $\C$ une partie proprement convexe et un sous-groupe sans torsion discret $\G$ de $\s$ qui pr\'eserve $\C$ et tel que le quotient $S = \Quc$ est une surface \`a bord g\'eod\'esique. Soient $x_0$ un point de l'int\'erieur $\O$ de $\C$ et un point $x_{\infty}$ de $\partial \O$. On note $\widetilde{\lambda}$ le segment $[x_0,x_{\infty}[$ param\'etr\'e par la longueur d'arc (pour la distance de Hilbert) et $\lambda$ la demi-g\'eod\'esique obtenue en projetant $\widetilde{\lambda}$ sur $S$. Si la demi-g\'eod\'esique $\lambda$ est ultimement incluse dans un bout $B$ de $\overset{\circ}{S}$ alors il existe $\g \in \G$ repr\'esentant l'holonomie du bout $B$ tel que:
\begin{itemize}
\item si $\g$ n'est pas hyperbolique avec $p^0_{\g} \in \partial \O$ alors $\widetilde{\lambda}$ est ultimement inclus dans tout secteur de $\g$.

\item si $\g$ est hyperbolique avec $p^0_{\g} \in \partial \O$ alors $\widetilde{\lambda}$ est ultimement inclus dans la r\'eunion de tout secteur de $\g$ et de l'unique triangle $T_{\g}$ inclus dans $\C$ d\'efinit par les points $p^+_{\g},p^-_{\g},p^0_{\g}$.
\end{itemize}
\end{lemm}

\begin{proof}
On consid\`ere un lacet simple $c$ qui fait le tour du bout $B$ de $S$ et l'\'el\'ement $\g=\Hol(c)$. On note $\F$ un secteur de $\g$. La proposition \ref{fonda1} montre que si $\g$ n'est pas hyperbolique avec $p^0_{\g} \in \partial \O$ alors $\F/_{<\g>}$ est un voisinage du bout $B$. Elle montre aussi que si $\g$ est hyperbolique avec $p^0_{\g} \in \partial \O$ alors pour obtenir un voisinage du bout $B$, on peut prendre la projection de la r\'eunion de $\F$ et du triangle $T_{\g}$ d\'efinit par les points $p^+_{\g},p^-_{\g},p^0_{\g}$.

Comme $\F$ ou $\F \cup T_{\g}$ est convexe et que la demi-g\'eod\'esique $\lambda$ converge vers le bout $B$ de $S$, quitte \`a conjuguer $\g$ on peut supposer que $\lambda$ est inclus dans $\F$.
\end{proof}

On obtient le corollaire suivant:

\begin{coro}\label{corostrict}
Soient $\O$ un ouvert proprement convexe et un sous-groupe
discret $\G$ qui pr\'eserve $\O$. On suppose que l'action de $\G$ sur $\O$ est de covolume fini. On se donne un domaine fondamental $D$ convexe et localement fini. Soient $x_0$ un point de $\O$ et $x_{\infty}$ un point de $\partial \O$. On note $S$ la surface $\Quo$, $\widetilde{\lambda}$ le segment $[x_0,x_{\infty}[$ param\'etr\'e par la longueur d'arc et $\lambda$ la demi-g\'eod\'esique obtenue en projetant $\widetilde{\lambda}$ sur $S$. Si la demi-g\'eod\'esique $\lambda$ est ultimement incluse dans un bout $B$ de $S$ alors il existe $\delta \in \G$ et un $T>0$ tel que pour tout $t>T$, $\lambda(t) \in \delta D$. En particulier, $x_{\infty}$ est le point fixe d'un \'el\'ement parabolique de $\G$ qui repr\'esente l'holonomie du bout $B$.
\end{coro}

\begin{proof}
Le corollaire \ref{caract} montre que l'holonomie de tous les lacets \'el\'ementaires de $S$ est parabolique. L'intersection de l'adh\'erence de tous les secteurs d'un \'el\'ement parabolique ne contient qu'un seul point du bord de $\O$: le point fixe de cet \'el\'ement parabolique. Par cons\'equent, le point $x_{\infty}$ est le point fixe d'un \'el\'ement parabolique de $\G$. Il vient que le segment $[x_0,x_{\infty}[$ est ultimement inclus dans un domaine fondamental.
\end{proof}

On peut \`a pr\'esent montrer le th\'eor\`eme suivant:

\begin{theo}\label{strictabord}
Soient $\C$ une partie proprement convexe et $\G$ un sous-groupe
discret de type fini et sans torsion de $\s$ qui pr\'eserve $\C$. On suppose que le quotient $\Quc$ est une surface projective \`a bord g\'eod\'esique et que $\C$ n'est pas un triangle. Alors, tout segment maximal non trivial de $\partial \C$ est pr\'eserv\'e par un \'el\'ement non parabolique de $\G$ qui correspond \`a un lacet \'el\'ementaire de $\Quc$.
\end{theo}

On rappelle le th\'eor\`eme suivant dû \`a Benz\'ecri que l'on a d\'ej\`a utilis\'e dans la partie \ref{Hil}.

\begin{theo}[Benz\'ecri]\label{ben2}
L'action de $\s$ sur l'ensemble $\{ (\O,x) \, | \, \O \textrm{ est
un}$ $\textrm{ouvert proprement convexe de } \P \textrm{ et } x
\in \O \}$ est propre et cocompacte.
\end{theo}

\begin{proof}
[D\'emonstration du th\'eor\`eme \ref{strictabord}]
On note $\O$ l'int\'erieur de $\C$. On obtient ainsi $\overset{\circ}{S}=\Quo$ une surface projective convexe sans bord et de type fini. Supposons qu'il existe un segment $s$ non trivial et maximal inclus dans le bord $\partial \C$ de $\C$. Soit
$(e_i)_{i=1...3}$ une base de $\R^3$. On peut supposer que les
extr\'emit\'es de $s$ sont les points $[e_2]$ et $[e_3]$ et que $[e_1]
\notin \overline{\O}$. On consid\`ere les \'el\'ements:
$$
g_{t} = \left(\begin{array}{c c c} 2^t & 0 & 0\\
0 & 2^{-t}& 0\\
0 & 0 & 2^{-t}
\end{array} \right).
$$
Ainsi, $\underset{t \rightarrow +\infty}{\lim} g_t \O$ est un
triangle $T$ dont les sommets sont $e_1, \, e_2$ et $e_3$.

Soit $x_0 \in \O \cap T$, on consid\`ere la famille $\widetilde{\lambda}_t=(g_t^{-1} x_0
)_{t \geqslant 0}$ qui est dans $\O \cap T$ pour tout $t$. Il est essentiel de remarquer que la famille  $(\widetilde{\lambda}_t)_{t>0}$ est un segment ouvert dont l'une des extr\'emit\'es est $x_0$ et l'autre est dans $s$. Soit $D$ un domaine fondamental convexe et
localement fini pour l'action de $\G$ sur $\O$ contenant le point
$x_0$. On note $\lambda_t$ la projection de $\widetilde{\lambda}_t$ sur $\overset{\circ}{S}$. La surface $\overset{\circ}{S}$ est hom\'eomorphe \`a la surface compacte $\Sigma_{g}$ de genre $g$ \`a laquelle on a retir\'e $p$ points $x_i$, $i=1...p$. On consid\`ere des disques ouverts $B_i$ de centre $x_i$ et suffisament petit pour \^etre disjoints deux \`a deux. On note $K'$ le compl\'ementaire de ces disques et $K$ le compact correspondant dans $\overset{\circ}{S}$. On a l'alternative:
\begin{enumerate}
\item Ou la demi-g\'eod\'esique $\lambda$ est ultimement incluse dans un bout $B$ de $S$.

\item Ou alors, il existe une suite de r\'eels $(t_n)_{n \in \N}$ tendant vers l'infini et tel que $\lambda_{t_n} \in K$.
\end{enumerate}

\begin{enumerate}
\item Commençons par traiter le second cas. Dans ce cas, il existe une famille $\g_{t_n} \in \G$ tel que
$\g_{t_n} g_{t_n}^{-1} x_0$ appartienne \`a une partie compacte de $D$. Quitte \`a extraire on peut supposer que la suite $(\g_{t_n} g_{t_n}^{-1} x_0)_{n \geqslant 0}$ converge vers un point $x_{\infty} \in D$. On a donc les convergences suivantes:

$$\underset{n \rightarrow \infty}{\lim} (\O, \g_{t_n} g_{t_n}^{-1} x)
= (\O , x_{\infty}) \textrm{ et } \underset{n \rightarrow
\infty}{\lim}  g_{t_n} \g_{t_n}^{-1}  (\O, \g_{t_n} g_{t_n}^{-1} x) = (T,x)$$

Le th\'eor\`eme \ref{ben2}  montre que la suite $(g_{t_n} \g_{t_n}^{-1})_{{t_n} \geqslant 0}$
est born\'ee, ce qui montre que $\O$ est un triangle, ce qui est
absurde.

\item Il  reste donc \`a traiter le premier cas. Il peut \^etre utile de remarquer que si on fait l'hypoth\`ese plus forte que l'action de $\G$ sur $\O$ est de covolume fini alors le corollaire \ref{corostrict} montre que la limite de $\widetilde{\lambda}_t$ en $+\infty$ qui est un point du segment $s$ devrait \^etre fix\'ee par un \'el\'ement parabolique, ce qui est absurde (lemme \ref{para}). Mais revenons, au cas g\'en\'eral, on note $\g$ un repr\'esentant de l'holonomie du bout $B$. La proposition \ref{fonda2}  montre qu'il faut distinguer quatre cas:
\begin{itemize}
\item L'\'el\'ement $\g$ est parabolique.

\item L'\'el\'ement $\g$ est quasi-hyperbolique ou hyperbolique avec $p^0_{\g} \notin \partial \O$ et $\Ax(\g) \subset \partial \O$.

\item L'\'el\'ement $\g$ est hyperbolique avec $p^0_{\g} \in \partial \O$ et l'un des axes secondaires de $\g$ est inclus dans $\partial \O$.
\end{itemize}

Dans les trois premiers cas, le lemme \ref{lemstrict}  montre qu'il existe un secteur $\F$ de $\g$ tel que la g\'eod\'esique $\widetilde{\lambda}$ est ultimement inclus dans $\F$. Le deuxi\`eme et troisi\`eme cas se traite de la m\^eme façon. On commence par le cas ou $\g$ est parabolique.

\begin{itemize}
    \item Si $\g$ est parabolique alors tout secteur de $\g$ ne poss\`ede qu'un point d'adh\'erence sur le bord de $\O$, \`a savoir le point fixe de $\g$. Par cons\'equent, $x_{\infty}$ est fix\'ee par $\g$, or, le lemme \ref{para}  montre que le point fixe d'un \'el\'ement parabolique ne peut-\^etre sur un segment non trivial du bord de $\O$. Ce cas ne peut donc pas se produire.

    \item Si $\g$ est quasi-hyperbolique ou hyperbolique avec $p^0_{\g} \notin \partial \O$ alors l'adh\'erence de tout secteur de $\g$ contient l'axe de $\g$. Par cons\'equent, le segment $s$ est l'axe de $\g$. C'est ce que l'on voulait montrer.

    \item Si $\g$ est hyperbolique avec $p^0_{\g} \in \partial \O$ alors d'apr\`es le lemme \ref{lemstrict}, la g\'eod\'esique $\widetilde{\lambda}$ est ultimement inclus dans la r\'eunion $\widetilde{\F}$ de $\F$ et de l'unique triangle inclus dans $\C$ d\'efinit par les points $p^+_{\g},p^-_{\g},p^0_{\g}$. Par cons\'equent, le segment $s$ est l'un des deux axes secondaires de $\g$.
\end{itemize}
\end{enumerate}
\end{proof}

\begin{coro}\label{strict}
Soient $\O$ un ouvert proprement convexe et un sous-groupe
discret sans torsion $\G$ qui pr\'eserve $\O$. On suppose que l'action de $\G$ sur
$\O$ est de covolume fini et que $\O$ n'est pas un triangle.
Alors, $\O$ est strictement convexe.
\end{coro}

\begin{proof}
Tout segment maximal et non trivial du bord de $\O$ est stabilis\'e par un \'el\'ement non parabolique $\g \in \G$ (th\'eor\`eme \ref{strictabord}) qui correspond \`a un lacet \'el\'ementaire trac\'e sur $\Quo$. Mais le th\'eor\`eme \ref{caract} montre que l'holonomie de tout lacet \'el\'ementaire de $\Quo$ est parabolique. L'ouvert $\O$ est donc strictement convexe.
\end{proof}

\subsection{Dualit\'e}

Soit $S$ une surface sans bord. Nous allons d\'efinir une op\'eration de dualit\'e sur l'espace $\beta(S)$ et nous allons voir que cette op\'eration pr\'eserve le sous-espace $\beta_f(S)$.

\begin{defi}
Soit $\O$ un ouvert proprement convexe de $\P$. On note $\O^*$
l'ouvert convexe de $\Pd$ : $\O^* = \{ \R f \in \Pd \, | \,
\forall \, \R v \in \overline{\O}, \, f(v)\neq 0 \}$. C'est un
ouvert proprement convexe. On l'appelle \emph{le dual de $\O$}.
\end{defi}

Pour faire passer cette notion de dualit\'e au niveau des surfaces nous allons avoir besoin d'une application $\Aut(\O)$-\'equivariante $\theta_{\O} : \O \rightarrow \O^*$. On rappelle ici une construction dû \`a Vinberg d'une telle application.

Soit $\Cc$ un des deux c\^ones ouverts proprement convexe de $\R^3$
dont l'image dans $\P$ est $\O$. On va de nouveau utiliser la
fonction caract\'eristique de $\Cc$. On note $\Cc^* = \{f \in \R^{3*}\,
| \, \forall \, v \in \overline{\Cc}-\{ 0 \}, \, f(v) > 0 \}$ le
c\^one dual de $\Cc$. Son image dans $\Pd$ est $\O^*$. Pour tout $v
\in \Cc$, on d\'efinit $v^* \in \Cc^*$ par la formule suivante:

$$ v^* = \frac{\int_{\Cc^*} fe^{-f(v)}df}{\int_{\Cc^*} e^{-f(v)}df}$$

Le point $v^*$ est le centre de gravit\'e du convexe $\{ f \in \Cc^* \, | \,
f(v) = 3 \}$. On d\'efinit donc $x^* = \theta_{\O}(x)$ comme la
droite engendr\'e par l'image $v^*$ de n'importe quel g\'en\'erateur $v \in \Cc$
de la droite $x \in \O$.

\begin{rem}
Il est vrai que $\O^{**} = \O$ via l'identification naturel entre
un espace vectoriel et son bidual. Par contre, l'application
$\theta_{\O^*} \circ \theta_{\O} \neq Id$ en g\'en\'eral. On pourra
consulter \cite{Beno4} pour un contre-exemple.
\end{rem}

On obtient ainsi la d\'efinition suivante:

\begin{defi}
Soient $S$ une surface sans bord et un point base $x_0 \in S$. Notons $\pi_1$ le groupe fondamental de $S$ bas\'e en $x_0$. On munit $S$ d'une structure projective proprement convexe via une d\'eveloppante $d:\widetilde{S}\rightarrow \P$ et une holonomie $\rho:\pi_1 \rightarrow \s$. La surface duale de $S$, not\'e $S^*$ est la surface projective proprement convexe d\'efinit par la d\'eveloppante $d^*=\theta_{\O} \circ d:\widetilde{S}\rightarrow \Pd$ et l'holonomie $\rho^*$ est la repr\'esentation duale de la repr\'esentation $\rho$. En particulier, la surface projective proprement convexe $S^*$ s'identifie au quotient $\O^*/_{\G^t}$. On notera $\theta_S$ l'hom\'eomorphisme induit par $\theta_{\O}$ entre $S$ et $S^*$.
\end{defi}

\begin{theo}\label{dual}
Soit $S$ une surface sans bord. On munit $S$ d'une structure
projective convexe. Alors, $S$ est de volume fini si et seulement
si $S^*$ est de volume fini.
\end{theo}

\begin{proof}
Le th\'eor\`eme \ref{caract}  montre qu'une surface projective proprement convexe est de volume fini si et seulement si l'holonomie de chaque bout est parabolique. Soit $\g$ un lacet qui fait le tour d'un bout de $S$. Le lacet $\theta_S(\g)$ fait le tour d'un bout de $S^*$. L'holonomie $\rho(\g)$ de $\g$ est parabolique si et seulement si l'holonomie $\rho^*(\g) = ^t\g^{-1}$ de $\theta_S(\g)$ est parabolique. Le fait d'\^etre de volume fini est donc stable par dualit\'e.
\end{proof}

Le fait suivant est tr\`es classique.

\begin{fait}\label{strictc1}
Soit $\O$ un ouvert proprement convexe de $\P$ alors le bord $\partial \O$ de $\O$ est $C^1$ si et seulement si $\O^*$ est strictement convexe.
\end{fait}

\begin{theo}\label{c1}
Soient $\O$ un ouvert proprement convexe de $\P$ et un sous-groupe discret sans torsion $\G$ non virtuellement ab\'elien de $\s$ qui pr\'eserve $\O$. Si l'action de $\G$ sur $\O$ est de covolume fini alors l'ouvert $\O$ est \`a bord $C^1$.
\end{theo}

\begin{proof}
L'action de $\G$ sur $\O$ est de covolume fini. La proposition \ref{dual}  montre que l'action dual du groupe $\G$ sur l'ouvert dual $\O^*$ de $\O$ est de covolume fini. Par cons\'equent, le corollaire \ref{strict} affirme que l'ouvert $\O^*$ est strictement convexe. Le fait \ref{strictc1}  permet de conclure que le bord
$\partial \O$ de $\O$ est $C^1$.
\end{proof}

\subsection{Caract\'erisation de la finitude du volume en termes d'ensemble limite}

\begin{defi}
Lorsque $\G$ est un sous-groupe discret irr\'eductible de $\s$ qui préserve un ouvert proprement convexe, on peut d\'efinir $\Lambda_{\G}$ \emph{l'ensemble limite} de $\G$, c'est le plus petit
ferm\'e invariant non vide de $\P$ (\cite{Beno5}).
\end{defi}

\begin{rem}
Le corollaire \ref{irre}  montre que l'ensemble limite $\Lambda_{\G}$ est bien d\'efini d\`es que $\G$ n'est pas virtuellement ab\'elien. De plus, l'ensemble limite est l'adh\'erence de l'ensemble
des points fixes attractifs des \'el\'ements hyperboliques de $\G$ (\cite{Beno5}). Il est clair que $\Lambda_{\G} \subset \partial \O$.
\end{rem}

\begin{theo}\label{enslimite}
Soient $\O$ un ouvert proprement convexe et $\G$ un sous-groupe discret non virtuellement ab\'elien de $\s$ qui pr\'eserve $\O$. Alors, l'action de $\G$ sur $\O$ est de covolume fini si et seulement si $\G$ est de type fini et $\Lambda_{\G} =\partial \O$.
\end{theo}

\begin{proof}
Remarquons que $\O$ n'est pas un triangle puisque $\G$ n'est pas virtuellement ab\'elien. Commençons par supposer que $\Lambda_{\G} \neq \partial \O$ et
montrons que $\mu(\Quo) = \infty$. L'int\'erieur $K$ de l'enveloppe convexe de
$\Lambda_{\G}$  est invariante par $\G$ et $\O - K$ est non vide
puisque $\O$ est strictement convexe (corollaire \ref{strict}). Par
cons\'equent, les composantes connexes de $\O - K$ sont permut\'ees
par $\G$ et leur stabilisateur est soit trivial, soit engendr\'e par
un \'el\'ement hyperbolique. Dans tous les cas, la proposition \ref{mubord} montre que $\mu(\Quo)=\infty$.

Supposons \`a pr\'esent que $\mu(\Quo)=\infty$ et $\G$ est de type fini. On peut supposer que $\G$ est sans torsion puisque $\G$ est de type fini. Le corollaire \ref{caract}  montre que l'holonomie de l'un des bouts de la
surface sans bord et de type fini $\Quo$ est hyperbolique ou
quasi-hyperbolique. Par cons\'equent, si on consid\`ere $c$ un lacet autour de l'un de ces bouts, la r\'eunion $E$ des relev\'es de la composante \'el\'ementaire associ\'ee \`a $c$ est un ferm\'e de $\O$ pr\'eserv\'e par $\G$. La partie $\overline{(\partial \O - \overline{E^{\P}})^{\P}}$ est un ferm\'e $\G$-invariant. Il est diff\'erent de $\partial \O$ car $\Hol(c)$ n'est pas parabolique. Donc $\Lambda_{\G} \neq \partial \O$.
\end{proof}

\begin{rem}
L'hypoth\`ese $\G$ type fini est essentielle comme le montre la construction classique suivante. Soient $\O$ un ouvert proprement convexe et $\G$ un sous-groupe discret non virtuellement ab\'elien et sans torsion de $\s$ qui pr\'eserve $\O$ tel que l'action de $\G$ sur $\O$ est de covolume fini. La proposition \ref{typefini}  montre que $\G$ est ou un groupe libre non ab\'elien de type fini ou isomorphe au groupe fondamental d'une surface compacte de genre sup\'erieur ou \'egale \`a 2. Par cons\'equent, le groupe d\'eriv\'e $[\G,\G]$ de $\G$  est un sous-groupe d'indice infini de $\G$. Son action sur $\O$ n'est donc pas de covolume fini. Mais, l'ensemble limite $\Lambda_{[\G,\G]}$ est $\G$-invariant car $[\G,\G]$ est distingu\'e dans $\G$, par cons\'equent  $\Lambda_{[\G,\G]} =\Lambda_{\G}$. Ceci ne contredit pas notre th\'eor\`eme puisqu'il est bien connu que $[\G,\G]$ n'est pas de type fini.
\end{rem}

\begin{prop}\label{enslimiteabord1}
Soit $\G$ un sous-groupe discret non virtuellement ab\'elien, sans torsion et de type fini de $\s$ qui
pr\'eserve un ouvert proprement convexe de $\O$. On consid\`ere l'enveloppe convexe $\C$ dans $\overline{\O} - \Lambda_{\G}$ de $\Lambda_{\G}$ l'ensemble limite de $\G$. La surface $\Quc$ est de volume fini.
\end{prop}

\begin{proof}
Par d\'efinition, l'int\'erieur $\U$ de $\C$ est le plus petit ouvert convexe de $\P$ pr\'eserv\'e par $\G$. La partie $\C$ est le convexe obtenu en ajoutant \`a $\U$ les axes des \'el\'ements hyperbolique et quasi-hyperbolique de $\G$ qui repr\'esente un lacet qui fait le tour d'un bout de $\Quo$ d'holonomie hyperbolique ou quasi-hyperbolique (théorème \ref{strictabord} et proposition \ref{fonda1}). Par cons\'equent, tous les bouts de la surface $\Quc$ ont une holonomie parabolique et l'holonomie des composantes connexes du bord de $\Quc$ est hyperbolique ou quasi-hyperbolique. La proposition \ref{fonda2} montre que la surface $\Quc$ est de volume fini.
\end{proof}

%

\subsection{Unicit\'e de l'ouvert $\O$ lorsque le volume est fini}

\begin{prop}\label{unique}
Soit $\G$ un sous-groupe discret de $\s$, on suppose qu'il existe
un ouvert proprement convexe $\O$ tel que $\mu(\Quo) < \infty$ et
que $\O$ n'est pas un triangle. Alors, $\O$ est l'unique ouvert convexe de
$\P$ sur lequel $\G$ agit proprement.
\end{prop}

\begin{proof}
Soit $F= \overline{\{x \in \P \, | \, \exists \g \in \G \textrm{
tel que } \g x = x \}}$, commençons par montrer que $F = \P - \O$.
On sait d'apr\`es \cite{Beno5} que pour tout couple de points $(x^+,x^-) \in \Lambda_{\G}
\times \Lambda_{\G}$, il existe une suite d'\'el\'ements hyperboliques $\g_n \in \G$ tel que
$\underset{n \rightarrow \infty}{\lim} p^+_{\g_n} = x^+$ et
$\underset{n \rightarrow \infty}{\lim} p^-_{\g_n} = x^-$. La proposition \ref{hyp} montre que le point $p^0_{\g_n}$ est l'intersection des tangentes \`a $\partial \O$ en $p^+_{\g_n}$ et $p^-_{\g_n}$.
Le th\'eor\`eme \ref{enslimite} montre que $\Lambda_{\G} = \partial \O$. Par cons\'equent, comme le bord $\partial \O$ de $\O$ est $C^1$ (proposition \ref{c1}), l'ensemble $\{ p^0_{\g} \, | \, \g \in \G \textrm{ hyperbolique}
\}$ est dense dans $\P - \O$. Par cons\'equent, si $\O'$ est un
ouvert sur lequel $\G$ agit proprement alors $\O' \subset \O$.
Il  reste \`a remarquer que l'enveloppe convexe de toute orbite d'un point de $\O$ est $\O$ lui-m\^eme. L'adh\'erence de toute orbite d'un point quelconque de $\O$ contient l'ensemble limite $\Lambda_{\G}$, or $\Lambda_{\G} = \partial \O$. Ainsi, le seul ouvert convexe $\G$-invariant de $\P$ est $\O$.
\end{proof}

Cette proposition  montre que l'holonomie d'une structure projective proprement convexe de volume fini sur une surface sans bord caract\'erise cette structure. On obtient facilement le corollaire suivant:

\begin{coro}\label{pourlechap3}
Soient $S$ une surface de type fini \underline{sans bord}  de caractéristique d'Euler strictement négative, et $\rho$ une repr\'esentation du groupe fondamental de $S$, il existe au plus une structure projective proprement convexe de volume fini sur la surface $S$ dont l'holonomie est $\rho$.
\end{coro}

\begin{rem}
Ce th\'eor\`eme est faux pour les surfaces \`a bord. En effet, consid\'erons une surface \`a bord $S$ de type fini et donnons nous une structure projective proprement convexe de volume fini sur $S$. On note $\C$ la partie proprement convexe de $\P$ donn\'e par la d\'eveloppante de cette structure et $\O$ l'int\'erieur de $\C$. Consid\'erons un lacet $c$ homotope \`a une composante connexe $L$ du bord, on a trois cas \`a distinguer (proposition \ref{fonda1}):

\begin{itemize}
\item L'\'el\'ement $\Hol(c)$ est quasi-hyperbolique,

\item L'\'el\'ement $\Hol(c)$ est hyperbolique et $\Ax(\Hol(c)) \subset \partial \C$,

\item L'\'el\'ement $\Hol(c)$ est hyperbolique et $\Ax(\Hol(c)) \subset \O$.
\end{itemize}

Tout relev\'e du bord $L$ est un segment inclus dans $\partial \C$ pr\'eserv\'e par un conjugu\'e de $\Hol(c)$. Par cons\'equent, dans les deux premiers cas, le bord $L$ de $S$ est la projection de l'axe principal de $\Hol(c)$ sur la surface $S$. Et, dans le troisi\`eme cas, le bord $L$ est la projection de l'un des deux axes secondaires de $\Hol(c)$.

Par cons\'equent, \'etant donn\'e une structure projective proprement convexes de volume fini sur une surface \`a bord $S$. Notons $n_h$ le nombre de composantes connexes du bord de $S$ dont l'holonomie est hyperbolique, le paragraphe pr\'ec\'edent montre que l'on peut construire au moins $3^{n_h}$ structures projectives proprement convexes de volume fini sur la surface $S$ qui ont la m\^eme holonomie que $S$. Et, le th\'eor\`eme \ref{fonda2} montre qu'il n'y en a pas d'autres.

Pour r\'esoudre ce probl\`eme dans le but d'\'etudier l'espace des modules des structures projectives proprement convexe de volume fini sur les surfaces \`a bord, on introduit la notion de surface \`a bord g\'eod\'esique principal.
\end{rem}

\begin{defi}
Une structure projective proprement convexe sur une surface \`a bord $S$ est dite \`a \emph{bord g\'eod\'esique principal} lorsqu'elle est \`a bord g\'eod\'esique et toute composante connexe $L$ du bord de $S$ qui a une holonomie hyperbolique est la projection de l'axe principal de cette holonomie.
\end{defi}

La proposition suivante est \`a pr\'esent \'evidente.

\begin{prop}
Soient $S$ une surface de type fini de caractéristique d'Euler strictement négative, et $\rho$ une repr\'esentation du groupe fondamental de $S$, il existe au plus une structure projective proprement convexe de volume fini \`a bord g\'eod\'esique principal sur la surface $S$ dont l'holonomie est $\rho$.
\end{prop}

\subsection{Discr\'etude du groupe $\Aut(\O)$}

\begin{theo}\label{alter}
Soient $\G$ un sous-groupe discret de $\s$ et $\O$ un ouvert proprement convexe tel que $\mu(\Quo) < \infty$. On suppose que $\O$ n'est pas un triangle alors on a l'alternative
\underline{exclusive} suivante:
\begin{itemize}
\item L'adh\'erence de Zariski $\overline{\G^Z}$ de $\G$ est
conjugu\'ee au groupe $\mathrm{SO}_{2,1}(\R)$ et $\O$ est un
ellipsoïde et $\Aut(\O) = \overline{\G^Z}$.

\item Le groupe $\G$ est Zariski dense dans $\s$ et $\Aut(\O)$ est un
sous-groupe discret de $\s$. En particulier, $\G$ est d'indice
fini dans $\Aut(\O)$.
\end{itemize}
\end{theo}

\begin{proof}
Supposons que l'adh\'erence de Zariski de $\G$ est
conjugu\'ee au groupe $\mathrm{SO}_{2,1}(\R)$ par cons\'equent, le groupe $\G$ pr\'eserve une ellipse. La proposition \ref{unique} montre que $\O$ est cet ellipsoïde.

Si l'adh\'erence de Zariski de $\G$ n'est pas conjugu\'ee au groupe $\mathrm{SO}_{2,1}(\R)$ alors le corollaire \ref{Zariski} montre que $\G$ est Zariski dense. Le lemme \ref{disoudens} montre que $\Aut(\O)$ est un
sous-groupe discret ou dense de $\s$. Or, il pr\'eserve $\O$ il ne
peut donc pas \^etre dense.
\end{proof}

\begin{lemm}\label{disoudens}
Soit $H$ un sous-groupe Zariski dense de $\s$, on a l'alternative
suivante:
\begin{itemize}
\item Le groupe $H$ est discret, ou bien

\item Le groupe $H$ est dense.
\end{itemize}
\end{lemm}

\begin{proof}
Ce lemme est vrai plus g\'en\'eralement pour les sous-groupes Zariski
dense d'un groupe alg\'ebrique quasi-simple. La composante neutre $G$ de l'adh\'erence (pour la topologie usuelle) de $H$ est un sous-groupe ferm\'e,
connexe de $\s$ qui est normalis\'e par $H$, et donc par $\s$
puisque $H$ est Zariski dense. Par cons\'equent, comme $\s$ est
simple. On obtient que:
\begin{itemize}
\item $G= \{ 1 \}$, ou bien

\item $G = \s$.
\end{itemize}
\end{proof}

Ce th\'eor\`eme entra\^ine le corollaire suivant.

\begin{coro}
Soit $\O$ un ouvert proprement convexe de $\P$, on suppose qu'il
existe $\G_1$ et $\G_2$ tel que $\mu(\O/_{\G_1}) < \infty$ et
$\mu(\O/_{\G_2}) < \infty$. On suppose de plus que $\O/_{\G_1}$
est compact et que $\O/_{\G_2}$ n'est pas compact. Alors, $\O$ est
un ellipsoïde.
\end{coro}

\begin{proof}
Tout d'abord, l'ouvert $\O$ n'est pas un triangle puisque
$\O/_{\G_2}$ n'est pas compact. Ensuite, si $\O$ n'est pas un
ellipsoïde alors le th\'eor\`eme \ref{alter}  montre que $\G_1$ et
$\G_2$ sont des sous-groupes d'indice fini de $\Aut(\O)$ par
cons\'equent $\G_1 \cap \G_2$ est un sous-groupe d'indice fini de
$\G_1$ et $\G_2$. Il vient que $\O/_{\G_1 \cap \G_2}$ devrait
\^etre compact et non compact, ce qui est absurde.
\end{proof}

\backmatter


\begin{thebibliography}{CVV04}

\bibitem[BBI01]{BBI}
Dimitri Burago, Yuri Burago, and Sergei Ivanov.
\newblock A course in metric geometry.
\newblock {\em Graduate Studies in Mathematics}, 33, 2001.

\bibitem[Ben]{Beno3}
Yves Benoist.
\newblock Sous-groupes discrets des groupes de Lie.
\newblock {\em European Summer School in Group Theory Luminy 7-18 July 1997}.

\bibitem[Ben60]{Ben}
Jean-Paul Benzécri.
\newblock Sur les variétés localement affines et localement projectives.
\newblock {\em Bulletin de la Société Mathématique de France}, 88:p.229--332,
  1960.

\bibitem[Ben00]{Beno5}
Yves Benoist.
\newblock Automorphismes des cônes convexes.
\newblock {\em Invent. Math.}, 141:p.149--193, 2000.

\bibitem[Ben03]{Beno4}
Yves Benoist.
\newblock Convexes hyperboliques et fonctions quasisymétriques.
\newblock {\em Publ. Math. IHES}, 97:p.181--237, 2003.

\bibitem[Ben06]{Beno1}
Yves Benoist.
\newblock Convexes hyperpoliques et quasiisométries.
\newblock {\em Geometriae Dedicata}, 122:p. 109--134, 2006.

\bibitem[Bor63]{Bor}
Armand Borel.
\newblock Compact Clifford-Klein forms of symmetric spaces.
\newblock {\em Topology}, 2:p.111--122, 1963.

\bibitem[Cho94]{Choi}
Suhyoung Choi.
\newblock Convex decompositions of real projective surfaces. II: Admissible
  decompositions.
\newblock {\em J. Differential Geom.}, 40:p. 239--283, 1994.

\bibitem[CVV04]{CVV1}
Bruno Colbois, Constantin Vernicos, and Patrick Verovic.
\newblock L'aire des triangles idéaux en géométrie de Hilbert.
\newblock {\em L'enseignement mathématique}, 50:p. 203--237, 2004.

\bibitem[CVVre]{CVV2}
Bruno Colbois, Constantin Vernicos, and Patrick Verovic.
\newblock Area of ideal triangles and gromov hyperbolicity in hilbert geometry.
\newblock {\em Illinois Journal of Math.}, A paraître.

\bibitem[Gol90]{Gold1}
William Goldman.
\newblock Convex real projective structures on compact surfaces.
\newblock {\em J. Differential Geom.}, 31:791--845, 1990.

\bibitem[JM84]{JoMil}
D.~Johnson and John Millson.
\newblock Deformation spaces associated to compact hyperbolic manifolds.
\newblock {\em Discrete groups in Geometry and Analysis Progr. Math}, 67:p.
  48--106, 1984.

\bibitem[Kap07]{Kapo}
Misha Kapovich.
\newblock Convex projective structures on Gromov-Thurston manifolds.
\newblock {\em Geometry and Topology}, Vol. 11:p. 1777--1830, 2007.

\bibitem[KV67]{KaV}
Victor Kac and Ernest~Borisovich Vinberg.
\newblock Quasi-homogeneous cones.
\newblock {\em Math. Notes}, 1:p.231--235, 1967.

\bibitem[Lee]{JL}
Jaejeong Lee.
\newblock Convex fundamental domains for properly convex real projective
  structures.
\newblock {\em preprint}.

\bibitem[Mar09]{moi3}
Ludovic Marquis.
\newblock Espace de Modules Marqués des Surfaces Projectives Convexes de Volume Fini.
\newblock {\em preprint arxiv.org/abs/0910.5839}.

\bibitem[Ric63]{Rich}
Ian Richards.
\newblock On the classification of noncompact surfaces.
\newblock {\em Trans. Amer. Math. Soc.}, 106:p. 259--269, 1963.

\bibitem[Vey70]{Vey}
Jacques Vey.
\newblock Sur les automorphismes affines des ouverts convexes saillants.
\newblock {\em Anna Scuola Normale Superiore di Pisa}, 24:p. 641--665, 1970.

\bibitem[Vin63]{Vin}
{È}rnest Borisovich Vinberg.
\newblock The theory of convex homogeneous cones.
\newblock {\em Trudy Moskov. Mat. Ob\v s\v c.}, 12:p. 303--358, 1963.

\bibitem[Vin65]{Vin2}
{È}rnest Borisovich Vinberg.
\newblock The structure group of automorphisms of a homogeneous convex cone.
\newblock {\em Trudy Moskov. Mat. Ob\v s\v c.}, 13:p. 56--83, 1965.
\end{thebibliography}
\end{document}